\documentclass[a4paper,11pt,leqno]{amsart}
\usepackage{a4wide,color,array}
\usepackage{mathtools}
\usepackage{hyperref}
\usepackage{cite}
\hypersetup{linktocpage,colorlinks, citecolor={magenta}}
\usepackage{amsmath,amssymb,amsthm}
\usepackage[english]{babel}
\usepackage[utf8]{inputenc}
\usepackage[cyr]{aeguill}
\usepackage{esint}
\usepackage{xargs}
\usepackage{enumitem}
\setlist{nosep} % or \setlist{noitemsep} to leave space around whole list
\usepackage{lmodern}
\usepackage{bm}

\newcommand{\N}{\mathbb{N}}
\newcommand{\Z}{\mathbb{Z}}

\newcommand{\R}{\mathbb{R}}
\newtheorem{theo}{Theorem}
\newtheorem{prop}{Proposition}[section]
\newtheorem{lem}[prop]{Lemma}
\newtheorem{coro}[prop]{Corollary}
\newtheorem{remark}[prop]{Remark}
\newtheorem{claim}[prop]{Claim}
\newtheorem{defi}[prop]{Definition}

\theoremstyle{plain}

\numberwithin{equation}{section}

\newcommand{\then}{\Longrightarrow} % Implication longue
\def\t0{\rightarrow 0} % Vers zéro
\def\ti{\rightarrow \infini} % Vers l'infini

\newcommand{\f}{\frac}
\newcommand{\infini}{\infty}
\newcommand{\ep}{\varepsilon}
\newcommand{\tens}{\otimes}
\newcommand{\hal}{\frac{1}{2}}

\def\div{\mathrm{div} \, } % Divergence
\def\1{\mathbf{1}} % Fonction caractéristique

\def \mc{\mathcal}
\def \ep{\epsilon}
\renewcommand{\epsilon}{\varepsilon}

\def\Zd{\Z^d} % Réseau  
\def\Supp{\Sigma} % Support de la mesure d'équilibre
\def\meseq{\mu_{V}} % Mesure d'équilibre
\def \ZNbeta{Z_{N,\beta}} % Fonction de partition

\def \um{\underline{m}} % Bornes sur la densité
\def \om{\overline{m}} % Bornes sur la densité
\def\({\left(}
\def\){\right)}

\def\yg{|y|^\gamma}
\def \W{\mathbb{W}} % Énergie d'une configuration de points
\def\tW{\widetilde{\mc{W}}} % Énergie d'un processus électrique
\def\ttW{\widetilde{\W}} % Énergie d'un processus ponctuel
\def \bttW{\overline{\W}} % Énergie d'un processus ponctuel marqué
\def \WN{w_N} % Énergie à N points

\def\config{\mathrm{Config}} % Espace des configurations de points
\def\Elec{\mathrm{Elec}}
\def\Lip{\mathrm{Lip}_1} % Espace 1-Lipschitz
\def\Loc{\mathrm{Loc}}% Fonctions locales

\def\dconfig{d_{\config}}

\def \probas{\mathcal{P}}
\def \Pelec{P^{\mathrm{elec}}} % Processus électrique

\def\P{\mathbb{P}} % "Vraies" mesures
\def \Pst{P} % Processus ponctuel standard sans marque
\def \bPst{\bar{P}} % Processus ponctuel standard avec marque
\def \Pgot{\mathfrak{P}}
\def \bPgot{\overline{\Pgot}}
\def \PNbeta{\P_{N, \beta}} % Mesure de Gibbs à \beta
\def \PgNbeta{\mathbf{P}_{N,\beta}} % Processus Gibbsien à N points et à \beta
\def \PgN2{\mathbf{P}_{N,2}} % Processus Gibbsien à N points et à \beta = 2..
\def \QN{\mathbb{Q}_{N,\beta}} % Mesure de référence (Lebesgue + zeta)
\def \bQpN{\bar{\mathfrak{Q}}_{N,\beta}} % Processus ponctuel de référence avec marque
\def \sineb{\mathrm{Sine}_{\beta}} % Sine_{\beta}
\def \bsineb{\overline{\sineb}} % Sine_{\beta} avec marques

\def \HN{\mathcal{H}_N}

\def\Esp{\mathbf{E}} % Espérance

\def \Ent{\mathrm{Ent}}   % Entropie relative classique
\def \ERS{\mathsf{ent}} % Entropie relative spécifique (sans marques)
\def \bERS{\overline{\mathsf{ent}}} % Entropie relative spécifique (sans marques)

\def \Leb{\mathbf{Leb}}
\def \Poisson{\mathbf{\Pi}}
\def \B{\mathbf{B}} % Bernoulli

\def \bR{\bar{R}} 
\def \bS{\bar{S}}
\def \conf{\mathrm{Conf}} % Configuration sous-jacente
\def \cc{c_1}
\def \dist{d}

\def \dist{\mathrm{dist}}
\def\cc{c_1}
\def\cC{c_2}

\renewcommand{\Supp}{\Sigma}
\def\muv{\meseq}
\def\I{\mathcal I}
\def\mr{\mathbb{R}}
\def\nab{\nabla}
\def\indic{\mathbf{1}}

\def\Xint#1{\mathchoice
   {\XXint\displaystyle\textstyle{#1}}%
   {\XXint\textstyle\scriptstyle{#1}}%
   {\XXint\scriptstyle\scriptscriptstyle{#1}}%
   {\XXint\scriptscriptstyle\scriptscriptstyle{#1}}%
   \!\int}
\def\XXint#1#2#3{{\setbox0=\hbox{$#1{#2#3}{\int}$}
     \vcenter{\hbox{$#2#3$}}\kern-.5\wd0}}
\def\dashint{\Xint-}

\def \Op{\mathcal{O}_{R, \eta}}

\newcommandx \Sp{\mathcal{S}^{M,\epsilon}_{R, \eta}}
\newcommandx \So[2][1=e,2=M]{\mathcal{S}^{#2,#1,\epsilon}_{R, \eta,-}}
\renewcommandx \H{F^{M,\epsilon}_{R, \eta}}

\def \carr{\square} % Carré
\def \barcarr{\overline{\carr}}
\def \Int{\mathrm{Int}}
\def \Ext{\mathrm{Ext}}
\def \bM{\dot{\mathfrak{M}}}
\def \HM{\widehat{\mathfrak{M}}}

\def \btW{\overline{\mc{W}}}
\def \Pelec{\Pst^{\mathrm{elec}}}
\def \bPelec{\bPst^{\mathrm{elec}}}

\def \Old{\mathrm{Old}}
\def \New{\mathrm{New}}

\def  \comeg{c_{\omega,\Sigma}}

\def \cmuvum{r_{\mu_V, \um}}

\def\Aabs{A^{\rm{abs}}}
\def \Amod{A^{\rm{mod}}}
\def \Aext{A^{\rm{ext}}}

\def\Cabs{\mathcal{C}^{\rm{abs}}}
\def\Cmod{\mathcal{C}^{\rm{mod}}}
\def \Ctot{\mathcal{C}^{\rm{tot}}}
\def \Creg{\mathcal{C}^{\rm{reg}}}
\def \Cext{\mathcal{C}^{\rm{ext}}} 
\def \Emod{E^{\rm{mod}}}
\def \Eext{E^{\rm{ext}}}
\def \Etot{E^{\rm{tot}}}
\def \Escr{E^{\rm{scr}}}
\def \Ereg{E^{\rm{reg}}}
\def\Egen{E^{\rm{gen}}}
\def \Egene{E_\eta^{\rm{gen}}}
\def \Cgen{\mathcal{C}^{\rm{gen}}}
\def \Cscr{\mathcal{C}^{\rm{scr}}}
\def \Phiscr{\Phi^{\rm{scr}}}

\def \Phireg{\Phi^{\rm{reg}}}
\def \Phimod{\Phi^{\rm{mod}}}
\def \Phigen{\Phi^{\rm{gen}}}
\def\g{\mathsf{g}}
\def \V{V}
\def \I{\mathcal{I}}
\def \LogU{\textsf{Log1}}
\def \LogD{\textsf{Log2}}
\def \Riesz{\textsf{Riesz}}
\def \emp{\mu^{\mathrm{emp}}_N}
\def \XN{\vec{X}_N}
\def \Emp{\mathrm{Emp}}
\def \bEmp{\overline{\Emp}}
\def \C{\mathcal{C}}
\def \bP{\bar{P}}
\def \KNbeta{K_{N, \beta}}
\def \cV{\mathsf{c}_V}
\def \hmuv{\mathsf{h}^{\mu_V}}
\makeatletter
\def\namedlabel#1#2{\begingroup
    #2%
    \def\@currentlabel{#2}%
    \phantomsection\label{#1}\endgroup
}
\makeatother
\def \k{\mathsf{k}}
\def \d{\mathsf{d}}
\def \s{\mathsf{s}}

\def \Hmu{H^{\mu}}
\def \f{\mathsf{f}}
\def \p{\mathsf{p}}
\def \hpN{H'_{N}}
\def \hpNe{H'_{N, \eta}}
\def \cds{\mathsf{c}_{\d,\s}} %% Constante c,d,s.
\def \c{\cds}
\def\Rd{\R^\d} % Espace physique
\def \Lploc{L^\p_{\mathrm{loc}}}
\def \drd{\delta_{\Rd}}
\def \ent{\mathrm{ent}}
\def \bPstx{\bar{P}^{x}}
\def \dis{\mathrm{Dis}}

\def \Numb{\mathrm{Num}}
\def \QNbeta{\mathbb{Q}_{N,\beta}}
\def \fbeta{\mathcal{F}^{1}_{\beta}} % Rate function sans marque
\def \fbarbeta{\overline{\mathcal{F}}^{\meseq}_{\beta}} % Rate function avec marque
\def \Comp{\mathrm{Comp}}
\def \Screen{\mathrm{Screen}}
\def \tcarr{\check{\carr}} 
\def \bQ{\bar{Q}}
\def \EK{E^{(K)}}
\def \x{\mathsf{x}}
\def \mNR{m_{N,R}}
\def \DiscrAv{\mathrm{DiscrAv}}
\def \ContAv{\mathrm{ContAv}}
\def \Sigmatil{\Sigma'_{\mathrm{til}}}
\def \calA{\mathcal{A}}
\def \Ntil{N_{\mathrm{til}}}

\def \Sigmaum{\Sigma_{\um}}
\def \Gammaj{\Gamma^{(j)}}
\def \Gammapj{\Gamma^{' (j)}}
\def \ll{\ell}
\def \llj{\ell_{j}}
\def \alphaj{\alpha_j}
\def \v{v}
\def \Nn{\mathcal{N}}
\def \Nnint{\mathcal{N}^{\mathrm{int}}}
\def \Dint{\mathcal{D}^{\mathrm{int}}}
\def \bPum{\bar{P}_{\um}}
\def \Intep{\Int_{\epsilon}}
\def \Extep{\Ext_{\epsilon}}
\def \Uj{\mathcal{U}_j}

\begin{document}
\title{Large Deviation Principle  for Empirical Fields of Log and Riesz Gases}
\author{Thomas Lebl\'e and Sylvia Serfaty}
\begin{abstract}
We study a system of $N$ particles with logarithmic, Coulomb or Riesz pairwise interactions, confined by an external potential.  We examine a microscopic quantity, the tagged empirical field, for which we prove a large deviation principle at speed $N$. The rate function is the sum of an entropy term, the specific relative entropy, and an energy term, the renormalized energy introduced in previous works, coupled by the temperature.

We deduce a variational property of the sine-beta processes which arise in random matrix theory. We also give a next-to-leading order expansion of the free energy of the system, proving the existence of the thermodynamic limit.
\end{abstract}
\date{\today}
\maketitle


\setcounter{tocdepth}{1}
\tableofcontents

\section{Introduction}

\subsection{General setting} \label{sec-gensetting}
We consider a system of $N$ points (or particles) in the Euclidean space $\R^\d$ ($\d \geq 1$) with logarithmic, Coulomb or Riesz pairwise interactions, confined by an external potential $\V$ whose amplitude is chosen to be proportional to $N$. For any $N$-tuple of positions $\XN = (x_1, \dots, x_N)$ in $(\R^{\d})^N$ we associate the energy given by
\begin{equation} \label{HN}
\HN(\XN) := \sum_{1 \leq i \neq j \leq N} \g(x_i-x_j) +  \sum_{i=1}^N N \V(x_i).
\end{equation}
The interaction kernel $\g$ is given by either  
\begin{align*} 
%\label{wlog} 
& (\LogU \ \text{case}) \quad \g(x) = -\log |x| , \quad \text{in dimension } \d=1, \\
%\label{wlog2d}
 & (\LogD \ \text{case}) \quad \g(x) = - \log |x| , \quad  \text{in dimension } \d=2, \\
%\label{kernel}
 & (\Riesz \ \text{case}) \quad \g(x) = |x|^{-\s}, \quad   \text{with }\max(0, \d-2)\leq \s<\d, \text{ in dimension $\d \geq 1$}.
\end{align*} 
\LogU \ (resp. \LogD) corresponds to a one-dimensional (resp. two-dimensional) logarithmic interaction, we will call \LogU, \LogD \ the \textit{logarithmic cases}. \LogD \ is also the Coulomb interaction in dimension $2$. For $\d \geq 3$, taking $\s = \d-2$ in the \Riesz\ cases corresponds to a Coulomb interaction in higher dimension, while $\max(\d-2, 0) < \s < \d$ corresponds to more general Riesz interactions. Whenever the parameter $\s$ appears, it will be with the convention that $\s = 0$ in the logarithmic cases. The potential $V$ is a confining potential, growing fast enough at infinity, on which we shall make assumptions later.

For any $\beta > 0$, we consider the canonical Gibbs measure at inverse temperature $\beta$, given by the following density
\begin{equation}\label{gibbs}
d\PNbeta(\XN) = \frac{1}{\ZNbeta} \exp \left( -\frac{\beta}{2} N^{-\frac{\s}{\d}} \HN(\XN)  \right) d\XN,
\end{equation}
where $d\XN = \prod_{i=1}^N dx_i$ is the Lebesgue measure on $(\Rd)^N$, and $\ZNbeta$ is the normalizing constant, called the \textit{partition function}.
In \eqref{gibbs} the inverse temperature $\beta$ appears with a factor $\frac{1}{2}$ in order to match existing convention in random matrix theory. In the \Riesz \  cases, the temperature scaling $\beta N^{-\s/\d}$ is chosen to obtain non-trivial results.

\subsection{The macroscopic behavior: empirical measure}
It is well-known since \cite{choquet} (see e.g. \cite{safftotik} for the logarithmic cases, or \cite[Chap.2]{serfatyZur} for a simple proof in the general case) that under suitable assumptions on $V$, we have
\begin{equation*}
\min \HN = N^2 \I_V (\meseq) +o(N^2), 
\end{equation*}
where $\I_V$ is the mean-field energy functional defined on the set of Radon measures by
\begin{equation}\label{MFener}
\I_V (\mu) : = \iint_{\Rd\times \Rd} \g(x-y) \, d\mu(x)\, d\mu(y) + \int_{\Rd} V(x)\, d\mu(x).
\end{equation}
There is a unique minimizer of $\I_V$ on the space $\probas(\R^{\d})$ of probability measures on $\R^{\d}$, it is called the \textit{equilibrium measure} and we denote it by $\meseq$. We will always assume that $\meseq$ is a measure with a H\"older continuous density on its support, we abuse notation by denoting its density $\meseq(x)$ and we also assume that its support $\Supp$ is a compact set with a nice boundary. We allow for several connected components of $\Supp$ (also called the \textit{multi-cut regime} in the case \LogU). The precise assumptions are listed in Section \ref{sec:assumptions}. 

A convenient macroscopic observable is given by the empirical measure of the particles: if $\XN$ is in $(\R^{\d})^N$ we form
\begin{equation} \label{def:empmeasure}
\emp(\XN) := \frac{1}{N} \sum_{i=1}^N \delta_{x_i},
\end{equation}
which is a probability measure on $\R^{\d}$. The minimisation of $\I_V$ determines the macroscopic (or global) behavior of the system in the following sense:
\begin{itemize}
\item Minimisers of $\HN$ are such that $\emp(\XN)$ converges to $\meseq$ as $N \to \infty$.
\item In fact $\emp(\XN)$ converges weakly to $\meseq$ as $N \ti$ almost surely under the canonical Gibbs measure $\PNbeta$.
\end{itemize}
In other words, not only the minimisers of the energy, but almost every (under the Gibbs measure) sequence of particles is such that the empirical measure converges to the equilibrium measure. Since $\meseq$ does not depend on the temperature, \textit{the asymptotic macroscopic behavior of the system is independent of $\beta$.}

\subsection{The microscopic behavior: empirical fields}
In contrast, several observations (e.g. by numerical simulation, see the figure below) suggest that the behavior of the system at microscopic scale\footnote{Since the $N$ particles are typically confined in a set of order $O(1)$, the microscopic, inter-particle scale is $O(N^{-1/{\d}})$.} depends heavily on $\beta$. 
\begin{figure}[h!] 
\begin{minipage}[c]{.46\linewidth}
\begin{center}
\includegraphics[scale=0.12]{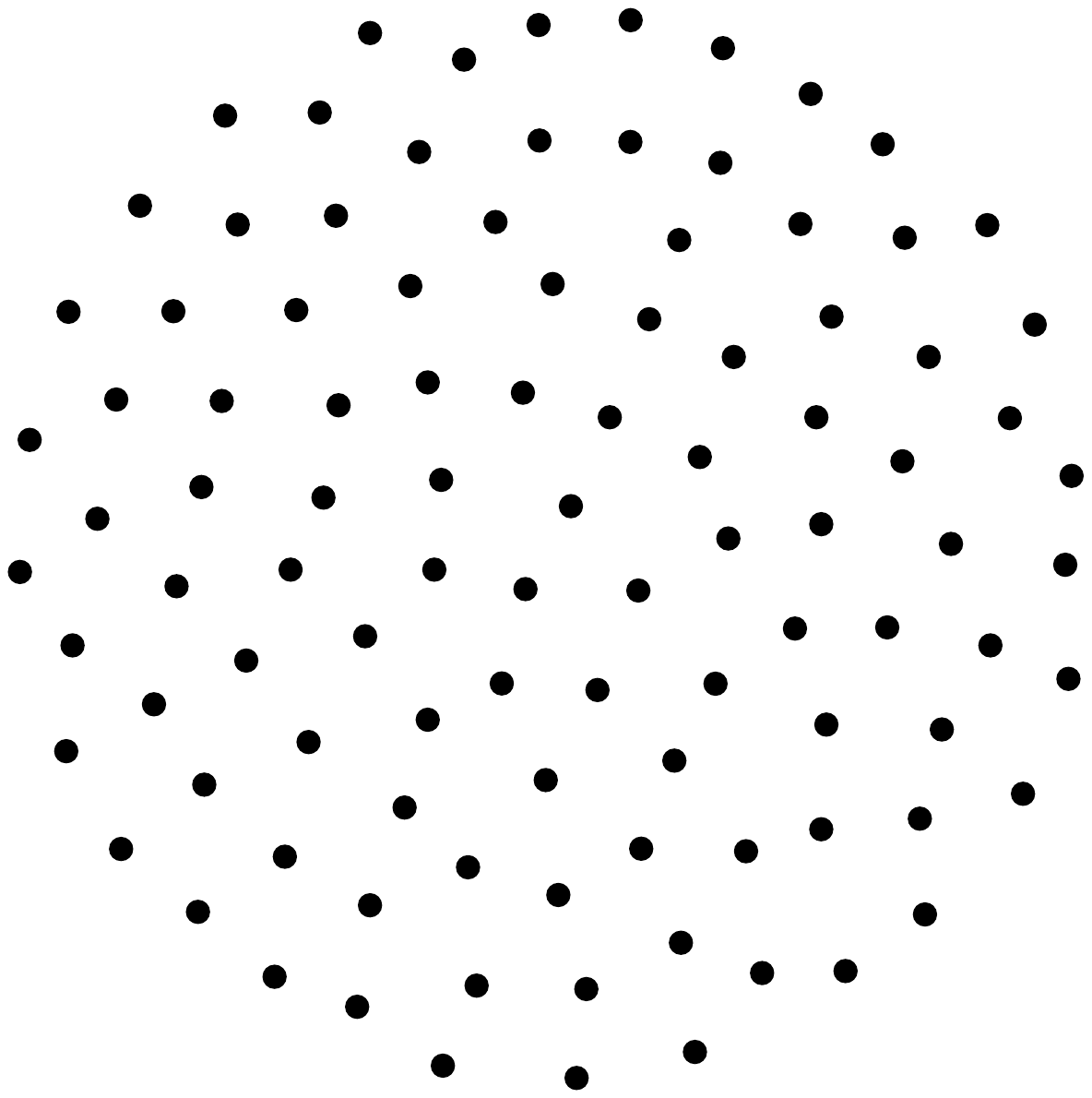} 
\end{center}
\end{minipage}
\begin{minipage}[c]{.46\linewidth}
\begin{center}
\includegraphics[scale=0.13]{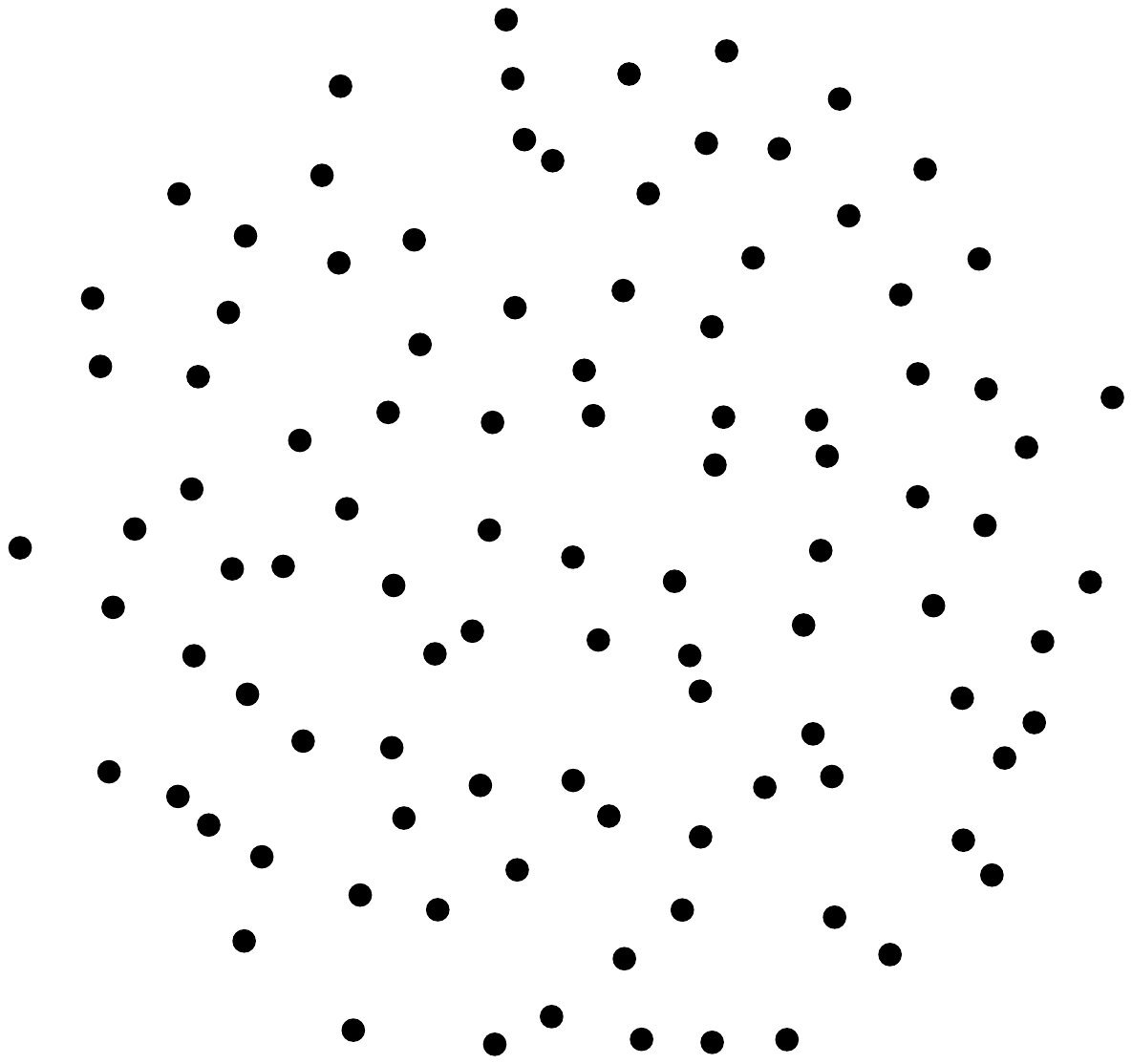} 
\end{center}
\end{minipage}
\vspace{-0.5cm}
%\label{fig:fig}
\caption{Case \LogD \ with $N = 100$ and $V(x) = |x|^2$, for $\beta = 400$ (left) and $\beta = 5$ (right).}
\end{figure}
In order to investigate it, we choose a microscopic observable which encodes the \textit{averaged microscopic behavior} of the system: the (tagged) empirical field, that we will now define. In the following, $\Sigma$ is the support of the equilibrium measure, and $\config$ denotes the space of point configurations. 

Let $\XN = (x_1, \dots, x_N)$ in $(\R^{\d})^N$ be fixed. 
\begin{itemize}
\item We define $\XN'$ as the finite configuration rescaled by a factor $N^{1/{\d}}$ 
\begin{equation} \label{def:ompN}
\XN' := \sum_{i=1}^N \delta_{N^{1/{\d}} x_i}.
\end{equation}
It is a point configuration (an element of $\config$), which represents the $N$-tuple of particles $\XN$ seen at microscopic scale.

\item We define the \textit{tagged empirical field}\footnote{Bars will always indicated tagged quantities} $\bEmp_N(\XN)$ as
\begin{equation}
\label{def:bEmp}
\bEmp_N(\XN) :=  \frac{1}{|\Sigma|} \int_{\Sigma} \delta_{\left(x,\,  \theta_{N^{1/{\d}} x} \cdot \XN' \right)} dx,
\end{equation}
where $\theta_x$ denotes the translation by $- x$. It is a probability measure on $\Sigma \times \config$.
\end{itemize}
For any $x$ in $\Sigma$, the term $\theta_{N^{1/{\d}}x} \cdot \XN'$ is an element of $\config$ which represents the $N$-tuple of particles $\XN$ centered at $x$ and seen at microscopic scale (or, equivalently, seen at microscopic scale and then centered at $N^{1/{\d}} x$). In particular any information about this point configuration in a given ball (around the origin) translates to an information about $\XN'$ around $x$. We may thus think of $\theta_{N^{1/{\d}}x} \cdot \XN'$ as encoding the behavior of $\XN'$ around $x$.

The \textit{empirical field} is the measure 
\begin{equation} \label{empiricalfield}
\frac{1}{|\Sigma|} \int_{\Sigma} \delta_{\theta_{N^{1/{\d}}x} \cdot \XN'} dx,
\end{equation}
it is a probability measure on $\config$ which encodes the behaviour of $\XN'$ around each point $x$ in $\Sigma$. 

The \textit{tagged empirical field} $\bEmp_N(\XN)$ defined in \eqref{def:bEmp} is a finer object, because for each $x$ in $\Sigma$ we keep track of the centering point $x$ as well as of the microscopic information $\theta_{N^{1/{\d}}x} \cdot \XN'$ around $x$. It yields a measure on $\Sigma \times \config$ whose first marginal is the Lebesgue measure on $\Sigma$ and whose second marginal is the (non-tagged) empirical field defined in \eqref{empiricalfield}. Keeping track of this additional information allows one to test $\bEmp_N(\XN)$ against functions $F(x, \C)$ which may be of the form
$$
F(x, \C) = \chi(x) \tilde{F}(\C),
$$
where $\chi$ is a smooth function localized in a small neighborhood of a given point of $\Sigma$, and $\C \mapsto \tilde{F}(\C)$ is a bounded continuous function on the space of point configurations. Using such test functions, we may thus study the microscopic behavior of the system after a small average (on a small, macroscopic domain of $\Sigma$).

Our main goal in this paper is to characterize the typical behavior of the tagged empirical field under $\PNbeta$. 

\subsection{Main result: large deviation principle and thermodynamic limit}
We let again $\config$ be the space of point configurations in $\R^{\d}$, endowed with the topology of vague convergence, and we consider the space $\probas(\Sigma \times \config)$ with the topology of weak convergence.

For any $\bP$ in $\probas(\Sigma \times \config)$ we will define two terms:
\begin{enumerate}
\item The renormalized energy $\bttW(\bP, \meseq)$  of $\bP$ with background $\meseq$ (see Section \ref{sec:defenergie}).
\item The specific relative entropy $\bERS[\bP|\Poisson^1]$ of $\bP$ with respect to the Poisson point process $\Poisson^1$ of intensity $1$ (see Section \ref{sec:defentropy}).
\end{enumerate}

For any $\beta > 0$, we define a free energy functional $\fbarbeta$ as
\begin{equation}
\label{def:bfbeta} \fbarbeta(\bP) := \frac{\beta}{2} \bttW(\bP,\meseq) + \bERS[\bP|\Poisson^1].
\end{equation}

For any $N, \beta$ we let $\bPgot_{N,\beta}$ be  the push-forward of the canonical Gibbs measure $\PNbeta$ by the \textit{tagged empirical field map} $\bEmp_N$ as in \eqref{def:bEmp} (in other words, $\bPgot_{N, \beta}$ is the law of the tagged empirical field when the particles are distributed according to $\PNbeta$).

We may now state our main result, under some assumptions on $V$ that  will be given in Section \ref{sec:def}.
\begin{theo}[Large Deviation Principle for the tagged empirical fields] \label{TheoLDP} Assume that \ref{H1}--\ref{H6} are satisfied. For any $\beta>0$ the sequence $\{\bPgot_{N,\beta}\}_N$ satisfies a large deviation principle at speed $N$ with good rate function $\fbarbeta - \inf \fbarbeta$.
\end{theo}
In particular, in the limit $N\to \infty$, the law $\bPgot_{N,\beta}$ concentrates on minimizers of $\fbarbeta$. One readily sees the effect of the temperature: in the minimization  there is a competition between the renormalized energy term $ \bttW( \cdot, \meseq)$, which is expected to favor very ordered configurations, and the entropy term which in contrast favors disorder (it is minimal for a Poisson point process).

As a by-product of the large deviation principle we obtain the order $N$ term in the expansion of the partition function.
\begin{coro}[Next-order expansion and thermodynamic limit]
\label{corothermo}
 Under the same assumptions, we have, as $N \to \infty$:
\begin{itemize}
\item In the logarithmic cases \LogU \ and \LogD, 
 \begin{equation} \label{expansionlog}
 \log \ZNbeta=  - \frac{\beta}{2} N^2 \mathcal \I_V(\meseq) +\frac{\beta}{2} \frac{N \log N}{\d} - N \min \fbarbeta  +N(|\Sigma|-1) + No_N(1).
\end{equation}
\item In the \Riesz \ cases
 \begin{equation}\label{expansionriesz}
 \log \ZNbeta= - \frac{\beta}{2} N^{2-\frac{\s}{\d} } \mathcal \I_V(\meseq) - N \min \fbarbeta  +N(|\Sigma|-1)  +  No_N(1).
 \end{equation} 
\end{itemize} 
The logarithmic cases enjoy a scaling property which allows to re-write the previous expansion as
\begin{multline} \label{logz}
 \log \ZNbeta =- \frac{\beta}{2} N^2 \mathcal \I_V(\meseq) +\frac{\beta}{2}\frac{N \log N}{\d} 
 -N C(\beta, \d)
 \\- N \left( 1 -\frac{\beta}{2 \d}\right) \int_\Sigma \meseq(x)\log \meseq(x) \, dx + N o_N(1),
 \end{multline} 
where $C(\beta, \d)$ is a constant depending only on $\beta$ and the dimension, but \textit{independent of the potential} $V$.
\end{coro}

In the \Riesz \ cases, by a similar scaling argument, we get\footnote{See \eqref{def:fbetasansbarre} for the definition of $\fbeta$.}
\begin{multline*}
\log \ZNbeta = - \frac{\beta}{2} N^{2-\frac{\s}{\d} } \mathcal \I_V(\meseq) - N \int_\Sigma  \meseq(x) \min \mathcal{F}^1_{\beta \meseq(x)^{\s/\d}}  - N \int_\Sigma \meseq(x)\log \meseq(x)\, dx \\ + N o_N(1).
\end{multline*}
Here $\beta$ and $\meseq$ are coupled, and at each point $x \in \Sigma$ there is  an \textit{effective temperature} depending on the equilibrium density $\meseq(x)$.

\subsection{Variational property of the sine-beta process}
In the particular case of \LogU \ with a quadratic potential $V(x) = x^2$, the equilibrium measure is known to be  Wigner's semi-circular law whose density is given by 
\begin{equation*}
x \mapsto \frac{1}{2\pi} \1_{[-2,2]} \sqrt{4-x^2}.
\end{equation*}
The limiting process at microscopic scale around a point $x \in (-2,2)$ (let us emphasize that here there is \textit{no averaging}) has been identified for any $\beta > 0$ in \cite{vv} and \cite{MR2484278}. It is called the \textit{sine-$\beta$ point process} and we denote it by $\sineb(x)$ (so that $\sineb(x)$ has intensity $\frac{1}{2\pi} \sqrt{4-x^2}$). 
For $\beta > 0$ fixed, the law of these processes do not depend on $x$ up to rescaling and we denote by $\sineb$ the corresponding process with intensity $1$. 

A corollary of our main result is a new variational property of $\sineb$.
 \begin{coro}[Sine-beta process]
\label{bsinebmin} 
For any $\beta > 0$, the point process $\sineb$ minimizes 
\begin{equation} \label{def:fbetasansbarre}
\fbeta(P) := \frac{\beta}{2} \ttW(P,1) + \ERS[P|\Poisson^1]
\end{equation}
among stationary point processes of intensity $1$ in $\R$.  
\end{coro}
The objects $\ttW$ and $\ERS$ are defined in Sections \ref{sec:defenergie} and \ref{sec:defentropy} respectively. They are the non-averaged  versions of $\bttW$ and $\bERS$. Corollary \ref{bsinebmin} is proven in Section \ref{proofSinGin}. The main interest of this result is to give a one-parameter family of free energy functionals which are minimized by $\sineb$.

The main other setting in which the limiting Gibbsian point process is identified is the case \LogD \ with quadratic external potential, which gives rise to the so-called \textit{Ginibre point process} (see \cite{ginibre,bors}). We can also prove that this process minimizes a similar free energy functional among stationary point processes of intensity $1$ in $\R^2$. However, the Ginibre point process does not come within a family indexed by $\beta$ and its properties are already very well-known, so we omit the proof here.

\subsection{Motivation} The main motivation for studying such systems comes from statistical physics and random matrix theory. 

In all cases of interactions, the systems governed by the Gibbs measure $\PNbeta$ are considered as difficult systems in statistical mechanics because the interactions are truly long-range, singular, and the points are not constrained to live on a lattice. The \LogD \ case is a two-dimensional \textit{Coulomb gas} or \textit{one-component plasma} (see e.g. \cite{alastueyjancovici}, \cite{jlm}, \cite{sm} for a physical treatment). Two-dimensional Coulomb interactions are also at the core of the fractional quantum Hall effects \cite{girvin2005introduction,stormer1999fractional}, Ginzburg-Landau vortices \cite{sandier2008vortices} and vortices in superfluids and Bose-Einstein condensates.  The \Riesz \ case with $\d = 3, \s = \d-2$ corresponds to higher-dimensional Coulomb gases, which can be seen as a toy (classical) model for matter (see e.g. \cite{PenroseSmith,jlm,LiLe1,LN}). 

The \LogU \ case corresponds to a one-dimensional \textit{log-gas} or \textit{$\beta$-ensemble}, and is of particular importance because of its connection to Hermitian random matrix theory (RMT), we refer to \cite{forrester} for a comprehensive treatment. In the most studied cases $\beta = 1, 2, 4$ with $V$ quadratic, the canonical Gibbs measure $\PNbeta$ coincides with the joint law of the $N$ eigenvalues of the so-called GOE, GUE, GSE ensembles.  The connection between the law of the eigenvalues of  random matrices and Coulomb gases was first noticed in  the foundational papers \cite{wigner,dyson}.

The general \Riesz \ case can be seen as a generalization of the Coulomb case, and motivations for its study are numerous in the physics literature (in solid state physics, ferrofluids, elasticity), see for instance \cite{mazars,bbdr,campa2009statistical,torquato2016hyperuniformity}. This case also corresponds to systems with Coulomb interaction constrained to a lower-dimensional subspace. Another motivation for studying such systems is the topic of approximation theory\footnote{In that context, the systems are usually studied on a $\d$-dimensional sphere or torus.}, as varying $\s$ from $0$ to $\infty$ connects \textit{Fekete points} to best  packing problems. We refer to the forthcoming monograph \cite{bhslivre}, the review papers \cite{sk,bhs} and references therein.

As always in statistical mechanics, one would like to understand if there are phase transitions for particular values of the (inverse) temperature $\beta$. For the systems studied here, one may expect what physicists call a liquid for small $\beta$, and a crystal for large $\beta$. Such a transition, occuring at finite $\beta$, has been conjectured in the physics literature for the \LogD \ case (see e.g. \cite{bst,caillol1982monte,choquard1983cooperative}) but its precise nature is still unclear (see e.g. \cite{stishov} for a discussion). Recent progress in computational physics concerning such phenomenon in two-dimensional systems (see e.g. \cite{kapfer2015two}) suggests a possibly very subtle transition between the liquid and solid phase.

\subsection{Related works}
The case \LogU \ has been most intensively studied, for general values of $\beta$ and general potentials. This culminated with very detailed results, including precise asymptotic expansions of the partition function \cite{bg1,bg2,sh2}, characterizations of the point  processes at the microscopic level  \cite{vv,MR2484278}, universality and rigidity results \cite{bey1,bey2,bfg, li2016rigidity}. The case \LogD \ has been mostly studied for $V$ quadratic or analytic in the case $\beta=2$, which is determinantal, see e.g. \cite{ginibre,bors, ridervirag,ahm,ahm2}, however the general $\beta$ case has recently attracted some attention, see \cite{BBNY, bauerschmidt2016two} and below. The Coulomb cases without temperature (formally $\beta=\infty$) are well understood with rigidity results on the number of points in microscopic boxes \cite{aoc,rns,petrache2016equidistribution,lieb2017local}. 

In all  cases, at the macroscopic scale, a large deviation principle for the law of the empirical measure $\emp(\XN)$ holds under the Gibbs measure 
\begin{equation*}
\frac{1}{\ZNbeta} \exp\left(-\frac{\beta}{2}\HN( \XN ) \right) d\XN,
\end{equation*}
i.e. \eqref{gibbs} with a different temperature scaling in the \Riesz \ cases. This LDP takes place at speed $N^2$ with rate function given by
$\frac{\beta}{2} (\mathcal \I_V - \I_V(\meseq))$, and was proven in \cite{hiaipetz,bg} (for \LogU), \cite{bz,bodguionnet} (for \LogD), and \cite{chafai2014first} in a general setting including \Riesz \ (see also \cite[Chap.2]{serfatyZur}). Our main result, Theorem \ref{TheoLDP}, can be understood as a next-order LDP on a microscopic quantity, or ``type-III" LDP. The idea of using large deviations methods for such systems already appeared in \cite{bbdr} where results of the same flavor but at a more formal level  are presented. 

The existence of a thermodynamic limit (as in Corollary \ref{corothermo}) had been known for a long time for the two and three dimensional Coulomb cases \cite{LN,sm,PenroseSmith}.  Our formula \eqref{logz} is to be compared with the results of \cite{sh2,bg1,bg2} in  the \LogU \ case, where  asymptotic expansions of $\log \ZNbeta$ are pushed much further, at the price of quite strong assumptions on the regularity of the potential $V$.  In the \LogD \ case, our result can be compared to the formal result of \cite{wz}. In both logarithmic cases, we recover  in \eqref{logz}  the   cancellation of the order $N$ term when $\beta=4$ in dimension 2 and $\beta =2$ in dimension $1$, as was observed in \cite[Part.II, Sec.II]{dyson} and \cite{wz}. 

\medskip

Our approach is in line with the ones of \cite{sandier20151d} for the case \LogU, \cite{sandier20152d} for the case \LogD, \cite{rougerie2016higher} for the general Coulomb cases and \cite{petrache2014next} for the general \Riesz \ case, and we borrow some tools from these papers.  They focused on the analysis of the microscopic behavior of minimizers, which formally corresponds to $\beta=\infty$, but the understanding of $\HN$ also allowed to deduce information on $\PNbeta$ for finite $\beta$, in terms of an asymptotic expansion the partition function and a qualitative description of the limit of $\PNbeta$, which are sharp only as $\beta \to \infty$. Our goal here is to obtain a complete LDP at speed $N$, valid for all $\beta$. 
\medskip

Several subsequent works by the authors rely strongly on the results of the present paper. 
\begin{itemize}
\item In \cite{leble2016logarithmic}, an alternative (more explicit) definition of the renormalized energy is introduced and used to study the limit $\beta \to 0$ of the minimisers of the free energy functional (proving convergence to the Poisson point process). For \LogU \ and \Riesz \  in dimension $\d = 1$, we also study  the limit $\beta \to +\infty$ and prove a rigorous crystallization result.
\item In \cite{loiloc}, focusing on \LogD , the result of Theorem \ref{TheoLDP} is pushed further to arbitrary mesoscopic averaging scales. It yields (non-optimal) local laws and rigidity estimates (see \cite{BBNY} for a similar, independent result with optimal rigidity). 
\item In \cite{leble2016fluctuations}, we prove a central limit theorem for the fluctuations of linear statistics in the \LogD \ case, for $\beta > 0$ arbitrary, under mild regularity assumptions on the test functions and the potential, see \cite{bauerschmidt2016two} for a similar, independent result. In \cite{bls} the approach is also implemented in the \LogU \ case, for possibly \textit{critical} potentials.
The analysis of \cite{leble2016fluctuations,bls} uses in a crucial way the expansion of the partition function as given in \eqref{logz}.
\item  The two-dimensional Coulomb system  with particles of opposite signs, also called classical Coulomb gas or \textit{two-component plasma} is  a fundamental model of statistical mechanics, related to the sine-Gordon or XY models, and to the celebrated Kosterlitz-Thouless phase transition. The approach of the present paper is extended and adapted to that setting in \cite{2D2CP}.
\item Finally, the case of \textit{hypersingular} Riesz interactions $\s > \d$, which are essentially not long-range, is treated in \cite{hardin2017large}.
\end{itemize}

\subsection{Outline of the proof, plan of the paper}
The starting point of the analysis is the following \textit{splitting formula}, obtained in \cite{petrache2014next} for the greatest generality. For any $\XN = (x_1, \dots, x_N)$ we have
\begin{align}
\label{splitlog} \text{(\LogU, \ \LogD)}  \quad & \HN(\XN)= N^2 \mathcal \I_V (\meseq)
 - \frac{N \log N}{\d}  + N\WN(\XN, \meseq) +2 N \sum_{i=1}^N \zeta(x_i), \\ 
\label{splits}\text{(\Riesz)}  \quad &
\HN(\XN)= N^2 \mathcal \I_V(\meseq) +  N^{1+\s/\d}\WN(\XN, \meseq) +2 N \sum_{i=1}^N \zeta(x_i),
 \end{align}
where $\WN$ is a next-order energy which will be defined later, and $\zeta $ is an effective confining term. In this paragraph, for simplicity, we will work as if $\zeta$ was $0$ on $\Sigma$ and $+ \infty$ on the complement $\Sigma^c$.

Using \eqref{splitlog}, \eqref{splits}, one can factor out some constant terms from the energy and the partition function, and reduce $\PNbeta$ to 
 \begin{equation} \label{gibequiv}
 d\PNbeta(\XN)= \frac{1}{K_{N,\beta}} \exp\left(-\frac{\beta N}{2} \WN(\XN, \meseq)\right) \1_{\Sigma^N}(\XN) d\XN,
 \end{equation}
 where $\KNbeta$ is a new partition function.
 
To prove a LDP, the standard method consists in evaluating  the logarithm of $\bPgot_{N,\beta} ( B(\bar{P},\ep))$, where $\bar{P}$ is a given element of $\probas (\Supp\times\config)$ and $B(\bar{P},\ep)$ is a ball of small radius $\ep$ around it, for a distance that metrizes the weak topology.

We may write 
 \begin{equation*}
 \bPgot_{N,\beta} (B(\bar{P},\ep)) \simeq \frac{1}{K_{N,\beta}}\int_{\bEmp_N(\XN) \in B(\bar{P},\ep)} \exp\left( -\frac{\beta N}{2} \WN(\XN, \meseq) \right) \prod_{i=1}^N\indic_{\Sigma}(x_i) dx_i, 
 \end{equation*}
and thus we obtain formally 
\begin{multline}\label{formel}
\lim_{\ep\to 0} \log \bPgot_{N,\beta}( B(\bar{P}, \ep))= -\log K_{N,\beta} -\frac{\beta N}{2} \WN(\bar{P}, \meseq)\\ + \lim_{\ep \to 0}\log \left|\{\XN \in \Sigma^N, \bEmp_N(\XN) \in B(\bar{P},\ep)\}\right|.
\end{multline}
Extracting this way the exponential of a function is the idea of Varadhan's integral lemma (cf. \cite[Theorem 4.3.1]{dz}), and works when the function (here $\WN(\cdot, \meseq)$) is continuous. In similar contexts to ours, this idea is used e.g. in \cite{Georgii1,georgiz}.

In \eqref{formel} the term in the second line is the logarithm of the volume of point configurations whose associated tagged empirical field is close to $\bar{P}$. By classical large deviations theorems, such a quantity is expected to be the entropy of $\bar{P}$. More precisely since we are dealing with  empirical fields, we need to use the specific relative entropy (as e.g. in \cite{Georgii1}), which is a relative entropy \textit{per unit volume} (as opposed to the usual relative entropy, which in this context would only take the values $0$ or $+ \infty$). 
 
 The most problematic term in \eqref{formel} is the second one in the right-hand side, $\WN(\bar{P}, \meseq)$, which really makes no sense. The idea is that it should be close to $\bttW(\bar{P}, \meseq)$ which is the well-defined infinite-volume quantity appearing in the rate function \eqref{def:bfbeta}.  If we were dealing with a continuous function of $\bar{P}$ then the replacement of $\WN(\bar{P}, \meseq) $ by $\bttW(\bar{P}, \meseq)$ would be fine. 
 However there are three difficulties:
 \begin{enumerate}
 \item $\WN(\cdot, \meseq)$ depends on $N$ and we need to take the limit $N\to \infty$,
 \item  this  limit cannot be uniform because the interaction becomes infinite when two points approach each other,
 \item $\WN$ is not adapted to our topology, which retains only local information on the point configurations, while $\WN(\cdot, \meseq)$ contains long-range interactions and does not depend only on the local arrangement of the points but on the global configuration.
 \end{enumerate}
Thus, the approach outlined above cannot work directly. Instead, we look again at the ball $B(\bar{P}, \ep)$ and show that we can find therein a logarithmically  large enough volume of configurations for which we can replace $\WN(\XN, \meseq)$ by $\bttW(\bar{P}, \meseq)$. This will give a lower bound on $\log \bPgot_{N,\beta}( B(\bar{P}, \ep))$, while the upper bound is in fact much easier to deduce from the previously known results of \cite{petrache2014next}. The second obstacle above, related to the discontinuity of the energy near the diagonals of $(\Rd)^N$, is handled by truncating the interaction at small distances and controlling the error, which is shown to be small often enough (namely,  the volume of the configurations where it is small is large enough at a logarithmic scale).
 
The third point above (the fact that the total energy is nonlocal in the data of the configuration) is the most delicate one.  The way we circumvent it is via the \textit{screening procedure} developed in \cite{gl13,sandier20152d,sandier20151d,rougerie2016higher,petrache2014next}.  Roughly speaking, we can always modify a bit each configuration in order to make the energy that it generates additive (hence local) in space, while not changing the empirical field too much nor losing too much logarithmic volume in \smallskip phase-space.

The paper is organized as follows:
\begin{itemize}
\item  Section \ref{sec:def} contains our assumptions, the definitions of the renormalized energy and of the specific relative entropy, as well as some  notation.
\item In Section \ref{sec:prelim} we present preliminary results on the renormalized energy, some borrowed from previous works.
\item Section \ref{sec-main} contains the proofs of the large deviation principle and its corollaries, assuming two intermediate results.
\item  In Section \ref{sec:screening} we adapt the screening procedure of previous works to the present setting, and we prove that it can be applied with high probability. We introduce the regularization procedure which deals with the singularity of the interaction at short distances. 
\item  In Section \ref{sec:construction} we complete the proof of the main technical result, by showing that given a random point configuration we can often enough screen and regularize it to have the right energy.
\item  In Section \ref{preuvesanov} we prove a second intermediate result, a large deviation principle for empirical fields under a reference measure without interactions. 
\item  In Section \ref{annex}, we collect miscellaneous additional proofs.
\end{itemize}

\subsection{Open questions}
Let us conclude our introduction by gathering some open questions related to the present work.
\begin{itemize}
\item  Our result naturally raises two questions: the first is to better understand  $\bttW(\cdot, \meseq)$ and its minimizers  and the second is to better understand the specific relative entropy, about which not much is  known in general. Identifying the minimizers of   $\fbarbeta $ or identifying some of their properties seems to be a difficult problem. 
\item  It easy to see that if $\beta$ is finite, the minimizer of our free energy functional cannot be a periodic point process and in particular it cannot be the point process associated to some lattice (or crystal). Hence there is no crystallization in the strong sense at finite temperature, i.e. the particles cannot concentrate on an exact lattice. However some weaker crystallization could occur at finite $\beta$ e.g. if  the connected two-point correlation function $\rho_{2} -1$ of minimizers decays more slowly to $0$ as $\beta$ gets larger. Hints towards such a behavior of $\rho_2-1$ for \LogU \ may be found in \cite{ForresterCbeta} where an explicit formula for the two-point correlation function is computed for the limiting point process associated to the $\beta$-Circular Ensemble (which according to \cite{na14} turns out to  also be $\sineb$).
\item The uniqueness of minimizers of $\fbarbeta$ is expected to hold for \LogU, but the one-dimensional Riesz case is unclear. In higher dimensions, it is natural to ask whether the rotational invariance of $\fbarbeta$ accounts for all the degeneracy.
%\item Such a change in the long-distance behavior of the two-point correlation function would not imply a first-order phase transition, as would be implied e.g. by the existence of two minimizers of $\fbeta$ with different energies. Let us note that uniqueness of the minimizers for \LogU \  (or at least the fact that they all have the same energy, hence the same entropy) would for example allow to retrieve as a straightforward corollary of our LDP the equipartition property shown in \cite{BMSequi} for $\beta$-ensembles.
\item It is conjectured that the triangular (or Abrikosov) lattice has minimal energy in the \LogD \ case (see \cite{blanc2015crystallization} for a survey). Can we at least prove that any minimizer of the renormalized energy has infinite specific relative entropy, which would be a first hint towards their conjectural “ordered” nature?
\item Is there a limit to the Gibbsian point process defined as the push-forward of $\PNbeta$ by $\XN \mapsto \sum_{i=1}^N \delta_{N^{1/{\d}}x_i}$? Convergence is only known for \LogU \ (the limit is the $\sineb$ process mentioned above) and for $\LogD$ in the $\beta =2$ case.
\end{itemize}
 \vspace{0.3cm}
 
{\it Acknowledgements :} We would like to thank Paul Bourgade, Percy Deift, Tamara Grava, Jean-Christophe Mourrat, Nicolas Rougerie  and Ofer Zeitouni for useful comments.

\section{Assumptions and main definitions} \label{sec:def}
\subsection{The assumptions}\label{sec:assumptions}
Let us state our assumptions on $V$ and the associated equilibrium measure $\meseq$.
Before doing so, we recall that a compact set $K$ is said to have positive $\g$-capacity if there exists a probability measure $\mu$ on $K$ such that 
$$\iint \g(x-y)\, d\mu(x)\, d\mu(y)<+\infty,$$ if not it has zero $\g$-capacity.
A general set $E$  has positive $\g$-capacity if it contains a compact set that does. 
\begin{description}
\item[\namedlabel{H1}{(H1)} - Regularity of $V$] The potential $V$ is lower semi-continuous and bounded below, and the set 
$$\{x: V(x)<\infty\}$$ 
has positive $\g$-capacity. 
\item[\namedlabel{H2}{(H2)} - Growth assumption]  We have
\begin{align*}
(\LogU, \ \LogD)\quad & \lim_{|x|\to \infty} \frac{V(x)}{2}- \log |x|= + \infty, \\
(\Riesz) \quad & \lim_{|x|\to\infty} V(x)=+\infty. 
\end{align*}
\end{description}
Assumptions \ref{H1}, \ref{H2} imply, by the results of \cite{frostman}, that the functional $\I_V$ defined in \eqref{MFener} has a unique minimizer (denoted by $\meseq$) among probability measures on $\R^{\d}$, which furthermore has a compact support (denoted by $\Sigma$). 

We make the following additional assumptions: 
\begin{description}
\item[\namedlabel{H4}{(H3)} - Regularity of the equilibrium measure] $\muv$  has a density which is $ C^{0, \kappa}(\Sigma)$ for some $0<\kappa\le 1$. In particular, there exists $\om > 0$ such that 
$$
\meseq(x) \leq \om \text{ on $\Sigma$.}
$$
\item[\namedlabel{H5}{(H4)} - Regularity of the boundary] Let $\mathring{\Sigma}$ be the interior of $\Sigma$
$$
\mathring{\Sigma} := \{x \in \Rd, \meseq(x) > 0\},
$$
and let $\Gamma := \partial \mathring{\Sigma}$ be its boundary. We assume that there is a finite number of connected components of $\Gamma$, and we enumerate them as
$$
\Gamma := \bigcup_{j \in J} \Gammaj.
$$
Each $\Gammaj$ is a $ C^1$ submanifold of dimension $\ll_j \le \d-1$.  For any $j \in J$, there exist constants $\cc, \cC,\alpha_j \ge 0$ and a neighborhood $\Uj$ of $\Gammaj$ in $\Sigma$ such that  
\begin{align}
\label{conda1} & \cc  \,\dist (x, \Gammaj)^{\alpha_j} \le \muv(x) \le \cC\,\dist (x, \Gammaj)^{\alpha_j} \text{ on } \Uj \\ 
\label{conda2} & \muv \in C^{0, \min(\alpha_j, 1)}(\Uj).
\end{align}
Moreover, if $\alpha_j \geq 1$, we impose that
\begin{equation}
\label{conda3}
|\nab \muv(x)| \le \cC \dist (x, \Gammaj)^{\alpha_j - 1} \text{ in } \Uj.
\end{equation}
\end{description}
These assumptions include the case of the Wigner's semi-circle law arising for a quadratic potential in the \LogU \ case. We also know that in the Coulomb cases, a quadratic potential gives rise to an equilibrium  measure  which is a multiple of a characteristic function of a ball, also covered by our assumptions with $\alpha=0$. Finally, in the Riesz case, it was noticed in \cite[Corollary 1.4]{chafai2014first}  that any compactly supported radial profile can be obtained as the equilibrium measure associated to some potential. Our assumptions are thus never empty. They also allow to treat a variety of so-called {\it critical cases}, for instance those where $\mu_V$ vanishes in the bulk of its support.  We will also comment more on known regularity results and sufficient conditions at the end of Section \ref{sec22} below.

Our last assumption is an integrability condition, ensuring the existence of the partition function.
\begin{description}
\item[\namedlabel{H6}{(H5)}] Given $\beta$, for $N$ large enough, we have
\begin{align*}  
(\LogU, \ \LogD) &\int \exp\left(-\beta N^{1-\frac{\s}{\d}} \left( \frac{V(x)}{2}- \log |x|\right)\right) \, dx <\infty  \\ 
(\Riesz) & \int \exp\left(-\frac{\beta}{2}  N^{1-\frac{\s}{\d}}  V(x)\right) \, dx < + \infty.
\end{align*}\end{description}
It is easy to see that \ref{H6} is satisfied as soon as $V$ grows fast enough at infinity.

\subsection{The effective confinement term}\label{sec22}
This paragraph is devoted to defining the effective confinement term $\zeta$ appearing above. Under Assumptions \ref{H1}, \ref{H2}, the result of Frostman cited above ensures that, introducing the $\g$-potential generated by $\meseq$
\begin{equation}
\label{def:hpotential}
\hmuv(x) :=\int_{\R^{\d}} \g(x-y) d\meseq(y),
\end{equation}
and letting 
\begin{equation} 
\label{def:cV}
\cV :=\I_V(\meseq) - \hal \int_{\Rd}  V(x) d\meseq(x),
\end{equation} 
we have\footnote{Here it is only important to know that quasi-everywhere implies Lebesgue almost-everywhere.} 
\begin{align*}
&\hmuv  +\tfrac{V}{2} \ge \cV \text{ quasi-everywhere (q.e.) on } \R^{\d}, \\
&\hmuv  +\tfrac{V}{2} = \cV \text{ q.e. on } \Sigma.
\end{align*}
We may now define the function $\zeta$ that appeared before, as
\begin{equation}\label{defzeta}
\zeta:= \hmuv +\tfrac{V}{2} - \cV.
\end{equation}
We let $\omega$ be the zero set of $\zeta$, we have the inclusion
\begin{equation} \label{defomega}
\Sigma \subset \omega := \{\zeta = 0\}.
\end{equation}
In the Coulomb cases the function $\hmuv$ can be viewed as the solution to an obstacle problem (see for instance \cite[Sec. 2.5]{serfatyZur}) and in the \Riesz \, and \LogU \, cases, as the solution to a fractional obstacle problem (see \cite{css}). The set $\omega$  corresponds to the {\it contact set} or {\it coincidence set} of the obstacle problem, and $\Supp$ is the set where the obstacle is \textit{active}, sometimes called the {\it droplet}. 
 Thanks to this connection, the regularity of $\meseq$ can be implied by that of $V$. Let us summarize the known facts for the Coulomb case (see \cite{css} for the fractional case):
\begin{itemize}
\item If $V$ is  $C^{1,1}$ then the density of the equilibrium measure is given by
$$
d\meseq(x) = \frac{1}{4\pi} \Delta V(x) \mathbf{1}_{\Sigma}(x) dx.
$$
In particular, if $V\in C^{2, \kappa}$ then $\meseq$ has a $C^{0,\kappa}$ density on its support. If $\Delta V>0$ near $\omega$  then $\Sigma $ and $\omega$ coincide.
\item The points of the boundary $\partial \omega$ of the coincidence set can be either \textit{regular}, i.e. $\partial \omega$ is locally the graph of a $C^{1, \kappa}$ function, or \textit{singular}, i.e. $\partial \omega$ is locally cusp-like (this classification was introduced in \cite{MR1658612}). Singular points are nongeneric and we implicitly assume that they are absent by assuming \ref{H5}, for technical reasons which  might be bypassed.
\item If $V$ is $C^{3, \kappa}$, then $\partial \omega$ is locally $C^{2, \kappa}$ around each regular point (see \cite[Thm. I]{caffarelli1976smoothness}).
\item In the setting of a bounded domain with zero Dirichlet boundary condition, if $V$ is strictly convex (which implies that $\Sigma$ and $\omega $ coincide) and of class $C^{k+1,\kappa}$ on $\R^2$, it was shown (see \cite[Section 4]{kinderlehrer1978variational}) that $\Sigma$ is connected and that $\partial \Sigma$ is $C^{k, \kappa}$ with no singular points. 
%Hence any strictly convex potential in $C^{2,\kappa}$ with the growth condition \ref{H2} satisfies our assumptions. 
\end{itemize}

\subsection{The extension representation for the fractional Laplacian}
\label{sec:deflap}
%In the next two sections, we recall elements from \cite{petrache2014next}.
%Our method of proof relies on expressing the interaction part of the Hamiltonian  as a quadratic integral  of the potential  generated by the point configuration via
%$$ \g \star \sum_{i} \delta_{x_i}$$
%(where $\star$ denotes the convolution product) and expanding this integral interaction to next order in $N$.
In general, the kernel $\g$ is not the convolution kernel of a local operator, but rather of a fractional Laplacian. Here we use the \textit{extension representation} of \cite{cafsil}: by adding one space variable  $y\in \R$ to the space $\R^\d$, the nonlocal operator can be transformed into a local operator of the form $\div (\yg \nab \cdot)$.

In what follows, $\k$ will denote the dimension extension. We will take $\k=0 $ in the Coulomb cases (i.e. \LogD \ and \Riesz \ with $\d \ge 3, \ \s = \d-2$), for which $\g$ itself is the kernel of a local operator. In all other cases, we will  take $\k=1$. Points in the space $\mr^\d$ will be denoted by $x$, and points in the extended space $\mr^{\d+\k}$ by $X$, with $X=(x,y)$, $x\in \mr^\d$, $y\in \mr^\k$. We will often identify $\R^\d \times \{0\}$ and $\R^\d$.

% We will also use
%the notation 
%\begin{equation}\label{notkr}
%K_R=[-R/2,R/2]^d  \mbox{ as well as} \ K_R=[-R/2,R/2]^d \times \{0\}\end{equation}
%\begin{equation}\label{tkr}
%\tilde{K}_R= [-R/2,R/2]^d \times (-R/2,R/2)\end{equation}

If $\gamma$ is chosen such that 
\begin{equation}\label{gs}
\d-2+\k+ \gamma  =\s,
\end{equation}
then, given a probability measure $\mu$ on $\R^\d$, the $\g$-potential generated by $\mu$, defined in $\R^\d$ by
\begin{equation*}
\mathsf{h}^{mu}(x) := \int_{\R^{\d}} \g(x-\tilde{x}) \, d\mu(\tilde{x})
\end{equation*}
can be extended to a function $\Hmu$ on $\R^{\d+\k}$ defined by 
\begin{equation*}
\Hmu(X) := \int_{\R^{\d+\k}} \g(X - \tilde{X}) \, d\mu(\tilde{X}) = \int_{\R^{\d}} \g(X - (\tilde{x},0)) \, d\mu(\tilde{x}), 
\end{equation*}
and this function satisfies 
\begin{equation}
\label{divh}
- \div (\yg \nab H^\mu)=  \cds {\mu} \delta_{\R^\d}
\end{equation}
where by $\delta_{\R^\d}$ we mean the uniform  measure on $\R^\d\times \{0\}$. The corresponding values of the constants $\cds$ are given in \cite[Section 1.2]{petrache2014next}.  In particular, the potential $\g$ seen as a function of $\R^{\d+\k}$ satisfies
 \begin{equation}\label{eqg}
 - \div (\yg \nab \g)= \cds \delta_0.
 \end{equation}

%In other words  $\mu \drd$ 
%acts on test functions $\vp$ by 
%$$\int_{\R^{\d+\k}} \vp (X)d (\mu \drd)(X) = \int_{\R^{\d}} \vp(x, 0) \, d\mu(x),$$
% and 
% \begin{equation}
% \cds= \left\lbrace\begin{array}{ll}2s\,\frac{2\pi^{\frac{d}{2}}\Gamma\left(\frac{s+2-d}{2}\right)}{\Gamma\left(\frac{s+2}{2}\right)}&\text{ for }s>\max(0,d-2)\ ,\\[3mm]
%                      (d-2)\frac{2\pi^{\frac{d}{2}}}{\Gamma(d/2)}&\text{ for }s=d-2>0\ ,\\[3mm]
%                      2\pi&\text{ in cases \eqref{wlog}, \eqref{wlog2d}}\ .
%                     \end{array}
%\right.
% \end{equation}
% In order to recover the Coulomb cases, it suffices to take $k=\gamma=0$, in which case we retrieve the fact that $g$ is a multiple of the fundamental solution of the Laplacian. If $s>d-2$ we take $k=1$ and 
%$\gamma$ satisfying \eqref{gs}. In the case \eqref{wlog}, we note that $\g(x)=- \log |x|$ appears as the $y=0$ restriction 
%of $-\log |X|$, which is (up to a factor $2\pi$) the fundamental solution  to the Laplacian operator in dimension $d+k=2$. In this case, we may thus choose $k=1$ and $\gamma=0$, $\cds=c_{1,0}=2\pi$, and the potential $H^\mu=\g \star \mu$ still satisfies \eqref{divh}, while $\g$ still satisfies \eqref{eqg}. 

To summarize, we will take 
\begin{itemize}
\item $\k=0, \gamma=0$ in the Coulomb cases.
\item $\k=1, \gamma = 0$ for \LogU.
\item $\k=1, \gamma = \s - \d + 2 - \k$ in the remaining \Riesz \ cases.
\end{itemize}
We may note that our  assumption $\d-2\le \s<\d$  implies  that $\gamma$ is always in $(-1,1)$. We refer to \cite[Section 1.2]{petrache2014next} for details about this extension representation.

\subsection{Truncating the interaction, spreading the charges}
We briefly recall the procedure used in \cite{petrache2014next}, following \cite{rougerie2016higher}, for truncating the interaction or, equivalently, spreading out the point charges.

For any $\eta \in (0,1)$, we define
\begin{equation}
\label{def:truncation} \g_{\eta} := \min(\g, \g(\eta)), \quad \f_{\eta} := \g - \g_{\eta}.
\end{equation}
We also define  
\begin{equation} \label{defde}
\delta_0^{(\eta)}:= - \frac{1}{\cds} \div (\yg \g_{\eta}), 
\end{equation}
it is a positive  measure supported on $\partial  B(0, \eta)$.

\subsection{The next-order energy: finite configurations}
\label{sec:defW}
In this section, we define the quantity $\WN$ appearing in the splitting formulas \eqref{splitlog}, \eqref{splits}. 

Let $N \geq 1$ and let $\XN$ be a $N$-tuple of points in $\R^{\d}$. We introduce the following blown-up (or zoomed) quantities:
\begin{itemize}
\item For any $i \in \{1, \dots, N\}$ we let $x'_i := N^{1/\d} x_i$.
\item We define $\meseq'$ as $\meseq'(x) := \meseq(N^{-1/\d}x)$. In particular, $\meseq'$ is a positive, absolutely continuous measure of total mass $N$, with support $N^{1/\d} \Sigma$.
\end{itemize}
We let $\hpN$ be the $\g$-potential (in the \textit{extended} space $\R^{\d+\k}$) generated by the $N$ points $x'_1, \dots, x'_N$ and the measure $\meseq'$ seen as a negative density of charges
\begin{equation} \label{def:hpN}
\hpN(X) := \int_{\R^\d} \g\(X - (\tilde{x}, 0) \right) \left(\sum_{i=1}^N \delta_{x'_i} - d \meseq'\right) (\tilde{x}).
\end{equation}
We will call $\nab \hpN$ the \textit{local electric field}.
%generated by the electric system 
%$$
%\text{electric system = blown-up charges $x'_1, \dots, x'_N$ + negative background of density $\meseq'$.}
%$$
Let us observe that, from \eqref{eqg}, we have
\begin{equation}\label{HNeq}
-\div\left(\yg \nab \hpN\right)= \cds \left( \sum_{i=1}^N \delta_{x_i'}-\meseq' \delta_{\R^{\d}} \right)
\end{equation}
Next, we define the truncated version of $\hpN, \nabla \hpN$. We can equivalently let $\hpNe$ be
$$
\hpNe(X) := \int_{\R^\d} \g \(X - Y \right) \left(\sum_{i=1}^N \delta_{x'_i}^{(\eta)} - \meseq'\delta_{\R^{\d}} \right) (Y),
$$
or define it as
\begin{equation} \label{def:hpNe}
\hpNe(X) := \hpN(X) - \sum_{i=1}^N \f_{\eta}(X - x'_i),
\end{equation}
with the notation of \eqref{def:truncation}. We observe that
\begin{equation} \label{propnabhpNe}
- \div\left(\yg \nab \hpNe \right)= \cds \left( \sum_{i=1}^N \delta^{(\eta)}_{x_i'}-\meseq' \delta_{\R^{\d}} \right),
\end{equation}
with $\delta_x^{(\eta)}$ as in \eqref{defde}.

Finally, we let the next-order energy $\WN$ be
\begin{equation}
\label{defWN}
\WN(\XN, \meseq) := \frac{1}{N \c} \lim_{\eta\to 0}\left( \int_{\R^{\d+\k} } \yg |\nab \hpNe|^2 - N \cds \g(\eta)\right).
\end{equation}
It is proven in \cite{petrache2014next} that the limit in \eqref{defWN} exists and that with this definition, the splitting formulas \eqref{splitlog}, \eqref{splits} hold. With the factor $\frac{1}{N}$ the quantity $\WN(\cdot, \meseq)$ is expected to be typically of order $1$.

Let us emphasize two aspects of \eqref{defWN}. First, $\WN$ is defined as a single integral of a quadratic quantity (the local \textit{electric field}) instead of a double integral analogous to the summation $\sum_{i \neq j} \g(x_i - x_j)$ appearing in the original energy. This is due to the fact that (after extension of the space) the interaction $\g$ is the kernel of a local operator, and relies on a simple integration by parts. Secondly, $\WN$ is defined through a truncation procedure, letting the truncation parameter $\eta$ to $0$ and substracting a divergent quantity $\cds \g(\eta)$ for each particle, this is the \textit{renormalization} feature (hence the name \textit{renormalized energy}).

\subsection{Point configurations, point processes, electric fields} \label{sec-ppp}
Before defining the relevant \textit{limit objects} (the energy and entropy terms appearing in the rate function of our large deviation principle), we introduce the functional setting as well as some notation.

\subsubsection{Generalities.}
If $X$ is a Polish space and $d_X$ a compatible distance, we endow the space $\probas(X)$ of Borel probability measures on $X$ with the distance:
\begin{equation} \label{DudleyDistance}
d_{\probas(X)}(P_1, P_2) : = \sup \left\lbrace \left| \int_X F (dP_1 - dP_2) \right|,  F \in \Lip(X) \right\rbrace
\end{equation}
where $\Lip(X)$ denotes the set of functions $F : X \rightarrow \R$ that are $1$-Lipschitz with respect to $d_X$ and such that $\|F\|_{\infty} \leq 1$. It is well-known that this metrizes the topology of weak convergence on $\probas(X)$.

If $P \in \probas(X)$ is a probability measure, we denote by $\Esp_P \left[ \cdot \right]$ the expectation under $P$.

For any $x \in \Rd$ and $R>0$ we denote $\carr_R(x)$ the hypercube of center $x$ and sidelength $R$ (all the hypercubes will have their sides parallel to the axes of $\Rd$). If $x$ is not specified we let $\carr_R = \carr_R(0)$.

\subsubsection{Configurations of points.} Let us list some basic definitions.

If $A$ is a Borel set of $\Rd$ we denote by $\config(A)$ the set of locally finite point configurations in $A$ or equivalently the set of non-negative, purely atomic Radon measures on $A$ giving an integer mass to singletons (see \cite{dvj}). The mass $|\mathcal {C}|(A)$ of $\mc{C}$ on $A$ corresponds to the number of points of the point configuration in $A$. We mostly use $\C$ for denoting a point configuration and we will write $\C$ for $\sum_{p \in \C} \delta_p$.

We endow  $\config := \config(\Rd)$ with the topology induced by the topology of weak convergence of Radon measure (also known as vague convergence or convergence against compactly supported continuous functions), and we define the following distance on $\config$
\begin{equation} \label{dconfig}
\dconfig(\C,\C') := \sum_{k \geq 1} \frac{1}{2^k} \frac{\sup \left\lbrace \left| \int_{\carr_k} f d\C - d\C'\right|, f \in \Lip(\Rd) \right\rbrace}{|\C|(\carr_k) + |\C'|(\carr_k)}.
\end{equation}
The subsets $\config(A)$ for $A \subset \Rd$ inherit the induced topology and distance.

%If $B$ is a compact subset of $\Rd$ we endow $\config(B)$ with the following distance: 
%\begin{equation} \label{defdistanceconfigB}
%d_{\config(B)}(\mathcal{C},\mathcal{C'}) := \sup \left\lbrace \int_{B} F (d\mathcal{C} - d\mathcal {C'}) |\  F \in \Lip(B) \right\rbrace.
%\end{equation}

We say that a function $F : \config \rightarrow \R$ is local when there exists $k \geq 1$ such that for any $\C \in \config$ it holds
 \begin{equation} \label{funlocal}
F \left(\mathcal{C}\right) = F \left(\mathcal{C} \cap \carr_k \right).
\end{equation}
We denote by $\Loc_k(\config)$ the set of functions that satisfies \eqref{funlocal}.

\begin{lem} \label{Ldistconfig} The following properties hold: 
\begin{itemize}
\item The topological space $\config$ is Polish.
\item The distance $\dconfig$ is compatible with the topology on $\config$. 
\item For any $\delta > 0$ there exists an integer $k$ such that
\begin{equation*}
\sup_{F \in \Lip(\config)} \sup_{\C \in \config} |F(\C) - F(\C \cap \carr_k)| \leq \delta.
\end{equation*} 
\end{itemize}
\end{lem}
Lemma \ref{Ldistconfig} is proven in Section \ref{sec:preuveLdistconfig}.  

The additive group $\Rd$ acts on $\config$ by translations $\{\theta_t\}_{t \in \Rd}$ as follows: 
\begin{equation*}
\mathcal{C} = \{x_i, i \in I\} \mapsto  \theta_t \cdot \mathcal{C} := \{x_i - t, i \in I\}.
\end{equation*}
We will use the same notation for the action of $\Rd$ on Borel sets of $\Rd$: if $A$ is Borel and $t \in \Rd$, we denote by $\theta_t \cdot A$ the translation of $A$ by the vector $-t$.

For any finite configuration $\C$ with $N$ points we consider the subset of $(\Rd)^N$ of $N$-tuples corresponding to $\C$ (by allowing all the point permutations). If $\calA$ is a family of finite configurations with $N$ points we denote by $\Leb^{\otimes N}(\calA)$ the Lebesgue measure of the corresponding subset of $(\Rd)^N$. 

%The following lemma is elementary.
%\begin{lem}[Compactness in $\config$] 
%\label{CompactConfig} 
%Let $C : \R \rightarrow \R^{+}$ be an arbitrary function, then the following set is compact in $\config$ :
%$$\left\lbrace \mc{C} \in \config \ | \ \Numb(x,R)(\mc{C}) \leq C(R) \text{ for all } R  >0 \right\rbrace.$$
%\end{lem}
%\begin{proof} It follows from the compactness of the hypercubes $\carr_R$ for all $R > 0$ (hence of their powers $\carr_R^n$) and from the definition \eqref{dconfig} of the distance on $\config$, together with a diagonal extraction procedure in order to extract a subsequence converging on each $\carr_k$ for $k \geq 1$. 
%\end{proof}

\subsubsection{Point processes.} Strictly speaking, elements of $\config$ are \textit{point processes} and elements of $\probas(\config)$ are \textit{laws of point processes}. However, in this paper, in order to lighten the sentences, we make the following confusion: elements of $\config$ are called \textit{point configurations} (as above), a point process is defined as an element of $\probas(\config)$, and a tagged point process is a probability measure on $\Lambda \times \config$ where $\Lambda$ is some Borel set of $\R^{\d}$ with non-empty interior (usually $\Lambda$ will be $\Sigma$, the support of the equilibrium measure).

We impose by definition that the first marginal of a tagged point process $\bPst$ is the Lebesgue measure on $\Lambda$, normalized so that the total mass is $1$. As a consequence, we may consider the disintegration measures\footnote{We refer e.g. to \cite[Section 5.3]{AGS} for a definition.} $\{\bPst^x\}_{x \in \Lambda}$ of $\bPst$. For any $x \in \Lambda$, $\bPst^x$ is a probability measure on $\config$ and we have, for any $F \in C^0\left(\Lambda \times \config\right)$ 
\begin{equation*}
\Esp_{\bPst}[F] = \frac{1}{|\Lambda|} \int_{\Lambda} \Esp_{\bPst^x}[F(x, \cdot)] dx.
\end{equation*}

We denote by $\probas_s(\config)$ the set of translation-invariant (or stationary) point processes. We also call stationary a tagged point process such that the disintegration measure $\bPst^x$ is stationary for (Lebesgue-)a.e. $x \in \Lambda$ and we denote by $\probas_s(\Lambda \times \config)$ the set of stationary tagged point processes.  

If $P$ is stationary, we define its intensity as the quantity $\Esp_{P} \left[ \Numb_1 \right]$, where $\Numb_1(\C)$ denotes the number of points in the unit hypercube.

We will denote by $\probas_{s,1}(\config)$ the set of stationary point processes of intensity $1$ and by $\probas_{s,1}(\Lambda \times \config)$ the set of stationary tagged point processes such that
$$
\int_{x \in \Lambda} \Esp_{\bPst^x} [\Numb_1] dx = 1.
$$

%We define the following \textit{scaling map} allowing us to pass bijectively from a point process of intensity $m$ to a point process of intensity $1$.
%\begin{equation} \label{defscalingP}
%\sigma_m \Pst  := \text{ the push-forward of } \Pst  \text{ by } \mc{C} \mapsto m^{1/{\d}} \mc{C}.
%\end{equation}

\begin{remark} \label{memetopologie} We endow $\probas(\config)$ with the  topology of weak convergence of probability measures. Another natural topology on $\probas(\config)$ is \textit{convergence of the finite distributions} \cite[Section 11.1]{dvj2}, sometimes also called \textit{convergence with respect to vague topology for the counting measure of the point process}. These topologies coincide as stated in \cite[Theorem 11.1.VII]{dvj2}.
\end{remark}

\subsubsection{Electric fields.} Let $\p \in (1,2)$ be fixed, with 
\begin{equation}
\label{def:pmax}
\p < \p_{\rm{max}} := \min\left(2, \frac{2}{\gamma +1}, \frac{\d+\k}{\s+1}\right).  
\end{equation}

We define the class of electric fields as follows: let $\C$ be a point configuration and $m \geq 0$, let $E$ be a vector field in $\Lploc(\R^{\d+\k}, \R^{\d+\k})$, we say that $E$ is an \textit{electric field} compatible with $(\C, m)$ if\footnote{Compare with \eqref{HNeq}.}
\begin{equation}\label{eqclam}
-\div (\yg E) = \cds \left(\C - m \drd \right).
\end{equation}
We denote by $\Elec(\C, m)$ the set of such vector fields, by $\Elec_m$ the union over all configurations for fixed $m$, and by $\Elec$ the union over $m \geq 0$.

For any $E \in \Elec_m$, there exists a unique underlying configuration $\C$ such that $E$ is compatible with $(\C, m)$, we denote it by $\conf_m(\C)$.

We define an \textit{electric field process} as an element of $\probas(\Lploc(\R^{\d+\k}, \R^{\d+\k}))$ concentrated on $\Elec$, usually denoted by $\Pelec$. We say that $\Pelec$ is \textit{stationary} when it is invariant under the (push-forward by) translations $\theta_x \cdot E := E(\cdot - x)$ for any $x \in \R^{\d} \subset \R^{\d} \times \{0\}^k$. We say that $\Pelec$ is compatible with $(P, m)$, where $P$ is a point process, provided $\Pelec$ is concentrated on $\Elec_m$ and the push-forward of $\Pelec$ by the map $\conf_m$ coincides with $P$.

 Finally, we define a \textit{tagged electric field process} as an element of $\probas(\Sigma \times \Lploc(\R^{\d+\k}, \R^{\d+\k}))$ concentrated on $\Sigma \times \Elec$, usually denoted by  $\bPelec$, whose first marginal is the normalized Lebesgue measure on $\Sigma$. We say that $\bPelec$ is \textit{stationary} if for a.e. $x\in \Sigma$, the disintegration measure $\bar{P}^{\mathrm{elec},x}$ is stationary (in the previous sense).
%
%\subsubsection{Application of the stationarity.}
%We end this section with an elementary lemma exposing a consequence of the stationarity assumptions which we will make a constant use of.
%\begin{lem} \label{lemstat}
%For any $P$ stationary (point or electric) process,  resp. $\bar{P}$ stationary (point or electric) tagged process, for every $T,R>0$, for any $\Phi$ scalar nonnegative function of the point configuration or electric field $X$, we have 
%$$ \Esp_P\left[ \dashint_{\carr_T\times \R^k} \Phi(X (x) )\, dx\right]=   \Esp_P\left[ \dashint_{\carr_R\times \R^k} \Phi(X(x) )\right].$$ 
%Moreover $\Esp_P\left[ \lim_{R \ti}  \dashint_{\carr_R\times \R^k} \Phi(X(x) )\right]$ exists and coincides with $\Esp_P\left[ \dashint_{\carr_T\times \R^k} \Phi(X (x) )\, dx\right]$ for any $T > 0$. 
%\end{lem}
%\begin{proof}
%The multiparameter ergodic theorem (cf. \cite{becker}) ensures that for any $T > 0$
%\begin{multline*}
%\Esp_P\left[ \dashint_{\carr_T \times \R^k} \Phi(X(x)) \, dx\right]= \Esp_P\left[ \lim_{R \to \infty} \frac{1}{R^{\d}} \int_{\carr_R} \dashint_{\carr_T\times \R^k}  \Phi( X  (\lambda+ x) ) \, dx\, d\lambda \right]\\ = \Esp_P\left[ \lim_{R \to \infty} \frac{1}{R^{\d}} \dashint_{\carr_R} \Phi (X(x)  ) \right]
% \end{multline*} where we used Fubini's theorem and the fact that $\indic_{C_{R-T}}\le \indic_{C_R}* \indic_{C_T} \le \indic_{C_{R+T}}$ and $\Phi$ nonnegative. The result follows.
%\end{proof} 

\subsection{The renormalized energy: definition for infinite objects} \label{sec:renominfi}
%Note that while for a finite configuration of $N$ points, we may find a unique potential generated by it via \eqref{defHN}, for an infinite configuration there is no canonical choice of such a potential (one may always add the gradient of a function satisfying $-\div (\yg \nab H)=0$. This explains the need for a definition based on the electric field, and a definition down at the level of points. 

\subsubsection{For an electric field.} Let $\C$ be a point configuration, $m \geq 0$ and let $E$ be in $\Elec(\C, m)$. We define the renormalized energy of $E$, following \cite{petrache2014next}, as follows.

For any $\eta \in (0,1)$ we define the truncation of $E$ as
\begin{equation} \label{defEeta1}
E_{\eta}(X) := E(X) - \sum_{x \in \C} \nabla \f_{\eta}(X-(x,0)),
\end{equation}
where $\f_{\eta}$ is as in \eqref{def:truncation}.

The renormalized energy of $E$ with background $m$ is obtained by first defining
\begin{equation} \label{Weta}
\mc{W}_\eta(E,m) := \limsup_{R \ti} \frac{1}{R^{\d}} \int_{\carr_R\times \R^k} \yg |E_{\eta}|^2 - m \cds \g(\eta),
\end{equation}
and finally\footnote{The existence of the limit as $\eta \to 0$ is proven in \cite{petrache2014next}.}
\begin{equation}\label{defW}
\mc{W}(E,m) := \lim_{\eta\to 0} \mc{W}_\eta(E,m).
\end{equation}

The name \textit{renormalized energy} (originating from \cite{bbh} in the context of two-dimensional Ginzburg-Landau vortices) reflects the fact that the integral of $\yg |E|^2 $ is infinite, and is computed in a renormalized way by first applying a truncation and then removing the appropriate divergent part $\cds\g(\eta)$.

\subsubsection{For an electric field process.}
If $\Pelec \in \probas(\Elec)$, and $m \geq 0$ we define 
\begin{equation} \label{deftW}
\tW_\eta(\Pelec, m):= \Esp_{\Pelec} \left[ \mc{W}_{\eta}( \cdot, m) \right] \qquad \tW(\Pelec, m) :=  \Esp_{\Pelec} \left[ \mc{W}( \cdot, m) \right].
\end{equation}

Let $\bPelec \in \probas(\Sigma \times \Elec)$ be a tagged electric field process such that for a.e. $x \in \Sigma$,  the disintegration measure $\bar{P}^{\mathrm{elec},x}$ is concentrated on $\Elec_{\meseq(x)}$. We define 
\begin{equation}\label{238b}
\mc{W} (\bPelec, \meseq):= \int_{\Sigma} \tW( \bar{P}^{\mathrm{elec},x}, \meseq(x)) dx.
\end{equation}

\subsubsection{For a point configuration.} 
%For any $m > 0$ and for any admissible gradient vector field $E \in \mathcal{A}_m$ we let
%\begin{equation}
%\label{defconfm}
%\conf_m(E) := \frac{-\div (\yg E)}{\cds}  + m \drd
%\end{equation}
%be the underlying point configuration.
%For any $E \in \A$ there is exactly one value of $m > 0$ such that $E \in \mc{A}_m$ and we let $\conf(E) := \conf_m(E)$ for the suitable value of $m$, this defines a map $\A \rightarrow \config$ and we denote by $\configplus \subset \config$ its image i.e. the set of point configurations $\mc{C}$ for which there exists at least one admissible gradient vector field $E$ such that $\conf(E) = \mc{C}$. It is clear that  the maps $\conf_m : \A_m \rightarrow \config$ and $\conf : \A \rightarrow \config$ are measurable. Let us note that the fiber of $\conf$ at any $\mc{C} \in \configplus$ is always infinite, if $E$ is in the fiber of $\mc{C}$ we can simply add to $E$ the gradient of any function satisfying $\div (\yg \nab H)=0$  on $\R^{\d+\k}$ and by doing so we recover exactly the fiber of $\mc{C}$. 
%
%We may then define the renormalized energy of a point configuration/process by means of the renormalized energy of electric field/processes in the fiber of $\conf$.

Let $\C$ be a point configuration and $m \geq 0$. We define the renormalized energy of $\C$ with background $m$ as
\begin{equation*}
\W(\mc{C},m) := \frac{1}{\cds} \inf\{\mc{W}(E,m) \ | \ E \in \Elec(\C, m) \}
\end{equation*}
with the convention $\inf (\varnothing) = +\infty$.

In Section \ref{sec:proofinfatteint}, we prove the following:
\begin{lem}\label{infatteint}
Let $\C, m$ be fixed. If $\k =0$, two elements of $\Elec(\C, m)$ with finite energy differ by a constant vector field, and if $\k=1$, there is at most one element in $\Elec(\C,m)$ with finite energy. In all cases, the $\inf$ in the definition of $\W(\C,m)$ is a uniquely achieved minimum.
\end{lem}

\subsubsection{For a point process.} \label{sec:defenergie}
Let $\Pst$ be in $\probas(\config)$ and $m \geq 0$. We define its renormalized energy with background $m$ as
\begin{equation*}
\ttW(\Pst, m) := \Esp_{P}\left[\W(\cdot, m)\right].
\end{equation*}

Finally, if $\bPst \in \probas(\Sigma \times \config)$ is a tagged point process, we define its \textit{renormalized energy with background measure $\meseq$} as 
\begin{equation*}
\bttW(\bPst, \meseq) := \int_{\Sigma} \ttW(\bPst^{x}, \meseq(x)) dx.
\end{equation*}

\subsection{The specific relative entropy} \label{sec:defentropy}
We conclude by defining the second term appearing in the rate function $\fbarbeta$ (see \eqref{def:bfbeta}), namely the specific relative entropy. 

For any $m \geq 0$, we denote by $\Poisson^m$ the (law of the) Poisson point process of intensity $m$ in $\Rd$, which is an element of $\probas_s(\config)$. Let $P$ be in $\probas_s(\config)$. The specific relative entropy of $P$ with respect to $\Poisson^1$ is defined as
\begin{equation} \label{def:ERS}
\ERS[P|\Poisson^1] := \lim_{R \ti} \frac{1}{R^{\d}} \ent[P_{\carr_R}|\Poisson^1_{\carr_R}],
\end{equation}
where $P_{\carr_R}, \Poisson^1_{\carr_R}$ denote the restriction of the processes to the hypercube $\carr_R$. Here, $\ent[\cdot|\cdot]$ denotes the \textit{usual} relative entropy of two probability measures defined on the same probability space, namely
\begin{equation*}
\ent[\mu|\nu] := \int \frac{d\mu}{d\nu} \log \left(\frac{d\mu}{d\nu}\right)  d\nu,
\end{equation*}
if $\mu$ is absolutely continuous with respect to $\nu$, and $+ \infty$ otherwise. 

\begin{lem} \label{lem:ERS} The following properties are known:
\begin{itemize}
\item The limit in \eqref{def:ERS} exists for $P$ stationary.
\item The map $P \mapsto \ERS[P |\Poisson^1]$ is affine and lower semi-continuous on $\probas_s(\config)$.
\item The sub-level sets of $\ERS[\cdot| \Poisson^1]$ are compact in $\probas_s(\config)$ (it is a \textit{good} rate function).
\item We have $\ERS[P| \Poisson^1] \geq 0$ and it vanishes only for $P = \Poisson^1$.
\end{itemize}
\end{lem}
\begin{proof}
We refer to \cite[Chapter 6]{seppalainen} for a proof. The first point follows from sub-additivity, the thrid and fourth ones from usual properties of the relative entropy. The fact that $\ERS[\cdot|\Poisson^1]$ is an \textit{affine} map, whereas the classical relative entropy is strictly convex, is due to the infinite-volume limit taken in \eqref{def:ERS}.  
\end{proof}

Now, if $\bar{P}$ is in $\probas_{s}(\Sigma \times \config)$, we define the tagged relative specific entropy as
\begin{equation}
\label{def:bERS} \bERS[\bPst|\Poisson^1] := \int_{\Sigma} \ERS[\bPstx|\Poisson^1] dx.
\end{equation}

\section{Preliminaries on the energy}\label{sec:prelim}
\subsection{Connection between next-order and renormalized energy}
The renormalized energy, defined in Section \ref{sec:renominfi} for infinite objects, was derived in previous works  as a certain limit of the next-order energy $\WN$ (defined for finite $N$) which appears in the \textit{splitting} identities \eqref{splitlog}, \eqref{splits}. In particular, we have the following lower bound.
\begin{prop} \label{prop:LowerBoundenergies} 
Let $\{\XN\}_N$ be a sequence of $N$-tuples of points in $\Rd$, and assume that $\{\WN(\XN, \meseq)\}_N$ is bounded. Then, up to extraction, the sequence $\{\bEmp_N(\XN)\}_N$ converges to some $\bP$ in $\probas_s(\Sigma \times \config)$, and we have
\begin{equation} \label{gliminf}
\liminf_{N \ti} \WN(\XN, \meseq) \geq \bttW(\bP, \meseq).
\end{equation}
\end{prop}
\begin{proof}
This follows from \cite[Proposition 5.2]{petrache2014next}  and our definitions.
\end{proof}

\subsection{Discrepancy estimates}
In this section we give estimates to control the discrepancy between the number of points in a domain and the expected number of points according to the background intensity, in terms of the energy. These estimates show that local non-neutrality of the configurations has an energy cost, which in turn implies that stationary point processes of finite energy must have small discrepancies. For simplicity we only consider processes of intensity $1$, but the results extend readily to the general case of intensity $m > 0$.

For a given point configuration $\C$, we denote by $\Numb_{R}(x)$ the \textit{number of points} of a configuration in $\carr_{R}(x)$ and by $\dis_R(x)$ the \textit{discrepancy} in $\carr(x,R)$, defined as
\begin{align}
\label{def:Numb} &\Numb_R(x)(\C) := |\C|(\carr_R(x)) \\
\label{def:Dis} &\dis_R(x)(\C) := \Numb_R(x)(\C)  - R^{\d}. 
\end{align}
If  $x$ is not specified, we let $\Numb_R(\C) = \Numb_R(0)(\C), \dis_R(\C) := \dis_R(0)(\C)$.
\begin{lem} \label{LemmeDiscr} Let $\Pst$ be in $\probas_s(\config)$.
If $\ttW(P, 1)$ is finite, then $\Pst$ has intensity $1$. Moreover, we have
\begin{equation}\label{borneDiscr}
\Esp_{\Pst}\left[\dis_R^2 \right] \leq C (C+  \ttW(\Pst,1)) R^{\d+\s},
\end{equation}
where $C$ depends only on $\d, \s$.
%In particular we have, for $R > 1$
%\begin{equation} \label{borneDiscr2}
%\Esp_{\Pst}\left[\Numb^2(0,R)\right] \leq R^{2\d} + C (C+  \ttW(\Pst, 1)) R^{\d+\s},
%\end{equation}
%where $C$ depends only on $\d, \s$.
\end{lem}
\begin{proof}
We postpone the proof of Lemma \ref{LemmeDiscr}  to Section \ref{sec:proofLemDiscr}.
\end{proof}

In particular, in the \LogD \ case, \eqref{borneDiscr} yields
\begin{equation*}
\Esp_{\Pst} \left[ \dis_R^2\right] = O(R^{2}),
\end{equation*}
hence the variance of the number of points for a process of finite energy is comparable to that of a Poisson point process. It is unclear to us whether this estimate is sharp or not. 

In the \LogU \ case, the same argument can be used (see again Section \ref{sec:proofLemDiscr} for a proof) to get the following.
\begin{remark} \label{rem:Discr1d}
Let $P$ be  in $\probas_{s,1}(\config)$ such that $\ttW(P,1)$ is finite, in the $\LogU$ case. Then we have 
\begin{equation} \label{discrcas1d}
\liminf_{R \ti} \frac{1}{R} \Esp_{P} [ \dis_R^2 ] = 0.
\end{equation}
In particular the Poisson point process $\Poisson^1$ has infinite renormalized energy for $\d=1,\s=0$.
\end{remark}

\subsection{Almost monotonicity of the energy and truncation error}
The following lemma, taken from \cite{petrache2014next}, expresses the fact that the limit $\eta \to 0$ defining $\WN$ as in \eqref{defWN} is almost monotonous. It also provides an estimate on the truncation error.
\begin{lem}\label{prodecr}
Let $\XN$ be a $N$-tuple in $\Rd$ and let $\hpN$ be as in \eqref{def:hpN}. For any $0 < \tau < \eta < \hal$, we have
\begin{multline*}
- C N \|\muv\|_{L^\infty}  \eta^{\frac{\d-\s}{2}} \le \left(\int_{\R^{\d+\k}}\yg |\nab H'_{N,\tau}|^2 - N \cds \g(\tau) \right) -
\left(\int_{\R^{\d+\k}}\yg |\nab \hpNe|^2 - N \cds  \g(\eta) \right)
\\ \le  C N \|\muv\|_{L^\infty} \eta^{\frac{\d-\s}{2}} +
\cds \sum_{i\neq j,  |x_i-x_j|\le 2\eta}  \g(x_i-x_j).
\end{multline*}
for some constant $C$ depending only on $\d$ and $\s$. 

In particular, sending $\tau \t0$ we get
\begin{multline}\label{truncaestimate}
o_\eta(1) \le  \WN(\XN, \meseq) -  \left(\frac{1}{\cds N} \int_{\R^{\d+\k}}\yg |\nab H'_{N,\eta}|^2 -  \g(\eta) \right)\\ \le 
   o_\eta(1) + \frac{1}{N} \sum_{i\neq j,  |x_i-x_j|\le 2\eta} \g(x_i-x_j).
\end{multline} 
where the error term $o_\eta(1)$ is independent of the configuration.
\end{lem}
\begin{proof}
This follows from \cite[Lemma 2.3]{petrache2014next}. 
\end{proof}

We will also need a \textit{lower bound} on the truncation error, as follows.
\begin{lem}\label{discerrtronc}
Let $\bP$ be in $\probas_{s}(\Sigma \times \config)$, such that $\bttW(\bP, \meseq)$ is finite. For any $\eta \in (0,1)$ and $\tau \in (0, \eta^2/2)$ we have
\begin{multline}\label{eqlemdisc}
\bttW_{\tau} (\bP, \meseq)- \bttW_\eta(\bP, \meseq)
\geq  C \frac{\g(2\tau) }{\tau^\d} \Esp_{\bar{P}} [ (\Numb_{\tau}^2 -1)_+ ] \\  
+ C \Esp_{\bar{P}} \left[  \sum_{p \neq q \in \mc{C} \cap \carr_1 , |p-q| \le \eta^2/2 } \g(p-q) \right]  - o_{\eta}(1),
\end{multline} 
with a  $C>0$ and $o_{\eta}$ depending only on $\d,\s$.
\end{lem}
\begin{proof}
We postpone the proof to Section \ref{sec:preuvelowerboundtrunca}.
\end{proof}

\subsection{Compactness results for electric fields}
\begin{lem} \label{compactElec} Let $K$ be some hyperrectangle in $\Rd$, let $\{E^{(n)}\}_{n}$ be a sequence of vector fields in $\Lploc(K, \R^{\d+\k})$,  let $\{\mc{C}_n\}_n$ be a sequence of point configurations in $K$ and $\{\mu_n\}_n$ be a sequence of bounded measures in $K$, such that $\{\mc{C}_n\}_n$ converges to some $\mc{C}$ in $\config(K)$ and that $\{\mu_n\}_n$ converges to some $\mu$ (in $L^{\infty}(K)$).  

Assume that for any $n \geq 1$, we have
\begin{equation}\label{eqstru}
-\div (\yg E^{(n)} )= \cds \left( \mc{C}_n - \mu_n \delta_{\Rd} \right) \text{ in } K\times \R^{\k},
\end{equation}
Moreover, let $\eta \in (0,1)$, and assume that $\{\int_{K \times \R^k} \yg |E^{(n)}_{\eta}|^2\}_{n}$ is bounded. 

Then there exists a vector field $E$ satisfying
\begin{equation} \label{passerlimiteEnE}
-\div(\yg  E) = \cds \left( \mc{C} - \mu \delta_{\Rd} \right) \text{ in } K \times \R^\k,
\end{equation}
and such that for any $z \in [0, +\infty]$
\begin{equation}\label{compiteE1}
\int_{K \times [-z,z]^\k} \yg |E_{\eta}|^2 \leq \liminf_{n + \infty} \int_{K \times [-z,z]^\k} \yg |E^{(n)}_{\eta}|^2.
\end{equation}
Moreover, if $\k =1$, for any $z \geq 1$ we have\footnote{We are looking at the field on the additional axis away from $\R^{\d} \times \{0\}$,  and $E_{\eta}$  coincides with $E$  there.}
\begin{equation} \label{compiteE2}
\int_{K \times \left(\Rd \backslash (-z, z)\right)} \yg |E|^2 \leq \liminf_{n + \infty} \int_{K \times \left(\R \backslash (-z, z)\right)} \yg |E^{(n)}|^2. 
\end{equation}
\end{lem}
\begin{proof} 
By H\"older's inequality  the space $L^2_{\yg}$ embeds locally into $L^{\p} $ for $\p \le \p_{\rm{max}}$ defined in \eqref{def:pmax}, and thus for any electric field we have
$$\|E_\eta\|_{L^{\p}(\carr_R\times \R^\k)}^2 \le C_R\int_{\carr_R \times \R^\k} \yg |E_\eta|^2.$$
In addition, using \eqref{defEeta1}, we have
\begin{equation}\label{Elp}
\|E\|_{L^{\p}(\carr_R\times \R^\k)}\le \|E_{\eta}\|_{L^{\p}(\carr_R\times \R^\k)}+ C_\eta \Numb_{R+1} .\end{equation}
Since the sequence $\{E^{(n)}_{\eta}\}_n$ is  bounded in $L^2_{\yg}(K \times \R^\k, \R^{\d+\k})$, and since the number of points in each cube $\carr_R$ is uniformly bounded by convergence of $\C_n$, we deduce that 
$\|E^{(n)}\|_{L^{\p} (\carr_R)} $ is bounded for each  $R>1$, 
 hence we may find a weak limit point $E$ in $L^{\p}_{\rm{loc}}$ which will satisfy \eqref{passerlimiteEnE}
 by taking the limit  as $n \to \infty$ in \eqref{eqstru} in the distributional sense. Lower semi-continuity as in \eqref{compiteE1} and \eqref{compiteE2} is then a consequence of the weak convergence.
\end{proof}

We also state a compactness result for stationary electric processes with bounded energy.
\begin{lem} \label{compacPelec}
Let $\{\Pelec_n\}_n$ be a sequence of stationary electric processes concentrated on $\Elec_1$ such that $\{\tW(\Pelec_n, 1)\}_n$ is bounded. Then, up to extraction, the sequence $\{\Pelec_n\}_n$ converges to a stationary electric process $\Pelec$ concentrated on $\Elec_1$ and such that
\begin{equation}
\label{lsctW1} \tW(\Pelec,1) \leq \liminf_{n \ti} \tW(\Pelec_n,1).
\end{equation}
\end{lem}
\begin{proof} 
By stationarity we have for any $R > 0$
\begin{equation} \label{Pelecnstationarity}
\tW_{\eta}(\Pelec_n,1)= \Esp_{\Pelec_n} \left[ \frac{1}{R^{\d}} \int_{\carr_R \times \R^k} \yg |E_\eta|^2 \right]- \cds \g(\eta),
\end{equation}
and thus by boundedness of the energy and the monotonicity of Lemma \ref{prodecr}, for any  $\eta<1$ fixed and for every $R$ we have 
$$\Esp_{\Pelec_n} \left[  \int_{\carr_R \times \R^\k} \yg |E_\eta|^2 \right]\le C_R.$$
Using \eqref{Elp} and the fact that  $\Numb_{R+1}$ is bounded by a constant depending only on $R$ in view of Lemma \ref{LemmeDiscr}, we deduce that 
$$\Esp_{\Pelec_n} \left[ \|E\|_{L^{\p}(\carr_R\times \R^\k)}^2  \right]\le C_{R,\eta}.$$
This immediately implies the tightness of $\Pelec$ for the $L^{\p}_{\rm{loc}}$ topology, and the existence of a limit point supported in $\Elec$.
 The function $E \mapsto \int_{\carr_R \times \R^\k} \yg |E_\eta|^2$ is lower semi-continuous, thus if $\Pelec$ is a limit point we have, for any $\eta > 0$
\begin{equation*}
\liminf_{n\to \infty} \tW_{\eta}(\Pelec_n,1) \ge \tW_{\eta}(\Pelec,1).
\end{equation*}
Sending $\eta $ to $ 0$ yields \eqref{lsctW1}.
\end{proof}

\subsection{Some properties of the renormalized energy}
We begin by the following technical lemma.
\begin{lem}  \label{lem:concordance}
Let $\Pst$ be in $\probas_s(\config)$ such that $\ttW(\Pst,1)$ is finite. We have
\begin{equation} \label{ecritureconcordance}
\ttW(\Pst,1) = \min \left\lbrace \Esp_{\Pelec}\left[\tW( \cdot, 1)\right] \ | \ \Pelec \textit{ stationary and compatible with } (P,1) \right\rbrace.
\end{equation}
\end{lem}
The proof of Lemma \ref{lem:concordance} is given in Section \ref{sec:preuveconcordance}.

We now study the regularity properties of the renormalized energy \textit{at the level of stationary point processes}.
\begin{lem} \label{lem:LSCW}
The map $P \mapsto \ttW(P,1)$ is lower semi-continuous on the space $\probas_{s}(\config)$, and its sub-level sets are compact.
\end{lem}
\begin{proof}
Let $P$ be in $\probas_s(\config)$ and let $\{P_n\}_n$ be a sequence of stationary point processes converging to $P$. We want to show that
\begin{equation*}
\liminf_{n \ti} \ttW(P_n,1) \geq \ttW(P,1).
\end{equation*}
We may assume that the left-hand side is finite (otherwise there is nothing to prove), and up to extraction we may also assume that the $\liminf$ is a $\lim$. 

By Lemma \ref{lem:concordance}, for each $n \geq 1$ we may find  a stationary electric process $\Pelec_n$  compatible with $(P_n,1)$ and such that
$$
\Esp_{\Pelec_n}\left[\tW( \cdot, 1) \right] = \ttW(P_n,1).
$$
The sequence $\{ \tW(\Pelec_n, 1)\}$ is bounded, which together with Lemma \ref{compacPelec} implies that up to extraction we have $\Pelec_n \to \Pelec $ for some electric process $\Pelec$ which is stationary and compatible with $(P,1)$.

Moreover, combining \eqref{lsctW1} with the fact that $\tW(\Pelec_n, 1) = \ttW(P_n, 1)$ for each $n$, and that $\ttW(P,1) \leq \tW(\Pelec,1)$ (by definition), we get 
$$
\liminf_{n \ti} \ttW(P_n,1) \geq \ttW(P,1),
$$
which proves that $\ttW(\cdot, 1)$ is lower semi-continuous on $\probas_s(\config)$. Also, we know from \cite{petrache2014next} that $\mc{W}(\cdot, 1)$ is bounded below hence so is $\ttW(\cdot, 1)$. 

Compactness of the sub-level sets is elementary, indeed from Lemma \ref{LemmeDiscr} we know that if $\ttW(P,1)$ is finite then $P$ has intensity $1$, but a family of stationary point processes with fixed intensity is tight in $\probas(\config)$.
\end{proof}

\subsection{Minimality of the local energy}
%Given a configuration $\mc{C}$ density $\mu $ on $K$, there exist many electric vector fields that are compatible with the configuration and the background $\mu$ in $K$ i.e. such that
%$$
%-\div(\yg E)= \cds (\mc{C}-\mu\drd\right) \text{ in K}
%$$
%indeed to any such vector field one may add any solution of $-\div(\yg E)=0$.
% 
%Since the configuration in a given compact set is finite there is however a natural choice, which we call the “local field”, given by 
%\begin{equation}\label{Eloc}
% \Eloc := \nabla \Hloc, \text{ with } \Hloc := \cds g \start (\mc{C} - \mu \delta_{\Rd}\indic_K).
%\end{equation}
%%which is well-defined if $\mc{C}(K)=\int_K \mu$.
For any $N \geq 1$, and any $\XN$ in $(\Rd)^N$, we have introduced in \eqref{def:hpN} the potential $\hpN$ and its gradient was called the \textit{local electric field}. Adding a solution of $-\div(\yg E)=0$ to this local electric field yields another vector field such that
\begin{equation} \label{ENversuslocal}
- \div(\yg E_N) = \cds \left(\sum_{i=1}^N \delta_{x'_i} - \mu'_V \delta_{\Rd} \right).
\end{equation}
The following lemma shows that among all $E_N$ satisfying \eqref{ENversuslocal}, the local electric field $\nabla \hpN$ has a smaller energy than any \textit{screened} electric field (in a sense made precise). The reason is that $\nabla \hpN$ is an $L^2_{\yg}$-orthogonal projection of any generic compatible $E_N$ onto gradients, and that the projection decreases the $L^2_{\yg}$-norm.
\begin{lem} \label{minilocale} Let $K$ be a compact subset of $\R^{\d}$ with piecewise $C^1$ boundary, let $N \geq 1$, let $\XN$ be in $(\Rd)^N$. We assume that all the points of $\XN$ belong to $K$ and that $\Sigma \subset K$. 

Let $E$ be a vector field in $\Lploc(\R^{\d+\k}, \R^{\d+\k})$ such that 
  \begin{equation}\label{checr}
\left\lbrace \begin{array}{ll}
 -\div (\yg E) 
= \cds \left( \sum_{i=1}^N \delta_{x'_i} - \meseq' \drd\right) & \text{in} \ K\times \R^\k \\
 E \cdot \vec{\nu} = 0  & \text{on} \  \partial  K \times \R^\k.
 \end{array}\right.
\end{equation} 
We let also $\hpN$ be as in \eqref{def:hpN}. Then, for any $\eta \in (0,1)$ we have
\begin{equation}\label{comparloc}
\int_{\R^{\d+\k}} \yg |\nabla \hpNe|^2 \leq \int_{K\times \R^{k}} \yg |E_{\eta}|^2.
\end{equation}
\end{lem}
\begin{proof}
First we note that we may extend $E$ by $0$ outside of $K$, and since $E \cdot \vec{\nu}$ is continuous across $\partial  K$, no divergence is created there, so that the vector field $E$ satisfies 
\begin{equation}
\label{diveg}
- \div (\yg E)= - \div (\yg \nabla \hpN) \quad \text{in} \ \R^{\d+\k}.
\end{equation}

Let us also observe that since the electric system $\XN$ with background $\meseq$ is globally neutral, the field $\hpN$ decays as $|x|^{-\s-1}$ as $|x|\to \infty$ in $\R^{\d+\k}$ and $\nabla \hpN$ decays like $|x|^{-\s-2}$ (still with the convention $\s=0$ in the logarithmic cases).

If the right-hand side of \eqref{comparloc} is infinite, then there is nothing to prove. If it is finite,  given $M>1$, and letting $\chi_M$ be a smooth nonnegative function equal to $1$ in $\carr_M\times [-M,M]^{\k}$ and $0$ at distance $\ge 1$ from $\carr_M \times [-M,M]^{\k}$,  we may write
\begin{multline}
\int_{\R^{\d+\k}} \chi_M \yg |E_{\eta}|^2 = \int_{\R^{\d+\k}}\chi_M \yg |E_{\eta} - \nabla \hpNe|^2 + \int_{\R^{\d+\k}}\chi_M \yg |\nabla \hpNe|^2 \\+ 2 \int_{\R^{\d+\k}}\chi_M  \yg (E_{\eta} - \nabla \hpNe) \cdot \nabla \hpNe
\\ \geq 
 \int_{\R^{\d+\k}}\chi_M  \yg|\nabla \hpNe|^2 + 2 \int_{\R^{\d+\k}}\chi_M  \yg(E_{\eta} - \nabla \hpNe) \cdot \nabla \hpNe \\ =  \int_{\R^{\d+\k}}\chi_M \yg |\nabla \hpNe|^2 + 2 \int_{\R^{\d+\k}}  \hpNe \yg (E_{\eta} - \nabla \hpNe) \cdot \nab \chi_M,
\end{multline}
where we integrated by parts and used \eqref{diveg} to remove one of the terms. Letting $M\to \infty$, the last term tends to $0$ by finiteness of the right-hand side of \eqref{comparloc} and decay properties of $ \hpNe$ and its gradient, and we obtain the result.
\end{proof}

\section{Proof of the main results}\label{sec-main}
\subsection{Statement of two intermediate results}
Our main theorem is a consequence of two intermediate results. 
 
\subsubsection{Empirical fields without interaction.} \label{sec:empwihtoutint}
The first one is a large deviation principle for the tagged empirical field, when the points are distributed according to a reference measure on $(\Rd)^N$ where there is no interaction.

For any $N \geq 1$ we define $\QNbeta$ as
\begin{equation} \label{def:QNbeta}
d\QNbeta(\XN) := \prod_{i=1}^N \frac{\exp\left( - N^{1-\frac{\s}{\d}} \beta \zeta(x_i) \right) dx_i}{\int_{\Rd} \exp\left( - N^{1-\frac{\s}{\d}} \beta \zeta(x) \right) dx} , 
\end{equation}
and we let $\bQpN$ be the push-forward of $\QNbeta$ by the map $\bEmp_N$, as defined in \eqref{def:bEmp}.
Let us point out that in view of \eqref{defzeta} and since $\meseq $ is a compactly supported probability measure, $\zeta$ is asymptotic to $\g + \frac{V}{2}-\cV$ at infinity, so by assumption \ref{H6} the integral in \eqref{def:QNbeta} converges.

We also introduce the following constant:
\begin{equation}
\label{defcomeg} \comeg := \log |\omega|-  |\Sigma|+1,
\end{equation}
where $\omega$ as in \eqref{defomega}.

We may now state the LDP associated to $\bQpN$ (the notation is as in Section \ref{sec-ppp}).
\begin{prop} \label{SanovbQN} 
For any $A \subset \probas_s(\Sigma \times \config)$, we have
\begin{multline} \label{EqSB}
- \inf_{\bPst\in\mathring{A} \cap \probas_{s,1}} \bERS[\bP|\Poisson^1] - \comeg \leq \liminf_{N \ti} \frac{1}{N} \log \bQpN (A) \\ \leq  \limsup_{N \ti} \frac{1}{N} \log \bQpN(A) \leq - \inf_{\bPst \in \bar{A}} \bERS[\bP|\Poisson^1] -\comeg.
\end{multline}
\end{prop}
\begin{proof}
The proof is given in Section \ref{sec:preuveSanovbQNoui}.
\end{proof}

 LDP's for empirical fields can be found in \cite{varadhansf}, \cite{follmersf}, the relative specific entropy is formalized in \cite{FollmerOrey} (for the non-interacting discrete case), \cite{Georgii1} (for the interacting discrete case) and \cite{georgiz} (for the interacting continuous case). In light of these results, Proposition \ref{SanovbQN} is not surprising, but there are some technical differences. In our case, the reference measure $\QN$ is not the restriction of a Poisson point process to a hypercube but only approximates a Bernoulli point process on some domain - which is not a hypercube - with the possibility of some points falling outside. Moreover we want to study large deviations for tagged point processes (let us emphasize that our \textit{tags} are not the same as the \textit{marks} in \cite{georgiz}) which requires an additional argument. These adaptations, leading to the proof of Proposition \ref{SanovbQN}, occupy Section \ref{preuvesanov}.

\subsubsection{Quasi-continuity.}
The second technical result is the \textit{quasi-continuity}\footnote{We use the terminology of \cite{bodguionnet}, which was a source of inspiration.} of the interaction, in the following sense.

For any integer $N$, and $\delta > 0$ and any $\bP \in \probas_s(\Sigma \times \config)$, let us define
\begin{equation} \label{def:TNdelta}
T_N(\bP, \delta) := \{ \XN \in (\Rd)^N, \WN(\XN, \meseq) \leq \bttW(\bP, \meseq) + \delta \}.
\end{equation}

\begin{prop} \label{quasicontinuite}
Let $\bPst\in\probas_{s, 1}(\Sigma \times \config)$. For  all $\delta_1, \delta_2>0$  we have 
\begin{equation}
\label{quasicontinu} 
\liminf_{N \ti} \frac{1}{N} \log \QNbeta \left( \{\bEmp_N \in B(\bPst, \delta_1)\} \cap T_N(\bP, \delta_2) \right) \ge  - \bERS[\bP|\Poisson^1]  -\comeg.
\end{equation}
\end{prop}
Let us compare with Proposition \ref{SanovbQN}. The first inequality of \eqref{EqSB} implies (taking $A = B(\bPst, \delta_1)$ and using the definition of $\bQpN$ as the push-forward of $\QNbeta$ by $\bEmp_N$) that
$$
\liminf_{N \ti} \frac{1}{N} \log \QNbeta \left( \{\bEmp_N \in B(\bPst, \delta_1)\} \right) \ge  - \bERS[\bP|\Poisson^1]  -\comeg.
$$ 
To obtain Proposition \ref{quasicontinuite}, which is the hard part of the proof, we thus need to show  that the event  $T_N(\bP, \delta_2)$ has  enough volume in phase-space near $\bPst$. This relies on taking arbitrary configurations in the ball $B(\bPst, \delta_1)$ and arguing that a large enough fraction of them (in the sense of volume at logarithmic scale) can be modified in order to have well-controlled energy.

The proof of Proposition \ref{quasicontinuite} occupies Section \ref{sec:construction}.

\subsection{The large deviation principle: proof of Theorem \ref{TheoLDP}}
With Propositions \ref{SanovbQN} and  \ref{quasicontinuite} at hand, the proof of Theorem \ref{TheoLDP} is standard. 
\begin{proof}
\textit{Step 1.} \textbf{Weak large deviation principle.} \\
Let $\bPst$ be in $\probas_{s}(\Sigma \times \config)$. Using the splitting formulas \eqref{splitlog}, \eqref{splits},  and the definition of $\bPgot_{N,\beta}$ we have, for any $\delta_1 > 0$ 
\begin{multline*}
\bPgot_{N,\beta} \left( B(\bPst, \delta_1) \right) \\ = \frac{1}{\KNbeta}  \int_{(\Rd)^N \cap \{\bEmp_N(\XN) \in B(\bP, \delta_1)\}} \exp\left(- \frac{\beta}{2} N \WN(\XN, \meseq)\right) \\ \exp\left(-\beta N^{1-\frac{\s}{\d}} \sum_{i=1}^N \zeta(x_i) \right) \prod_{i=1}^N dx_i,
\end{multline*} where $\KNbeta$ is the  new partition function defined by 
\begin{equation} \label{defKNBETA}
K_{N, \beta}= \begin{cases} Z_{N,\beta} e^{ \frac{\beta}{2} \( N^2 \I_V(\meseq) - \frac{N\log N}{\d}\) }&\quad \text{in the cases \LogU, \,\LogD}\\ Z_{N,\beta} e^{\frac{\beta}{2} N^{2 -\frac{\s}{\d}} \I_V(\meseq)} &\quad \text{in the case \Riesz}.\end{cases}\end{equation}
In view of our definition of $\QNbeta$ as in \eqref{def:QNbeta} we may write
\begin{equation*}
\exp\left(-\beta N^{1-\frac{\s}{\d}} \sum_{i=1}^N \zeta(x_i) \right) \prod_{i=1}^N dx_i = \left(\int_{\Rd} \exp\left( - N^{1-\frac{\s}{\d}} \beta \zeta(x) \right) dx\right)^N d\QNbeta(\XN).
\end{equation*}
Also, we have of course for any $\delta_2 > 0$
$$
(\Rd)^N \cap \{\bEmp_N(\XN) \in B(\bP, \delta_1)\} \supset (\Rd)^N \cap \{\bEmp_N(\XN) \in B(\bP, \delta_1)\} \cap T_N(\bP, \delta_2),
$$
where $T_N(\bP, \delta_2)$ is as in \eqref{def:TNdelta}. Using the definition of $T_N(\bP, \delta_2)$, we get
\begin{multline*}
\bPgot_{N,\beta} \left( B(\bPst, \delta_1) \right)  \\
\geq \frac{1}{K_{N,\beta}}  \left(\int_{\Rd} \exp\left( - N^{1-\frac{\s}{\d}} \beta \zeta(x) \right) dx\right)^N  \exp\left(-\frac{\beta}{2} (\bttW(\bPst, \meseq) + \delta_2)\right) \\ \times
  \QNbeta \left( \{\bEmp_N \in B(\bPst, \delta_1)\} \cap T_N(\bP, \delta_2) \right).
\end{multline*}
Using Proposition \ref{quasicontinuite}, we get
\begin{multline} \label{ps1ounon}
\frac{1}{N} \log \bPgot_{N,\beta} \left( B(\bPst, \delta_1) \right) + \frac{1}{N} \log K_{N,\beta} \\ \geq    \log \left(\int_{\Rd} \exp\left( - N^{1-\frac{\s}{\d}} \beta \zeta(x) \right) dx\right) - \bERS[\bP|\Poisson^1]  -  \comeg  - \frac{\beta}{2} \left(\bttW(\bPst, \meseq) + \delta_2\right) + o_N(1).
\end{multline}
On the other hand, by the monotone convergence theorem, since $\zeta = 0$ on $\omega$ and is $> 0$ outside $\omega$, we have
\begin{equation} \label{asymptotiquezetaNbeta}
\log \left(\int_{\Rd} \exp\left( - N^{1-\frac{\s}{\d}} \beta \zeta(x) \right) dx\right) = \log |\omega| + o_N(1).
\end{equation}
Thus, sending $N \to \infty, \delta_1 \to 0$ and $\delta_2 \to 0$ in \eqref{ps1ounon} and using \eqref{defcomeg} we obtain
\begin{equation} \label{weakLDP1}
\liminf_{\delta \to 0} \liminf_{N \ti} \left(\frac{1}{N} \log \bPgot_{N,\beta} \left( B(\bPst, \delta) \right) + \frac{1}{N} \log K_{N,\beta}\right) \geq - \fbarbeta(\bP) + |\Sigma| - 1.
\end{equation}
Conversely, Proposition \ref{prop:LowerBoundenergies} together with Proposition \ref{SanovbQN} and \eqref{asymptotiquezetaNbeta} imply that
\begin{equation} \label{weakLDP2}
\limsup_{\delta \t0} \lim_{N \ti} \frac{1}{N} \log \bPgot_{N,\beta} \left( B(\bPst, \delta) \right)  + \frac{1}{N} \log K_{N,\beta}  \leq - \fbarbeta(\bP) + |\Sigma| - 1.
\end{equation}

\textit{Step 2.} \textbf{Strong large deviation principle.}
Exponential tightness of $\bPgot_{N, \beta}$ is an easy consequence of the fact that the total number of points in $N^{1/{\d}}\Sigma$ is bounded by $N$. It allows one to pass from a weak LDP as in \eqref{weakLDP1}, \eqref{weakLDP2} to a strong large deviation inequality: for any $A \subset \probas(\Sigma \times \config)$, we get 
\begin{multline} \label{preLDP} |\Sigma| - 1 - \inf_{\bPst \in \mathring{A}} \fbarbeta(\bP) \\ \leq \liminf_{N \ti} \frac{1}{N} \left( \log \bPgot_{N,\beta}(A) + \log K_{N,\beta}\right)
\leq \limsup_{N \ti} \frac{1}{N} \left( \log \bPgot_{N,\beta}(A) + \log K_{N,\beta}\right) \\ \leq |\Sigma| - 1 - \inf_{\bPst \in \bar{A} } - \fbarbeta(\bP) .
\end{multline}

\textit{Step 3.} \textbf{Conclusion.}
Applying this relation to the whole space, we deduce the thermodynamic limit, namely
\begin{equation} \label{limitethermo}
\lim_{N \ti} \frac{1}{N} \log K_{N,\beta} = - \inf_{\probas_{s,1}(\Sigma \times \config)} \fbarbeta + |\Sigma| - 1.
\end{equation}
Inserting \eqref{limitethermo} into \eqref{preLDP} yields Theorem \ref{TheoLDP}, up to the fact that $\fbarbeta$ is indeed a \textit{good} rate function, which follows from Lemmas \ref{lem:ERS} and \ref{lem:LSCW}.
\end{proof}

\begin{remark}
One can observe that the result does not depend on the exact relation between $\meseq$ and the $\zeta$ associated to $V$, but rather holds for all sufficiently regular probability densities $\meseq$ and nonnegative functions $\zeta$ growing fast enough at infinity, as soon as the zero set of $\zeta$ contains the support of $\meseq$. In particular $K_{N,\beta}$ can be seen as a function of the couple $(\meseq, \zeta)$ rather than of $V$, and the expansion proven below still holds. \end{remark} 
 
\subsection{Expansion of the partition function: proof of Corollary \ref{corothermo}}
\begin{proof}
The first part of the corollary (the expansions \eqref{expansionlog}, \eqref{expansionriesz}) follows immediately from \eqref{limitethermo} and \eqref{defKNBETA}. For \eqref{logz} we use the following scaling results:
\begin{lem} \label{effetscalingent}
For any $m > 0$ and $\Pst$ in $\probas_s(\config)$ of intensity $m$ we define $\sigma_m P$ as the push-forward of $P$ by the map
$$
\C \mapsto m^{1/\d} \C, 
$$
so that $\sigma_m$ maps bijectively processes of intensity $m$ onto processes of intensity $1$.

We have
\begin{align}
\label{scalingenergie} &\ttW(\Pst, m)=  m \left( \ttW(\sigma_m P, 1) - \frac{1}{\d} \log m \right), \text{ in the cases  \LogU,\, \LogD} \\
\label{scalingentropie} & \ERS[\Pst|\Poisson^1] = m\, \ERS[(\sigma_m \Pst)|\Poisson^1] +1 - m + m\log m.
\end{align}
\end{lem}
\begin{proof}
The scaling relation \eqref{scalingenergie} was proven in \cite{sandier20151d}, \cite{sandier20152d} at the level of electric fields, and the extension to point processes is immediate. We now prove \eqref{scalingentropie}.

First, by a simple change of variables we observe that
\begin{equation} \label{sent1}
\ERS[\sigma_m P | \Poisson^1] = \frac{1}{m}\, \ERS[P | \Poisson^m].
\end{equation}Next, we may write
$$
\ERS[P | \Poisson^m] = \ERS[P | \Poisson^1] + \lim_{R \ti} \frac{1}{R^{\d}} \Esp_{P_{\carr_R}} \left[ \log \frac{d\Poisson^1_{\carr_R}}{d\Poisson^m_{\carr_R}} \right].
$$
The Radon-Nikodym derivative $\frac{d\Poisson^1_{\carr_R}}{d\Poisson^m_{\carr_R}}$ only depends  on the number of points in $\carr_R$ and we have (since $P$ has intensity $m$)
$$
\Esp_{P_{\carr_R}} \left[ \log \frac{d\Poisson^1_{\carr_R}}{d\Poisson^m_{\carr_R}} \right] = \Esp_{P_{\carr_R}} \left[ (m-1) R^{\d} - \Numb_R \log m \right] = (m-1) R^{\d} - mR^{\d} \log m.
$$
Finally, we get
\begin{equation} \label{sent2}
\ERS[P | \Poisson^m] = \ERS[P | \Poisson^1] - m \log m + (m-1), 
\end{equation}
and \eqref{scalingentropie} follows from combining \eqref{sent1} and \eqref{sent2}.
\end{proof}

Now, in the logarithmic cases,  if $\bPst\in\probas_s(\Sigma \times \config)$ is such that $\bttW(\bPst, \meseq)$ is finite, we can apply $\sigma_{\meseq(x)}$ to $\bPst^x$ (for a.e. $x$ in $\Sigma$) and we get from Lemma \ref{effetscalingent}
\begin{multline*}
\fbarbeta(\bPst) := \int_{\Sigma} \meseq(x) \left(\frac{\beta}{2} \ttW\left(\sigma_{\meseq(x)} \bPst^x, 1\right) + \ERS[\sigma_{\meseq(x)} \bPst^x | \Poisson^1] \right) dx \\ + \int_{\Sigma} \left(1 - \frac{\beta}{2\d}\right) \meseq(x) \log \meseq(x)
+ |\Sigma| - 1.
\end{multline*}
In particular, we observe that
\begin{equation} \label{sfreener}
\inf_{\bPst \in \probas_s(\Sigma \times \config)} \fbarbeta(\bPst)  - (|\Sigma| - 1) = \inf_{P \in \probas_{s,1}(\config)} \fbeta(P) + \left(1 - \frac{\beta}{2\d}\right)  \int_{\Sigma} \meseq(x) \log \meseq(x),
\end{equation}
and we obtain \eqref{logz} by combining the expansion \eqref{expansionlog} with the scaling relation \eqref{sfreener}.
\end{proof}

 \subsection{Application to the sine-beta process: proof of Corollary \ref{bsinebmin}}
\label{proofSinGin}
Our large deviation principle deals with the \textit{tagged empirical fields}, defined by averaging the point configurations over translations in $\Sigma$. It is natural to ask about the behavior of the Gibbsian point process itself, that is the push-forward of $\PNbeta$ by the map 
\begin{equation} \label{noaveraging}
\XN \mapsto \sum_{i=1}^N \delta_{x'_i}.
\end{equation}

In the general case, the mere existence of limit points for the law of this quantity is unknown. In the \LogU \ case, however, it was proven in \cite{vv} (see also \cite{MR2484278}) that the limit in law exists and is given by the so-called \textit{sine-$\beta$ process} (named by analogy with the previously known \textit{sine-kernel}, or \textit{Dyson sine process}, which corresponds to $\beta = 2$).
\def\semicirc{\mu^{\mathrm{Wig}}}

For $x \in (-2,2)$ and any $\beta > 0$, we denote by $\sineb(x)$ the Sine-$\beta$ process of \cite{vv} rescaled to have intensity  $\semicirc(x) = \frac{1}{2\pi}\sqrt{4-x^2}$. For any $\beta > 0$, we define $\PgNbeta$ as the push-forward of $\PNbeta$ (\LogU \ case, $V(x) = x^2$) by the map \eqref{noaveraging}. The following is a rephrasing\footnote{The convergence $\then$ is proven in \cite{vv} “in law with respect to vague topology for the counting measure of the point process”. As explained in Remark \ref{memetopologie}, this topology coincides with the one used here.} of \cite[Theorem1]{vv}:
\begin{equation}\label{cvgVV}
\text{For any $x \in (-2,2)$, for any $\beta > 0$, we have } \theta_{Nx} \cdot \PgNbeta \then \sineb(x).
\end{equation}

We may now prove the variational property of $\sineb(x)$ as stated in Corollary \ref{bsinebmin}.
\begin{proof}
Let $F : [-2,2] \times \config \rightarrow \C$ be a bounded continuous function. By definition we have
\begin{multline*}
\Esp_{\bPgot_{N,\beta}}\left[\int F(x,\mathcal{C})\, dP(x,\mc{C})\right] = \Esp_{\PgNbeta} \left[ \int_{[-2,2]} F(x, \theta_{Nx} \cdot \mathcal{C}) dx \right] \\ = \int_{[-2,2]} \Esp_{\PgNbeta} \left[  F(x, \theta_{Nx} \cdot \mathcal{C}) \right]dx
\end{multline*}
From \eqref{cvgVV} we know that the sequence of functions $\{x \mapsto \Esp_{\PgNbeta} \left[  F(x, \theta_{Nx} \cdot \mathcal{C}) \right]\}_N$ converges almost-everywhere on $[-2,2]$ to $x \mapsto \Esp_{\sineb(x)} [ F(x, \mathcal{C})]$ as $N \ti$. Since $F$ is bounded the dominated convergence theorem implies that
\begin{equation} 
\lim_{N \ti} \Esp_{\bPgot_{N,\beta}}\left[\int F(x,\mathcal{C})\, dP(x,\mc{C})\right] = \int_{[-2,2]}\Esp_{\sineb(x)} [F(x, \mathcal{C})] dx = \Esp_{\bsineb}[F(x, \mathcal{C})],
\end{equation}
where $\bsineb$ is the associated tagged point process. Since this is true for any bounded continuous function on $[-2,2] \times \config$, the sequence of tagged point processes $\{\bPgot_{N,\beta}\}_N$ converges to $\bsineb$, and then the large deviation principle implies that $\bsineb$ is be a minimizer of $\bar{\mathcal{F}}^{\semicirc}_{\beta}$.

The fact that the point process $\sineb$ itself minimizes $\fbeta$ (as defined in the statement) among stationary point processes of intensity $1$ follows by scaling.
\end{proof}

\section{Screening and regularization}\label{sec:screening}
In this section, we introduce the main ingredients of the proof of Proposition \ref{quasicontinuite}. We define two operations on point configurations (say, in a given hypercube $\carr_R$) which can be  roughly described this way: 
\begin{enumerate}
\item The screening procedure $\Phiscr$ takes “good” (called here \textit{screenable}) configurations and replaces them by “better” configurations which are neutral (the number of points matches the integral of the equilibrium measure times $N$) and for which there is a compatible electric field supported in $\carr_R$ whose energy is well-controlled. 
\item The regularization procedure $\Phireg$ takes a configuration and separates pairs of points which are “too close” from each other.
\end{enumerate}
If a “bad” (non-screenable) configuration is encountered, it replaces it by a “standard” configuration (at the cost of a loss of information).

\subsection{The screening result}
When we get to the Section \ref{sec:construction} (which contains the proof of Proposition \ref{quasicontinuite}), we will want to construct point configurations by elementary blocks (hyperrectangles) and compute their energy additively in these blocks. One of the technical tricks borrowed from the original works above is that this may be done by gluing together electric fields whose normal components agree on the boundaries. More precisely, assume that space is partitioned into a family of hyperrectangles $K$. We would like to  construct a vector field $\EK$ in each $K$ such that 
\begin{equation}
\label{eqcellfond}
\left\lbrace \begin{array}{ll}
 -\div (\yg \EK) 
= \c \left( \mc{C}_K - \meseq'\drd\right) & \text{in} \ K\times \R^\k \\
 \EK \cdot \vec{\nu} = 0  & \text{on} \  \partial  (K\times \R^\k) \end{array}\right.
\end{equation} (where $\vec{\nu}$ is the outer unit normal to $K$)
for some discrete set of points $\mc{C}_K \subset K $, and with 
\begin{equation*}
\int_{K\times \R^\k} \yg |\EK_\eta|^2 
\end{equation*}
 well-controlled.
Integrating the relation \eqref{eqcellfond}, we see that for this equation to be solvable, we need the following compatibility condition 
\begin{equation}\label{compa}
\int_K d\mc{C} = \int_{K} d\meseq',
\end{equation}
in particular the partition must be made so that $\int_K d \meseq'$ are integers.

When solving \eqref{eqcellfond}, we could take $\EK$ to be a gradient, but we do not require it.
Once the relations \eqref{eqcellfond} are satisfied  on each $K$, we may paste together the vector fields $\EK$ into a unique vector field $\Etot$, and the discrete sets of points $\mc{C}_K$ into a configuration $\Ctot$. By \eqref{compa} the cardinality of $\Ctot$ will be equal to $\int_{\R^{\d}} d\meseq'$, which is exactly $N$.

We will thus have obtained a configuration of $N$ points, whose energy we may try to evaluate. The important fact is that the enforcement of the boundary condition $\EK\cdot\vec{ \nu}=0$ on each boundary ensures that 
\begin{equation}\label{deve}
-\div (\yg \Etot)= \c \left(\Ctot - \meseq' \drd\right) \quad \text{in } \R^{\d+\k}
\end{equation}
holds globally. Indeed, a vector field which is discontinuous across an interface has a distributional divergence concentrated on the interface equal to the jump of the normal derivative, but here by construction the normal components coincide hence there is no divergence created across these interfaces. 

Even if the $\EK$'s were gradients, the global $\Etot$ is in general no longer a gradient.  This does not matter however, since the energy of the \textit{local electric field} $\nab \hpN$ generated by the finite configuration $\Ctot$ (and the background $\muv'$) is always smaller, as stated in Lemma \ref{minilocale}.  We thus have 
\begin{equation*}
\int_{\R^{\d+\k}} \yg |\nab \hpNe|^2 \le \sum_{K} \int_{K \times \R^{\k}} \yg |\EK_\eta|^2,
\end{equation*}
and the energy, which originally was a sum of pairwise interactions,  has indeed become additive over the cells. 
%Thus, in order to obtain an upper bound on the energy $\WN(\Ctot, \meseq)$ (as in  \eqref{defWN} - here $\Ctot$ is a finite configuration with $N$ points), we may evaluate the energy of the vector fields $\EK$ constructed in each $K$.

Starting from a given point configuration in a cell $K$, assume that there exists $E$ compatible with the configuration and the background in $K$ (i.e. the first equation of \eqref{eqcellfond} is satisfied but not necessarily the second one) whose energy is not too large. Using the \textit{screening} tool developed in \cite{sandier20152d, sandier20151d, rougerie2016higher, petrache2014next}, we may modify the configuration in a thin layer near $\partial K$, and modify the  vector field $E$ a little bit as well, so that \eqref{eqcellfond} is satisfied, and so that the energy has not been increased too much. This is the object of Section \ref{sec:descriscri}. 

If, on the other hand, there exists no such $E$ of reasonable energy in the cell $K$, we will completely discard the configuration and replace it by an artificial one for which there exists a screened electric field whose energy is well controlled. This will be described in Section \ref{sec:artificial}.

\subsubsection{Compatible and screened electric fields.} \label{sec:notationscreen}
Let us introduce some notation. Let $K$ be a compact subset of $\R^{\d}$ with piecewise $C^1$ boundary (typically, $K$ will be an hyperrectangle or the support $\Sigma$), $\C$ be a finite point configuration in $K$, $\mu$ be a positive measure in $L^{\infty}(K)$, and $E$ be a vector field in $\Lploc(K \times \R^\k, \Rd \times \R^\k)$. 
\begin{itemize}
\item We say that \textit{$E$ is compatible with $(\C, m)$ in $K$} and we write
\begin{equation*}
E \in \Comp(\C, \mu, K)
\end{equation*}
 provided 
\begin{equation}
\label{def:compatible}
-\div (\yg E)= \c \left(\mc{C} - \mu \drd\right) \quad \text{in } K \times \R^\k.
\end{equation}
\item We say that \textit{$E$ is compatible with $(\C, \mu)$ and screened in $K$} and we write
\begin{equation*}
E \in \Screen(\C, \mu, K)
\end{equation*}
when \eqref{def:compatible} holds and moreover
\begin{equation} \label{lescreening}
 E \cdot \vec{\nu} = 0   \quad \text{on } \partial  (K\times \R^\k), 
\end{equation}
where $\vec{\nu}$ is the outer unit normal vector. 
\end{itemize}
If $E$ is in $\Comp(\C, \mu, K)$, for any $\eta \in (0,1)$ we define the truncation of $E$ as
\begin{equation}
\label{def:Eeta}
E_{\eta}(X) := E - \sum_{p \in \C \cap K} \nabla \f_{\eta} (X - (p,0)).
\end{equation}

\subsubsection{Screening.} \label{sec:descriscri}
We now rephrase the screening result from \cite{petrache2014next}, with two modifications: we have to deal with a non-constant background measure, and we also need a volume estimate. 
Roughly speaking, with the notation as in Section \ref{sec:notationscreen}, we start from $(\C, \mu)$ in some hyperrectangle $K$, and $E \in \Comp(\C, \mu, K)$, and we slightly modify $\C, E$ into $\Cscr, \Escr$ such that in the end $\Escr$ belongs to  $\Screen(\Cscr, \mu, K)$. We call  this \textit{screening}  because when \eqref{lescreening} holds the configuration has no influence outside the cell (as far as upper bound estimates on the energy are concerned, see below). The configuration\footnote{To be accurate, we do not really need the original configuration  to be defined in the whole $K$, but only in a subcube $\carr_R\subset K$, the configuration is then completed by hand.} and the field are modified in a thin layer near the boundary, and remain unchanged in an interior set denoted $\mathrm{Old}$.  

A new feature, compared to the previous screening results, is that we need the positions of the new points (those added “by hand” in the layer near the boundary), denoted by $\mathrm{New}$, to be flexible enough to create a positive volume of associated point configurations in phase space. We will let the points move in small balls, which creates some volume without altering the energy estimates.

\begin{prop}[Screening] \label{Lemscreening}
Let $\um, \om > 0$ be fixed. 

There exists $R_0, \eta_0 > 0$ depending only on $\d$ and $\om$, and a constant $C$ depending on $\d, \s, \um, \om$ such that the following holds. 

Let $R > 0, M > 1$, and $\epsilon \in (0,1)$ be such that the following inequalities are satisfied 
\begin{equation}\label{condR}
 R> \max \left( \frac{R_0}{\epsilon^{2}},  \frac{C R_0 M}{\epsilon^3}\right),  \\  
R>   \begin{cases}
C R_0 M^{1/2} \epsilon^{-\d-3/2} & \text{if } \k =0 \\ 
 \max\left(C R_0M^{1/(1-\gamma)} \epsilon^{\frac{-1-2\d+\gamma}{1-\gamma}} , R_0\epsilon^{\frac{2\gamma}{1-\gamma} }   \right)  & \text{if} \ \k=1
\end{cases}.
\end{equation} 

Let $\C$ be a point configuration in $\carr_R$, let $K$ be a hyperrectangle such that $\carr_R \subset K \subset \carr_{R+\epsilon R}$, and let $\mu \in C^{0, \kappa}(K)$ satisfying
\begin{equation*}
\um \leq \mu \leq \om \text{ on } K, \quad \int_{K} \mu \text{ is an integer}.
\end{equation*}

Let $\eta$ be in $(0, \eta_0)$, let $E$ be in $\Comp(\C, \mu, \carr_R)$, and let $E_{\eta}$ be as in \eqref{def:Eeta}. Assume that 
\begin{equation} \label{definiM}
\frac{1}{R^\d}\int_{\carr_R \times [-R, R]^\k}\yg |E_{\eta}|^2  \leq M
\end{equation}
and in the case $\k=1$ assume that
\begin{equation} \label{decrvert}
\frac{1}{\epsilon^4 R^\d} \int_{\carr_R \times \left( \R \backslash (-\hal \epsilon^2 R, \hal \epsilon^2 R)\right)} \yg |E|^2 \leq 1.
\end{equation}

Then there exists a (measurable) family $\Phiscr_{\epsilon, \eta, R}(\mc{C}, \mu)$ of point configurations in $K$ and a partition of $K$ as $\Old \sqcup \New$ with
\begin{equation} \label{IntinOld}
\Int_{\epsilon} := \{ x \in \carr_R, \dist(x, \partial  \carr_R)\} \geq 2 \epsilon R\} \subset \Old
\end{equation} 
such that for any $\Cscr$ in $\Phiscr_{\epsilon, \eta, R}(\mc{C}, \mu)$ we have 
\begin{align}
\label{coincideonold} & \text{The configurations $\mc{C}$ and $\Cscr$ coincide on $\Old$}, \\
\label{loindubord} & \min_{x \in \Cscr} \dist(x, \partial  K) \geq \eta_0, \quad \text{ and } \min_{x \in \Cscr \cap \New,\ y \in \Cscr} |x-y| \geq \eta_0,\\
\label{touchepasproche} & \sum_{x_i \neq x_j \in \Cscr,  |x_i - x_j| \leq 2  \eta} \g(x_i-x_j)  = \sum_{x_i \neq x_j \in \mc{C}, |x_i - x_j| \leq 2 \eta} \g(x_i-x_j).
\end{align}
For any $\Cscr$, there exists\footnote{In particular the configuration $\Cscr$ has exactly $\int_{K} \mu$ points in $K$.} a vector field $\Escr \in \Screen(\Cscr, \mu, K)$ such that
\begin{equation} \label{erreurEcrantage} 
\int_{K \times \R^\k} \yg |\Escr_{\eta}|^2 \leq  I+II+III
\end{equation}
with, for some constant $C$ depending only on $\s,\d,\um,\om$
\begin{align*}
I&=\left( \int_{\carr_R \times [-R, R]^\k} \yg |E_{\eta}|^2   \right)  + C \g(\eta) M  \ep R^\d
\\ II &=  CR^{\d+3-\gamma} \|\mu\|_{C^{0, \kappa}(K)}^2 \\
III &= \sqrt{ I \cdot II}.
\end{align*}
\end{prop}

\begin{proof} 
The statement is based on  a re-writing of \cite[Proposition 6.1.]{petrache2014next} provided by an examination of its proof.

First, let us assume that $\mu \equiv 1$, in that case we may apply directly \cite[Proposition 6.1]{petrache2014next} and we will sketch the proof here.
 
The first step uses a mean-value argument and the assumption \eqref{definiM}  in order to find a good boundary, that is the boundary of an hypercube $\Old$ included in $\carr_R$ and containing $\Int_{\epsilon}$, such that $\int_{\partial \Old \times \R^\k} \yg |E_\eta|^2$ is not too large. In the case where $\k=1$, we need to do the same “vertically”, using assumption \eqref{decrvert}, in order to find a good height $z$ such that $\int_{\Old \times \{-z,z\}} \yg |E_\eta|^2$ is controlled. 

The configuration $\mc{C}$ and the field $E$ are kept unchanged inside $\Old$, hence \eqref{coincideonold} is satisfied. Next, we tile $\New := K \backslash \Old$ by small hypercubes of sidelength $O(1)$ (and uniformly bounded below)  and place one point near the center of  each of these hypercubes. By construction, the new points are well-separated: the distances between two points in $\New$ or between a point in $\New$ and a point in $\Old$, or between a point in $\New$ and the boundary of $K$, are all bounded below by $2 \eta_0$ (which ensures \eqref{loindubord}). In particular if $\eta < \eta_0$ no new $2\eta$-close pair has been created and \eqref{touchepasproche} holds. Figure \ref{figureecrantage} is meant to illustrate the construction.

\begin{figure}[h!]
\begin{center}
\includegraphics[scale=0.9]{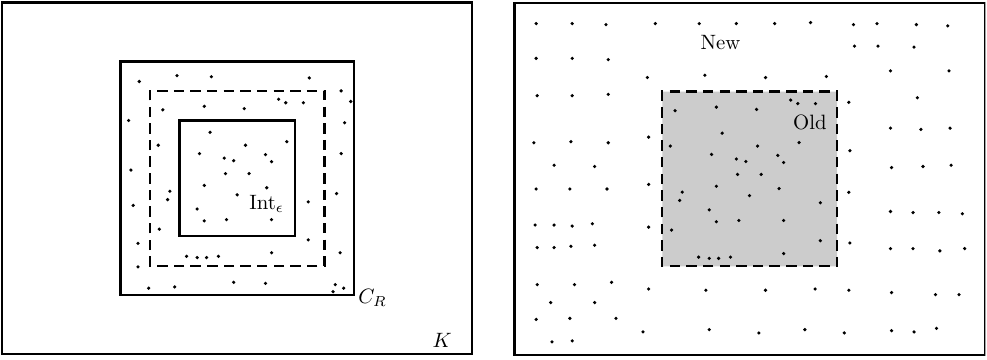}
\caption{The original configuration (on the left) and the screened configuration (on the right). The dashed line corresponds to the good boundary. Proportions are distorted and $\mathrm{Int}_\ep$ really contains most of the set $\carr_R$, which itself contains most of $K$.}
\end{center}
\label{figureecrantage}
\end{figure} 
 
Then, we construct a global electric field on $\New$ by pasting together vector fields defined on each small hypercube of the tiling. In order to ensure the global compatibility, i.e. no extra divergence to be added, we need the  normal components  to agree with that of $E$ on $\partial  \Old$ and to be $0$ on $\partial  K$.  The resulting vector field is thus in $\Screen(\Cscr, \mu, K)$. It is proven that we can construct it such that  (with a constant $C$ depending only on $\d,\s$)
\begin{equation}\label{casm1}
\int_{K\times \R^\k} \yg |\Escr_{\eta}|^2 \leq \left(\int_{\carr_R \times \R^k} \yg |E_{\eta}|^2\right) + C \g(\eta) M \epsilon R^{\d}.
\end{equation}

There remains to handle  the fact that the background measure $\mu$, in our setting, is not constant equal to $1$ but may vary between $\um$ and $\om$. The case where $\mu \equiv m$ is a constant function (with $m \in (\um, \om)$) follows by a simple scaling argument, provided that the constant $C$ in \eqref{condR} is chosen large enough (depending on $\um, \om, \d, \s$). The bounds below on the distances (as e.g. in \eqref{loindubord}) become worse when $\om$ is large, which explains why $\eta_0$ has to be chosen small enough depending on $\om$.  If $\mu$ is not constant, we approximate it by its average and use its H\"older continuity to control the error, as described in Lemma \ref{constvsmoy}. This is what creates the error term II.  

Finally, for each small hypercube in the tiling of $K \backslash \Old$, we may move the point placed therein by a small distance $\frac{\eta_0}{4}$ without affecting the conclusions, as explained e.g. in \cite[Remark 6.7]{petrache2014next}. This creates a family of configurations (with their associated screened vector field).
\end{proof}

Let us now estimate how this procedure changes the volume of a set of configurations in phase-space. Let $R, \epsilon, \mu$ be as in Proposition \ref{Lemscreening}. Let $\Int_\ep$ be as defined in \eqref{IntinOld} and let $\mathrm{Ext}_\ep:=  {\carr_R } \backslash \Int_\ep$.

We define $\v_0$ as the volume of a ball of radius $\frac{\eta_0}{4}$ in $\Rd$ 
\begin{equation}
\label{def:v0} \v_0(\d, \om) := \left|B\left(0,\frac{\eta_0}{4}\right)\right|.
\end{equation}
\begin{lem} 
Let $\Nn, \Nnint$ be two integers, let $N_K := \int_K \mu(x) dx$. Let $\v$ be such that $\v \leq \v_0$ and 
\begin{equation}\label{condnext}
N_K  - \Nnint \le \frac{|\mathrm{Ext}_\ep|}{2 \v}.
\end{equation}
Let $\calA$ be a measurable family of point configurations in $\config(\carr_R)$, and let assume that each element of $\calA$ has $\Nn$ points in $\carr_R$  and $\Nnint$ points in $\Int_{\epsilon}$. Let us also assume that $(\mc{C}, \mu)$ satisfies the conditions of Proposition \ref{Lemscreening} for all $\C$ in $\calA$. Then we have
 \begin{multline}\label{PreservVolume1}
\log \Leb^{\tens N_K} \left(\underset{\C \in \calA}{\bigcup} \Phiscr_{\epsilon, \eta , R}(\C, \mu) \right) \ge \log \Leb^{\tens \Nn}(\calA) \\+\log \left( \left(N_K-\Nnint \right)! \left( \frac{\v}{|\mathrm{Ext}_\ep|}\right)^{N_K-\Nnint}  \right) + (N_K-\Nn) \log |\mathrm{Ext}_\ep|.
\end{multline}
\end{lem}
\begin{proof}
We may partition $\calA$ as $\calA = \cup_{n_{\New}=0}^{N_K}  (\calA|n_{\New})$ according to the number of points $n_{\New}$ that are created in $\New$ (we denote by $\calA|n_{\New}$ the subset of $\calA$ consisting of configurations for which $n_{\New}$ points are created). By construction, the number of points in $\Old$ (points which remain unchanged) is thus given by $N_K-n_{\New}$. Moreover, again by construction, we have $\Int_{\ep} \subset \Old$, hence
\begin{equation} \label{nintKmu}
\Nnint \leq N_{K}-n_{\New}.
\end{equation}

Thus for each configuration in $\calA|n_{\New}$, at least $N_K-n_{\New}$ points are left untouched while the other ones, all belonging to $\Ext_{\ep}$, are deleted and replaced by $n_{\New}$ points, each one being allowed to move in a small ball of radius $\frac{\eta_0}{4}$, which creates a volume $\v_0 \geq \v$. We may thus write
 \begin{equation*}
  \Leb^{\tens N_K}\left(\underset{\C \in (\calA|n_{\New}) }{\bigcup} \Phiscr_{\epsilon, \eta,R}(\C, \mu) \right) \ge \frac{ \Leb^{\tens \Nn}(\calA|n_{\New}) (n_{\New})! \v^{n_\New}  } {|\mathrm{Ext}_\ep|^{\Nn-(N_K -n_{\New})}  }
  \end{equation*}
One may check that the map
 $$ 
 x\mapsto x! \left(\frac{\v}{|\Ext_\ep|}\right)^x
 $$
  is decreasing as long as $x \le \frac{|\mathrm{Ext}_\ep|}{2\v}$. Since \eqref{nintKmu} holds, we deduce that
\begin{multline*}
 \Leb^{\tens N_K} \left(\underset{\mc{C} \in \calA|n_{\New}}{\bigcup} \Phiscr_{\epsilon, \eta,R}(\mc{C}, \mu) \right) \\ \ge \Leb^{\tens \Nn}(\calA|n_{\New}) |\mathrm{Ext}_\ep|^{N_K-\Nn}  \left((N_K-\Nnint )! \left(\frac{\v}{|\mathrm{Ext}_\ep|} \right)\right)^{N_K - \Nnint}.
\end{multline*}
Summing over $n_\New$ and taking the $\log$ yields the result.
\end{proof}

%Finally, we state a result concerning the distance between a configuration before and after screening
%\begin{lem} \label{lem:distancescreening}
%\begin{equation}
%\label{distancescreening} \limsup_{\epsilon \to 0} \sup_{(\C, \mu)} \sup_{\Cscr \in \Phiscr_{\epsilon, \eta, R}(\mc{C}, \mu)} \dconfig(\C, \Cscr) = 0 
%\end{equation}
%\end{lem}

\subsection{Screenability} \label{secScreenable}
We introduce the notion of \textit{screenability}, whose main purpose is to ensure that the conditions \eqref{condR}
 of  Proposition \ref{Lemscreening} hold. We then prove the upper semi-continuity of the screening procedure. From now on, when we write $\lim_{a \to a_0, b\to b_0 }$ we mean $\lim_{a\to a_0} \lim_{b\to b_0}$, and the same inductively in case of more than two consecutive limits. This means in particular that each parameter is chosen depending on those on its left.

\subsubsection{Screenability.} \label{sec:additionalconditions} First, we introduce a class of electric fields which are compatible with a given $(\C, \mu)$ in some hypercube, with additional conditions on the energy.

\begin{defi}[Screenability] \label{defscreenable}
Let $0 < \um, \om < + \infty$ be fixed, let $R_0, \eta_0$ (depending only on $\d, \om$) and $C$ (depending on $\d, \s, \om, \um$) be as in Proposition \ref{Lemscreening}.

Let $R > 0, M > 1$ and $\ep \in (0,1)$ be such that the following inequalities are satisfied
\begin{equation} \label{condiScr} 
 R> \max \left( \frac{R_0}{\epsilon^{2}},  \frac{C R_0 M}{\epsilon^3}\right),    
R>   \begin{cases}
C R_0 M^{1/2}\epsilon^{-\d-3/2} & \text{if } \k =0 \\
\max \left(C R_0 M^{1/(1-\gamma)} \epsilon^{\frac{-1-2\d+\gamma}{1-\gamma}} , R_0\epsilon^{\frac{2\gamma}{1-\gamma} }   \right)  & \text{if } \k=1
\end{cases}.
\end{equation}
Let $\C$ be a point configuration in $\carr_R$, let $\mu \in L^{\infty}(\carr_R)$, and let $\eta \in (0, \eta_0)$ be fixed.

We define $\Op^{M,\epsilon}(\C, \mu)$ as the set of vector fields $E$ in $\Comp(\C, \mu, \carr_R)$ such that the following conditions hold:
\begin{itemize}
\item We have \begin{equation} \label{controle-M} 
\frac{1}{R^{\d}} \int_{\carr_{R} \times [-R, R]^\k}\yg |E_{\eta}|^2 < M.
\end{equation}
\item In the case $\k=1$,
\begin{equation} \label{control-e}
\frac{1}{\epsilon^4 R^{\d}} \int_{\carr_{R} \times \left( \R \backslash (-\hal \epsilon^2 R, \hal \epsilon^2 R)\right)} \yg |E|^2 < 1.
\end{equation}
\item We also require that 
\begin{equation}  \label{rajouteborneNscreen}
\Numb_R(\C) < MR^{\d}.
\end{equation}
Although \eqref{rajouteborneNscreen} is really an assumption on $\C$, it can easily be translated into an assumption on $E$ (for $E \in \Comp(\C, \mu, R)$) and it will be convenient for us to consider it as such.
\end{itemize}

Finally, we say that $(\C, \mu)$ is \textit{screenable} and we write $(\C, \mu) \in \Sp$  if $\Op^{M,\epsilon}(\mc{C}, \mu)$ is not empty.
\end{defi}

\subsubsection{Best screenable energy.}
\begin{defi}
For any $(\mc{C}, \mu)$ as above we define $\H$ to be the \textit{best screenable energy}
\begin{equation}
\label{BSE} \H(\mc{C}, \mu) : = \inf \left\lbrace \frac{1}{R^{\d}} \int_{\carr_{R} \times [-R, R]^\k}\yg |E_{\eta}|^2 ,\  \ E \in \Op^{M,\epsilon}(\mc{C}, \mu) \right\rbrace \wedge M,
\end{equation}
where $\wedge$ denotes the $\min$.
\end{defi}
The $\inf$ in \eqref{BSE} can be $+ \infty$ e.g. if the set is empty\footnote{Although it is not needed here, we may observe that if the set is not empty, then the infimum is attained.}, but by definition we always have the upper bound $\H(\C, \mu) \leq M$.

The following lemma will be used for studying the regularity of $\H$.
\begin{lem} \label{LemCorrections} Let $R >0$ and let $\mc{C}, \mc{C'}$ be two configurations in $\carr_R$ and $\mu, \mu'$ be in $L^{\infty}(\carr_R)$. Let $\tilde{E}$ be the electric field
\begin{equation}
\label{defEtilde}
\tilde{E} := \nab  \g* \( (\mc{C}-\mc{C'}) -( \mu-\mu') \drd\).
\end{equation}
Then for any $\eta > 0$, we have, as $(\mc{C'}, \mu')$ converges to $(\mc{C}, \mu)$ in $\config(\carr_R) \times L^{\infty}(\carr_R)$
 \begin{equation*}
\lim_{(\mc{C'}, \mu') \to (\mc{C}, \mu)} \int_{\carr_R \times \R^\k } \yg |\tilde{E}_{\eta}|^2 = 0.
\end{equation*}
\end{lem}
\begin{proof}
Let $\g_\eta :=\min(\g, \g(\eta))$ as in \eqref{def:truncation}. We have  by definition
\begin{equation*}
\tilde E_\eta = \nab g_\eta* (\mc{C}-\mc{C'}) -\nab  g*( (\mu-\mu')\drd).
\end{equation*}
To prove the lemma it is enough to show that, letting 
\begin{equation*}
H_1 := \g_\eta *  (\mc{C}-\mc{C'}), \quad H_2:= \g* ((\mu-\mu') \drd),
\end{equation*}
both quantities 
\begin{equation*}
\int_{\carr_R \times [-R, R]^\k} \yg |\nab H_1|^2, \quad \int_{\carr_R  \times [-R, R]^\k } \yg |\nab H_2|^2
\end{equation*}
tend to $0$ as $(\mc{C'}, \mu')$ converges to $(\mc{C}, \mu)$. 

The number of points in $\carr_R$ is locally constant on $\config(\carr_R)$, so we may assume that the signed measure $\mc{C}-\mc{C'}$ has total mass $0$ in $\carr_R$. Therefore $H_1$ (resp. $\nabla H_1$) decays like $|X|^{-\s-1}$ (resp. like $|X|^{-\s-2}$) as $|X| \ti$ in $\R^{\d + \k}$ (as noticed e.g. in the proof of Lemma \ref{minilocale}).  We may thus write, using an integration by parts
\begin{equation*}
\int_{\R^{\d+\k}}\yg |\nab H_1|^2 =\iint \g_\eta(x-y) (\mc{C}-\mc{C'}) (x)  (\mc{C}-\mc{C'}) (y)
\end{equation*}
 and  the desired result for $H_1$ follows by continuity of $\g_\eta$. 
 
 Concerning $H_2$, since the kernel $\g$ is always integrable with respect to the Lebesgue measure on $\Rd$ we have 
 \begin{equation}\label{dech2}
 |H_2| \le C \|\mu-\mu'\|_{L^\infty(\carr_R)},
 \end{equation} 
 where the constant $C$ depends on $R$. Moreover, by assumption we have
 \begin{equation*}
 -\div (\yg \nab H_2)= \c (\mu-\mu') \drd.
 \end{equation*}

 Let then $\chi$ be a smooth compactly supported positive function equal to $1$ in $\carr_R\times [-R, R]^\k$,  and such that $|\nab \chi|\le 1$. Integrating by parts, we have 
 \begin{multline*}\int_{\R^{\d+\k}} \chi^2 \yg |\nab H_2|^2 = - \int_{\R^{\d+\k} }\chi^2 \div (\yg \nab H_2) H_2 - 2\int_{\R^{\d+\k}} \chi \nab \chi \cdot \nab H_2 \yg H_2\\
 \le \c\left| \int_{\R^{\d+\k}} \chi^2  H_2(\mu-\mu') \drd \right| + \int_{\R^{\d+\k}} \chi|\nab \chi|\yg |H_2||\nab H_2|.
 \end{multline*}
 Using \eqref{dech2}  and the Cauchy-Schwarz inequality, we may thus write 
  \begin{multline*}
  \int_{\R^{\d+\k}} \chi^2 \yg |\nab H_2|^2\le   C \|\mu-\mu'\|_{L^\infty}\left(\int_{\R^{\d+\k}} \chi^2 \yg |\nab H_2|^2 \right)^\hal \left( \int_{\R^{\d+\k}} |\nab \chi|^2 \yg\right)^{\hal} \\
  +C\|\mu-\mu'\|_{L^\infty}^2,
  \end{multline*}
 therefore, for some constant $C$ depending on $R$, we have
 $$\int_{\carr_R  \times [-R,R]^\k } \yg |\nab H_2|^2 \leq \int_{\R^{\d+\k}} \chi^2 \yg |\nab H_2|^2\le  C(\|\mu-\mu'\|_{L^\infty}^2+ \|\mu-\mu'\|_{L^\infty}^4 ),$$
 which completes the proof.
 \end{proof}

We deduce the following:
\begin{lem} \label{lem:ouverture} The set $\Sp$ is open in $\config(\carr_R) \times L^{\infty}(\carr_R)$. 
\end{lem}
\begin{proof}
If $(\C, \mu) \in \Sp$ we may find $E$ in $\Op^{M,\epsilon}(\mc{C}, \mu)$.  If $(\C', \mu')$ is close enough to $(\C, \mu)$, by adding the field 
$\tilde{E}$ (as in \eqref{defEtilde}) to $E$, we get a vector field in $\Comp(\C', \mu')$ whose energy is bounded by that of $E$ plus an error term going to zero as $(\C',\mu')$ goes to $(\C, \mu)$, in particular the vector field $\tilde{E}$ satisfies the (open) conditions \eqref{controle-M}, \eqref{control-e} for $(\C', \mu')$ close enough to $(\C, \mu)$. 
\end{proof}

We also easily deduce:
\begin{lem} \label{uscH}  The function $\H$ is upper semi-continuous on $\config(\carr_R) \times L^{\infty}(\carr_{R})$.
\end{lem}
\begin{proof}
Let $(\C, \mu)$ be in $\Sp$ (otherwise there is nothing to prove), we may consider $E$ in $\Op^{M,\epsilon}(\mc{C}, \mu)$ such that 
\begin{equation*}
\H(\C, \mu)  = \frac{1}{R^{\d}} \int_{\carr_{R} \times [- R, R]^\k}\yg |E_{\eta}|^2.
\end{equation*}The infimum is indeed achieved as mentioned in the previous footnote.
Using Lemma \ref{LemCorrections} we see that if $(\C', \mu')$ is close enough to $(\C, \mu)$, by adding the field 
$\tilde{E}$ (as in \eqref{defEtilde}) to $E$, we get a vector field in $\Comp(\C', \mu')$ whose energy is bounded by that of $E$ plus an error term going to zero as $(\C',\mu')$ goes to $(\C, \mu)$. It implies 
\begin{equation*}
\limsup_{(\C',\mu') \to (\C, \mu)} \H(\C', \mu') \leq \H(\C, \mu), 
\end{equation*}
 which proves the upper semi-continuity.
\end{proof}

The next lemma shows that tagged point processes $\bPst$ of finite energy have good properties: most configurations under $\bPst$ are \textit{screenable} and the average \textit{best screenable energy} is close to the energy of $\bar{P}$. These controls are then extended to point processes in small balls $B(\bPst, \nu)$ around $\bPst$. 

For a given couple $(x, \mc{C})$  we may sometimes abuse notation and write “$(x, \mc{C}) \in \Sp$” instead of $(\mc{C}, \mu_V(x)) \in \Sp$.
\begin{lem} \label{svtecrantable}
Let $\bPst$ be a tagged point process in $\probas_{s}(\Sigma \times \config)$ such that $\bttW(\bPst, \meseq)$ is finite. We have
\begin{itemize}
\item For $\eta > 0$ small enough and any $\epsilon > 0$, 
\begin{equation} \label{presquetousscreenable}
\liminf_{M \to \infty, R\to \infty, \nu\to 0} \inf_{\bar{Q}  \in B(\bar{P}, \nu)}  \bar{Q} (\Sp) = 1.
\end{equation}
%\cm{dans quel ordre les lim sont prises? ca doit etre R puis M puis nu}
% where  $M, R \ti$ in such a way that the conditions \eqref{condiScr} are satisfied.
\item For any $\epsilon > 0$, 
\begin{equation} \label{HvsW}
\limsup_{\eta \to 0, M \to \infty, R\to\infty, \nu \to 0} \sup_{\bar{Q} \in B(\bar{P},\nu)}  \Esp_{\bQ} \left[ \H(\mc{C}, \meseq(x) ) \right] - \cds \g(\eta)
\leq \bttW(\bPst, \meseq).
\end{equation}
\end{itemize}
\end{lem}
\begin{proof}

\textbf{Screenability.} \\
From Lemma \ref{lem:concordance}, since $\bttW(\bP, \meseq)$ is finite, we know that there exists a stationary $\bPelec$ which is compatible with $(\bP, \meseq)$ and such that
\begin{equation*}
\btW(\bPelec, \meseq) = \bttW(\bPst, \meseq).
\end{equation*}
Using stationarity we see that for any $R > 0$ 
\begin{equation} \label{WsommesurR}
\btW(\bPelec, \meseq) = \Esp_{\bPelec} \left[\frac{1}{R^{\d}} \int_{\carr_R \times \R^\k} \yg |E_{\eta}|^2\right] -\cds \g(\eta).
\end{equation}
Markov's inequality implies that for any $M, R >0$ and $\eta \in (0,1)$ we have
\begin{equation} \label{MarkovPelec}
\bPelec\left(\frac{1}{R^{\d}} \int_{\carr_R \times \R^\k} \yg |E_{\eta}|^2 \geq M\right) \leq \frac{\bttW(\bar{P}, \meseq)+ \cds \g(\eta)}{M}.
\end{equation}
 In the case $\k = 1$, we have $\bPelec$-a.s., for any $\epsilon > 0$ fixed
\begin{equation*}
\lim_{R \ti} \int_{\carr_1 \times (\R \backslash (-\epsilon^2R, \epsilon^2R))^\k} \yg |E_{\eta}|^2  = 0,
\end{equation*}
which in turn implies (using again the stationarity) that for any $\epsilon > 0$
\begin{equation}\label{TCDdec} 
\lim_{R \ti} \bPelec \left(\frac{1}{R^{\d}} \int_{\carr_R \times (\R \backslash (-\epsilon^2R, \epsilon^2R))^\k} \yg |E_{\eta}|^2 \geq 1 \right) = 0.
\end{equation}
Lemma \ref{LemmeDiscr} implies that $\Esp_{\bPst} \left[\Numb_R^2\right] \leq C R^{2\d}$ with a constant $C$ depending only on $\bPst$, therefore we have
\begin{equation} \label{MarkovNbPoints}
\bPst\left (\Numb_R \geq MR^{\d}\right) \leq \frac{C}{M^2}.
\end{equation}

Combining \eqref{MarkovPelec}, \eqref{TCDdec} and \eqref{MarkovNbPoints} we obtain
\begin{equation} \label{preopen}
\lim_{M \to \infty} \lim_{R\to \infty} \bar{P}(\Sp) = 1.
\end{equation}
Since $\Sp$ is open, we may extend \eqref{preopen} to $\bQ \in B(\bP, \nu)$ as $\nu \to 0$, and we obtain \eqref{presquetousscreenable}.

\textbf{Best screenable energy.} \\
Since $\H$ is bounded by $M$, we have
\begin{equation*}
\Esp_{\bP} [ \H ] \leq \Esp_{\bPelec} \left[ \frac{1}{R^\d} \int_{\carr_R \times [-R, R]^{\k}} \yg |E_{\eta}|^2 \right] + M \left(1 - \bP(\Sp)\right).
\end{equation*}
The monotonicity as in Lemma \ref{prodecr}  yields
\begin{equation*}
\Esp_{\bP} [ \H ] - \cds \g(\eta) \leq  \bttW(\bP, \meseq) + M \left(1 - \bP(\Sp)\right)  + o_{\eta}(1),
\end{equation*}
with a $o_{\eta}$ depending only on $\d, \s, \|\meseq\|_{L^{\infty}}$. 

We claim that 
\begin{equation}
\label{leMnestpasgrave}
\limsup_{M \ti, R \ti}  M \left(1 - \bP(\Sp)\right) = 0.
\end{equation}
We may decompose  $1 - \bP(\Sp)$ as
\begin{multline*}
1 - \bP(\Sp) \leq \bPst\left (\Numb_R \geq MR^{\d}\right) + \bPelec \left(\frac{1}{R^{\d}} \int_{\carr_R \times (\R \backslash (-\epsilon^2R, \epsilon^2R))^\k} \yg |E_{\eta}|^2 \geq 1 \right) \\ + \bPelec\left(\frac{1}{R^{\d}} \int_{\carr_R \times \R^\k} \yg |E_{\eta}|^2 \geq M\right).
\end{multline*}
The first term in the right-hand side is $O\left(\frac{1}{M^2}\right)$ from \eqref{MarkovNbPoints}, and the second one tends to $0$ as $R \ti$ independently of $M$ (see \eqref{TCDdec}. To prove \eqref{leMnestpasgrave} it remains to prove that
$$
\limsup_{M \ti, R \ti} M \bPelec\left(\frac{1}{R^{\d}} \int_{\carr_R \times \R^\k} \yg |E_{\eta}|^2 \geq M\right) = 0.
$$
Using Markov's inequality we may bound the left-hand side by the tail of the expectation 
$$
\Esp_{\bPelec} \left[ \frac{1}{R^{\d}} \int_{\carr_R \times \R^\k} \yg |E_{\eta}|^2 \mathbf{1}_{\left\lbrace\frac{1}{R^{\d}} \int_{\carr_R \times \R^\k} \yg |E_{\eta}|^2 \geq M \right\rbrace} \right], 
$$
which tends to $0$ as $M \ti$ by dominated convergence. Thus \eqref{leMnestpasgrave} holds and, using \eqref{preopen}, we obtain 
\begin{equation*}  
\limsup_{M \to \infty, R \to \infty} \left(\Esp_{\bP} [ \H ] - \cds \g(\eta)\right) \leq \bttW(\bP, \meseq) +  o_{\eta}(1), 
\end{equation*}
and letting $\eta \to 0$ yields 
\begin{equation}
\label{prelim} 
\limsup_{\eta \to 0, M \to \infty, R \to \infty} \left(\Esp_{\bP} [ \H ] - \cds \g(\eta)\right) \\ \leq \bttW(\bP, \meseq). 
\end{equation}
We may again extend \eqref{prelim} to $\bQ \in B(\bP, \nu)$ as $\nu \to 0$ using the upper-semi continuity of $\H$ as stated in Lemma \ref{uscH}, and we obtain \eqref{HvsW}.
\end{proof}

\subsection{Regularization of point configurations}
The singularity of the interaction kernel has been dealt with by truncating the short distance interactions. For any configuration $\XN$,  the truncated energy converges  as $\eta\to 0$ to the next-order “renormalized” energy $\WN(\XN,\meseq)$ (see \eqref{defWN}), but in order to prove our LDP we need a control on the truncation error which is somehow uniform in $\XN$ (at least with high probability).

\subsubsection{The regularization procedure}
The purpose of the following lemma is to regularize a point configuration by spacing out the points that are too close to each other, while ensuring that the new configuration remains close to the original one. This operation generates a family of configurations $\Creg$ (“reg” as “regularized”) for which we control the contribution of the energy due to pairs of close points.

In this section, if $l > 0$ is fixed, for any $\vec{i} \in l \Z^d$, we define “the hypercube of center $\vec{i}$” as the closed hypercube of sidelength $l$ of center $\vec{i}$ (in other words, $\carr_{l}(\vec{i})$), and we identify it with its center. Let us recall that $\Numb_r(\vec{i})$ denotes the number of points in $\carr_{l}(\vec{i})$.

\begin{lem} \label{LemmeRegul} 
Let  $\tau \in (0,1)$, $R > 1$ be fixed, and let $K$ be an hyperrectangle such that 
$$\carr_R \subset K \subset \carr_{2R}.$$ There exists a measurable multivalued function $\Phireg_{\tau,R}$ which associates to any finite configuration $\C$ a subset of $\config$, such that 
\begin{enumerate} 
\item Any configuration $\Creg$ in $\Phireg_{\tau,R}(\mc{C})$ has the same number of points as $\mc{C}$ in $K$.
\item For any fixed $R > 0$, the distance to the original configuration goes to zero as $\tau \t0$, uniformly on $\C \in \config(K)$ and $\Creg \in \Phireg_{\tau,R}(\mc{C})$
\begin{equation} \label{distanceregular}
\lim_{\tau \to 0} \sup_{\C \in \config(K)} \sup_{\Creg \in \Phireg_{\tau,R}(\mc{C})}  \dconfig(\mc{C}, \Creg)  = 0.
\end{equation}
\item For any $\Creg \in \Phireg_{\tau,R}(\mc{C})$, for any $\eta \ge 8 \tau$, we have, for some constant $C$ depending only on $\d, \s$
\begin{multline} \label{controletroncatureregul}
\sum_{p \neq q \in \Creg, |p - q| \leq \eta} \g(p - q) \\ \le
C \g(\tau) \sum_{\vec{i} \in 6\tau \Zd} \left(\Numb_{12\tau}[\vec{i}]^2(\mc{C}) - 1\right)_+  + \sum_{p \neq q \in \mc{C} , \tau \le|p-q|\le 2\eta } \g(p-q).
\end{multline} 
\item For any integer $N_{K}$ and any family $\calA$ of configurations with $N_{K}$ points in $K$, we have, for some constant $C$ depending only on $\d$
\begin{multline} \label{VolumeReg}
\log \Leb^{\otimes N_K}\Big(\bigcup_{\mc{C} \in \calA} \Phireg_{\tau,R}(\mc{C})\Big)\\  \geq  \log \Leb^{\otimes N_K}\left(\calA\right) - C \int_{\mc{C} \in A} \sum_{\vec{i} \in 6 \tau \Zd}\Numb_{6\tau}[\vec{i}](\mc{C}) \log \Numb_{6\tau}[\vec{i}] (\mc{C}) \, d \Leb^{\otimes N_K}(\C).
\end{multline}
\end{enumerate}
\end{lem}

\begin{proof}
\textbf{Definition of the regularization procedure.}\\
For any $\tau > 0$ and $\mc{C} \in \config(K)$ we consider two categories of hypercubes in $6\tau \Z^d$ :
\begin{itemize}
\item $S_{\tau}(\mathcal{C})$ is the set of hypercubes $\vec{i} \in 6\tau \Z^d$ such that $\mathcal{C}$ has at most one point in $\vec{i}$ and no point in the adjacent hypercubes.
\item $T_{\tau}(\mc{C})$ is the set of the hypercubes that are not in $S_{\tau}(\mathcal{C})$ and that contain at least one point of $\mc{C}$.
\end{itemize}

We define $\varphi_{\tau}(\mc{C})$ to be the following configuration: 
\begin{itemize}
\item For any $\vec{i}$ in $S_{\tau}(\C)$, the configuration $\C \cap \vec{i}$ is left unchanged.
\item For any $\vec{i} \in T_{\tau}(\mc{C})$, the configuration $\mc{C} \cap \vec{i}$ is replaced by a well-separated configuration in a smaller hypercube. More precisely we consider the lattice $3 \Numb_{6\tau}[\vec{i}]^{-1/{\d}} \tau \Zd$  translated so that the origin coincides with the point $\vec{i} \in 6 \tau \Zd$ and place $\Numb_{6\tau}[\vec{i}]$ points on this lattice in such a way that they are all contained in the hypercube of sidelength $3 \tau$ and center $\vec{i}$ (a simple  volume argument shows that this is indeed possible, the precise way of arranging the points is not important - it is easy to see that one may do it in a measurable fashion). 
\end{itemize}
It defines a measurable function $\varphi_{\tau,R} : \config(K) \rightarrow \config(K)$ (illustrated in Figure \ref{figregul}). 
\begin{figure}[h!]
\begin{center}
\includegraphics[scale=0.8]{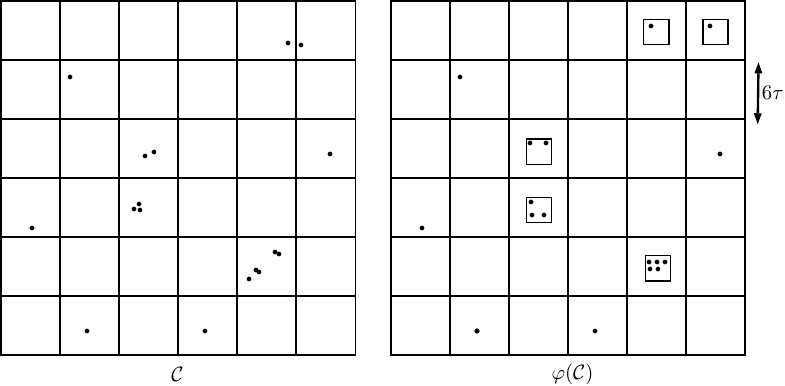}
\end{center}
\vspace{-0.5cm}
\caption{Effect of the regularization. On the right are shown the smaller hypercubes in which the new configurations are created for $\vec{i} \in T_{\tau}(\mathcal{C})$.}
\label{figregul}
\end{figure}

Then, we define $\Phireg_{\tau,R}(\mc{C})$ as the set of configurations that are obtained from $\C$ the following way:  we replace $\C$ by $\varphi_{\tau, R}(\C)$ and for any $\vec{i} \in T_{\tau}(\C)$  we allow the points (of the new configuration) to move arbitrarily and independently within a ball of radius $\Numb_{6\tau}[\vec{i}]^{-1/{\d}}\tau$. We claim that $\Phireg_{\tau,R}$ has the desired properties. 

\textbf{Number of points.} \\ 
By construction, the number of points is kept the same in each hypercube of the tiling, hence the global number of points is conserved.

\textbf{Distance estimate.}\\
By construction, for any $\Creg \in \Phireg_{\tau,R}(\mc{C})$ the configurations $\mc{C}$ and $\Creg$ have the same number of points in every hypercube of $6\tau \Zd$. It implies that every point of $\mc{C}$ is displaced by a distance $O(\tau)$ (depending only on the dimension). By the definition \eqref{dconfig} of the distance, it yields
%\footnote{In fact, the error depends only on $\tau$ and on the total number of points of $\mc{C}$ in $K$.} 
 \begin{equation*}
\sup_{\C \in \config(K)} \sup_{\Creg \in \Phireg_{\tau,R}(\mc{C})} \dconfig(\mc{C}, \Creg) \leq C \tau,
\end{equation*}
with $C$ depending only on $\d$.

\textbf{Truncation estimate.}\\
Let us distinguish three types of pairs of points $p, q \in \Creg$ which might satisfy $|p-q| \leq \eta$: 
\begin{enumerate}
\item The pairs of points $p, q$ belonging to some hypercube of $T_{\tau}(\mc{C})$.
\item The pairs of points $p, q$ belonging to two adjacent hypercubes of $T_{\tau}(\mc{C})$.
\item The pairs of points $p, q$ such that $|p-q| \leq \eta$ but neither of the two previous cases holds.
\end{enumerate}

To bound the contributions of the first type of pairs, we observe that in any hypercube $\vec{i} \in T_{\tau}(\mc{C})$ the sum of pairwise interactions is bounded above by
\begin{equation} \label{bornereseautau}
\sum_{p \neq s \in \Creg \cap \vec{i}} \g(p-q) \leq C\g(\tau)(\Numb_{6\tau}[\vec{i}]^2(\mc{C})-1)_+  .
\end{equation}
Indeed by construction in a hypercube $\vec{i} \in T_{\tau}(\mc{C})$ the points of $\Creg$ are separated by a distance at least $2\Numb_{6\tau}[\vec{i}]^{-1/{\d}}\tau$, and it is elementary (see e.g. \cite[(2.44), (2.45)]{serfatyZur} to check that their interaction energy is then bounded above by $C \g(\tau)Numb_{6\tau}[\vec{i}]^2$, with a constant depending only on $\d, \s$. In particular \eqref{bornereseautau} holds, because the left-hand side is obviously zero when $\Numb_{6\tau}[\vec{i}]=1$.

Since the distance between the points in $\Creg$ which belong to different hypercubes is bounded below by $6\tau$, 
we easily find that the total contribution of the second type of pairs is bounded by 
$$
C \g(\tau)\sum_{\vec{i} \in 6\tau \Z^d} \left(\Numb_{12\tau}[\vec{i}]^2(\mc{C})-1\right)_{+}.
$$

Finally the contribution of the third type of pairs is easily bounded by 
$$
\sum_{p, s \in \mc{C}, \tau \leq |p-q| \leq \eta+8\tau} \g(p-q),
$$
indeed any such two points were at distance at least $12\tau$ in $\mc{C}$ and their distance is reduced by at most $8\tau$ during the regularization (then one discusses according to the expression of $\g$).

\textbf{Volume loss estimate.}\\
 The pre-images by $\Phireg_{\tau,R}$ have a simple description\footnote{For a given configuration $\mc{C}$ in $K$ with $N_K$ points this defines a submanifold of $(\Rd)^{N_K}$ of co-dimension $\# S_{\tau}(\mc{C})$.}: we have $\Phireg_{\tau,R}(\mc{C}) = \Phireg_{\tau,R}(\mc{C}')$ only if  $\mc{C'} \cap \vec{i} = \mc{C} \cap \vec{i}$ for $ \vec{i} \in S_{\tau}(\mc{C})$ and $\Numb_{6\tau}[\vec{i}](\mc{C}') =\Numb_{6\tau}[\vec{i}](\mc{C})$ for $\vec{i} \in T_{\tau}(\mc{C})$ (these conditions are sufficient once symmetrized with respect to the roles of $\mc{C}$ and $\mc{C}'$). The volume of a pre-image is bounded by
\begin{equation}
\left(\sum_{\vec{i} \in T_{\tau}(\mc{C})} \Numb_{6\tau}[\vec{i}] \right)! \left(\tau^\d\right)^{\sum_{\vec{i} \in T_{\tau}(\mc{C})}\Numb_{6\tau}[\vec{i}]}
\end{equation}
whereas the volume of $\Phireg_{\tau,R}(\mc{C})$ (with respect to the Lebesgue measure) is given by 
\begin{equation*}
\left(\sum_{\vec{i} \in T_{\tau}(\mc{C})}\Numb_{6\tau}[\vec{i}] \right)! \prod_{\vec{i} \in T_{\tau}(\mc{C})} \left(\frac{\tau}{ \Numb_{6\tau}[\vec{i}]^{1/{\d}}}\right)^{\Numb_{6\tau}[\vec{i}]\d}
\end{equation*}
Taking the logarithm and integrating and integrating over $\C \in \calA$ yields \eqref{VolumeReg}.
\end{proof}

The truncation error after regularization is usually small for point processes of finite energy, as expressed in the next lemma.
\begin{lem} \label{lem:regumarchebien}
Let $\bP$ be in $\probas_s(\Sigma \times \config)$ such that $\bttW(\bP, \meseq)$ is finite.  Let $\eta, \tau$ be fixed, with $0<\tau<\eta^2/2<1$. We have
\begin{align} \label{HvsW2a}
& \sup_{x \in \Rd} \limsup_{ \eta \to 0, \tau \to 0}  \frac{\g(2\tau)}{\tau^d} \Esp_{\bar{P} } [(\Numb_{\tau}(x)^2-1)_+] = 0 \\ 
 \label{HvsW2b}
& \limsup_{ \eta \to 0, \tau \to 0, R \to \infty}  \frac{1}{R^{\d}}
\Esp_{\bar{P} } \left[ \sum_{p \neq q\in \mc{C} \cap \carr_R , \tau \le |p-q|\le \eta^2/2 } \g(p-q)  \right]=0.
\end{align}
\end{lem}
\begin{proof}
For  $\eta \in (0,1)$ and $0 < \tau<\eta^2/2$,  we have by Lemma \ref{discerrtronc} 
\begin{multline*}
\bttW_{\tau} (\bP, \meseq)- \bttW_\eta(\bP, \meseq)
\geq  C \frac{\g(2\tau) }{\tau^\d} \Esp_{\bar{P}} [ (\Numb_{\tau}^2 -1)_+ ] \\  
+ C \Esp_{\bar{P}} \left[  \sum_{p \neq q \in \mc{C} \cap \carr_1 , |p-q| \le \eta^2/2 } \g(p-q) \right]  - o_{\eta}(1)
\end{multline*}
for $C, o_{\eta}$ depending only on $\d, \s$. Letting $\eta \to 0$ and then $\tau \to 0$, it yields
\begin{align*}
& \limsup_{ \eta \to 0}\limsup_{\tau \to 0}  \frac{\g(2\tau)}{\tau^d} \Esp_{\bar{P} } [(\Numb_{\tau}^2-1)_+] = 0 
\\ 
& \limsup_{ \eta \to 0}\limsup_{\tau \to 0} 
\Esp_{\bar{P} } \left[ \sum_{p \neq q\in \mc{C} \cap \carr_1 , \tau \le |p-q|\le \eta^2/2 } \g(p-q)  \right]=0.
\end{align*}
By stationarity we may replace $\Numb_{\tau}$ by $\Numb_{\tau}(x)$  in the first lign, and we may write the sum over $\carr_1$ by an average over $\carr_R$ for any $R > 1$ in the second line. This yields \eqref{HvsW2a}, \eqref{HvsW2b}.
\end{proof}

\subsubsection{Effect on the energy}
We argue that the regularization procedure at scale $\tau$ has a negligible influence on the screened energy e.g. for configurations obtained by the screening procedure of Proposition \ref{Lemscreening}.

\begin{lem} \label{lemmeregpasgrave} 
Let $R > 0, M > 1, \epsilon > 0$ be fixed, satisfying the conditions \eqref{condiScr}, let $K$ be an hyperrectangle and let $(\C, \mu)$ satisfying the assumptions of Proposition \ref{Lemscreening}. Let $\eta < \hal \eta_0$ (with $\eta_0$ as in Proposition \ref{Lemscreening}), and let $0 < \tau < \hal \eta^2$. 

Let $\Phiscr_{\epsilon, \eta,R}(\mc{C}, \mu)$ be the set of configurations generated by the screening procedure, and for any $\Cscr$ in $\Phiscr_{\epsilon, \eta,R}(\mc{C}, \mu)$ let $\Escr$ be the corresponding screened vector field, finally let $\Phireg_{\tau,R}(\Cscr)$ be the set of configurations generated by the regularization procedure applied to $\Cscr$. 

Then for any $\Creg$ in $\Phireg_{\tau,R}(\Cscr)$ there exists a vector field $\Ereg \in \Screen(\Creg, \mu, K)$, such that
\begin{equation}
 \label{regulnetouchepasE} 
\limsup_{\tau \to 0} \sup_{\C \in \Sp}  \sup_{\Cscr \in \Phiscr_{\epsilon, \eta,R}(\mc{C}, \mu)} \sup_{\Creg \in  \Phireg_{\tau,R}(\Cscr)}  \int_{K \times \R^\k} \yg |\Ereg_{\eta}|^2 \\
\leq \int_{K \times \R^\k} \yg |\Escr_{\eta}|^2.
\end{equation}
\end{lem}
\begin{proof} 
Let $\g^{\textrm{Neu}}$ be the unique solution with  mean zero  to 
$$\left\lbrace\begin{array}{ll}
-\div(\yg \nabla \g^{\textrm{Neu}}) = \cds \left( \delta_0 - \frac{1}{|K|}\drd\right) & \  \text{in} \ K \times \R^\k\\
 \nabla \g^{\textrm{Neu}} \cdot \vec{\nu}=0 & \ \text{on} \ \partial  K \times \R^\k ,\end{array}\right.
$$
and let $\g^{\textrm{Neu}}_{\eta}$ be the truncated kernel at scale $\eta$ as above. For any $\Cscr$ in $\Phiscr_{\epsilon, \eta,R}(\mc{C}, \mu)$ and any $\Creg$ in $\Phireg_{\tau,R}(\Cscr)$ let us consider the vector field $\tilde{E}$ generated by the difference $\Creg - \Cscr$ with Neumann boundary conditions on $\partial  K$
$$
\tilde{E}(x) := \int \nabla \g^{\textrm{Neu}}(x-y) (d\Creg - d\Cscr)(y).
$$
Since the regularization procedure preserves the number of points, it is clear that, letting $\Ereg := \Escr + \tilde{E}$, the vector field $\Ereg$ is in $\Screen(\Creg, \mu, K)$. To bound its energy, we integrate by parts, and we claim that
$$
\iint \g_{\eta}^{\textrm{Neu}}(x-y) (d\Creg - d\Cscr)(x) (d\Creg - d\Cscr)(y) \leq C R^\d \tau,
$$
with a constant $C$ depending on $\d, \eta, \om$ (where $\om$ is a bound on the equilibrium density) but not on $\C, \Creg, \Cscr$. 

Indeed, by construction the number of points of $\Cscr$ in $K$ is bounded by $C R^\d$ with $C$ depending only on $\om$. Moreover, again by construction, there is no point of $\Cscr$ or $\Creg$ closer than some constant $\eta_0 >0$ (depending only on $\d, \om$) from the boundary $\partial  K$. The kernel $g_\eta^{\textrm{Neu}} $ is continuous at any point $x$ such that $\dist(x, \partial K) \geq \eta_0$, uniformly with respect to $x$ and to $\eta < \hal  \eta_0$. Finally, still by construction there is the same number of points in $\Creg$ and $\Cscr$, and the minimal connection distance between the points of $\Cscr$ and $\Creg$ is bounded by $C R^{\d} \tau$. We may then write
\begin{equation} \label{regulexplicite}
\iint \g_{\eta}^{\textrm{Neu}}(x-y) (d\Creg - d\Cscr)(x) (d\Creg - d\Cscr)(y)  \leq C R^\d \tau,
\end{equation}
with $C$ depending only on  $ \d, \om, \eta$, and not on $\C,\Creg$ or $\Cscr$.
\end{proof}

\subsection{Artificial configurations.} \label{sec:artificial}
Finally, we state a result concerning the construction of families of “artificial” configurations whose energy is well controlled. This will be used to replace non-screenable configurations.
 
\begin{lem} \label{constructreseau} 
Let $0 < \um \leq \om$ be fixed. There exists $\eta_0>0$ depending only on $\d, \om$ such that the following holds.

Let $R > 0$, let $K$ be a hyperrectangle with sidelengths in $[R,2R]$, and let $\mu$ be in $C^{0,\kappa}(K)$ such that $\um \leq \mu \leq \om$. We assume that 
\begin{equation*}
N_K := \int_{K} \mu \text{ is an integer.}
\end{equation*}

Then there exists a family $\Phigen(K, \mu)$ of configurations with $N_K$ points in $K$ such that for any $\Cgen$ in $\Phigen(K, \mu)$, the following holds:
\item \begin{equation}
\label{loindubordartificiel} \min_{x \in \Cgen} \dist(x, \partial K) \geq \eta_0, \quad \min_{x \neq x' \in \Cgen} |x-x'| \geq \eta_0.
\end{equation}
Moreover, there exists $\Egen$ in $\Screen(\Cgen, \mu, K)$ such that for any $0 < \eta < \eta_0$,
\begin{equation} \label{contEnerconstr2}
\int_{K\times \R^\k} \yg |\Egene|^2 - \c N_K \g(\eta)  \leq  C N_K +
C R^{\d+3-\gamma}\|\mu\|_{C^{0, \kappa}(K)}^2
\end{equation}
with a constant $C$ depending only on $\d, \s, \om,\um$.

There exists some constant $\v_1 > 0$ depending only on $\d, \om, \um$ such that the volume of $\Phigen(K, \mu)$ is bounded below by
\begin{equation}
\label{contVolRes} 
\Leb^{\tens N_K}\left(\Phigen(K, \mu)\right) \geq N_K! \v_1^{N_K}.
\end{equation}

\end{lem}
\begin{proof}
We postpone the proof of Lemma \ref{constructreseau} to Section \ref{sec:constructionreseau}.
\end{proof}

\subsection{Conclusion} \label{sectionupsilon} We may now combine the previous ingredients to accomplish the program stated at the beginning of the section. 
%We assume that the space $\R^{\d}$ (or some subset of it) has been partitioned into hyperrectangles $K$ on which $\int_K\mu$ is an integer.  
% We provide a procedure to assign to each screenable point configuration a screened one (with an associated electric field), and an artificial one to unscreenable configurations.  
% Again, since random point configurations occupy a certain volume in phase space, and since the artificial configurations as well as the new points in the screening “freeze" parts of the original configurations, we need to let the ``frozen" points  move in small balls and bound from below the volume in phase space that is thus created.
\begin{defi} \label{defi:upsilon}
Let $R, M, \epsilon$ satisfying \eqref{condiScr}, let $\eta < \hal \eta_0$ and let $0 < \tau < \hal \eta^2$. For $\C, \mu, K$ we define a family $\Phimod(\mc{C}, \mu, K)$ (depending on all the other parameters $\eta, \epsilon, M, R, \tau$) of point configurations which are contained in $K$ and have $N_K$ points, in  the following way :
\begin{enumerate}
\item If $(\mc{C}, \mu)$ is \textit{screenable}: 
\begin{enumerate}
\item We consider $E \in \Op^{M, \epsilon}(\mc{C}, \mu)$ such that
\begin{equation} \label{bienbonchamp}
\frac{1}{R^{\d}} \int_{\carr_{R} \times [-R,R]^k}\yg |E_{\eta}|^2 = \H(\mc{C},\mu).
\end{equation}
\item We let $\Phiscr_{\epsilon, \eta,R}(\mc{C}, \mu)$ be the family of point configurations in $K$ obtained by applying Proposition \ref{Lemscreening}. 
\item We let $\Phimod(\mc{C}, \mu)$ be the image by $\Phireg_{\tau,R}$ of the family $\Phiscr_{\epsilon, \eta,R}(\mc{C}, \mu)$.
\end{enumerate}
 
\item If $(\mc{C}, \mu)$ is \textit{not screenable}, we let $\Phimod(\mc{C}, \mu, K)$ be the family $\Phigen(K,\mu)$ defined in Lemma \ref{constructreseau}.
\end{enumerate}
In both cases, we end up with a family of point configuration in $K$ and associated screened electric fields whose energy is well-controlled.
\end{defi}

We may compare the volume of a certain set of configurations in $\carr_R$ with the volume of the resulting configurations after applying $\Phimod$. We distinguish between the cases of a set of \textit{screenable} configurations and a set of \textit{non-screenable} configurations.
\begin{lem} \label{volumeoperation}
Let $\calA$ be a family of point configurations in $\config(\carr_R)$ such that each configuration of $\calA$ has $\Nn$ points in $\carr_R$  and $\Nnint$ points in $\Int_{\epsilon}$. Let $\v \leq \v_0$, as in \eqref{def:v0}, and assume, as in \eqref{condnext},
\begin{equation}\label{condnextb}
N_K-\Nnint \le \frac{|\mathrm{Ext}_\ep|}{2 \v}.
\end{equation} 
\begin{enumerate}
\item If for all $\mc{C} \in \calA$, $(\mc{C}, \mu)$ is in $\Sp$ and \eqref{condnext} holds then there exists $C > 0$ depending only on  $\d, \um$ such that
\begin{multline} \label{estimvolumecasecrant}
\log \Leb^{\tens N_K}\left(\underset{\mc{C} \in A}{\bigcup} \Phimod(\mc{C}, \mu) \right) \ge \log \Leb^{\tens \Nn}(A) \\+\log \left((N_K-\Nnint)!\left( \frac{\v}{|\mathrm{Ext}_\ep|}\right)^{N_K-\Nnint}  \right) + (N_K-\Nn) \log |\mathrm{Ext}_\ep| \\ - C \int_{\mc{C} \in A} \sum_{\vec{i} \in 6 \tau \Zd} \Numb_{6\tau} [\vec{i}] \log \Numb_{6\tau} [\vec{i}] d\Leb^{\otimes \Nn}(\C).
\end{multline}
\item If for all $\mc{C} \in \calA$, $(\mc{C}, \mu)$ is not in $\Sp$  then there exists $\v_1 > 0$ depending only on $\d, \om, \um$ such that
\begin{equation}\label{estimvolumecasnonecrant}
\log \Leb^{\tens N_K}\left(\underset{\mc{C} \in \calA }{\bigcup} \Phimod(\mc{C}, \mu, K)\right) \geq \log \left( N_K! \v_1^{N_K} \right).
\end{equation}
\end{enumerate}
\end{lem}
\begin{proof}
The bound \eqref{estimvolumecasecrant} follows from combining \eqref{PreservVolume1} with \eqref{VolumeReg} whereas \eqref{estimvolumecasnonecrant} follows from \eqref{contVolRes}.
\end{proof}
Let us observe that we have 
\begin{equation*}
\log \Leb^{\tens \Nn}(\calA) \leq \log R^{\d \Nn}, 
\end{equation*}
hence in particular we may re-write \eqref{estimvolumecasnonecrant} as
\begin{equation} \label{estimvolumecasnonecrant2}
\log \Leb^{\tens N_K}\left(\underset{\mc{C} \in \calA }{\bigcup} \Phimod(\mc{C}, \mu, K)\right) \geq \log \Leb^{\tens \Nn}(\calA)  + \log \left( N_K! \v_1^{N_K}  R^{-\d \Nn} \right).
\end{equation}

\section{Construction of configurations}
\label{sec:construction}
This section is devoted to the proof of  Proposition \ref{quasicontinuite}, by exhibiting a set of configurations satisfying the conclusions with a large enough asymptotic (logarithmic) volume.

To do so, we first partition (some subset of) $\R^{\d}$ into hyperrectangles $K$ such that $\int_K \meseq'$ is an integer. Each hyperrectangle $K$ will contain a translate of $\carr_R$ such that $|K|-|\carr_R|$ is small (with respect to the total volume) and each hypercube will receive a point configuration.  Since we want the global configurations to approximate (after averaging over translations) a given tagged point process $\bar{P}$, we will draw the point configuration in each hypercube jointly at random according to a Poisson point process, and standard large deviations results imply that enough of the averages ressemble $\bar{P}$. These configurations drawn “abstractly” at random are then modified by the screening-then-regularizing procedure of the previous section. We eventually obtain a global configuration with $N$ points whose energy can be computed additively with respect to the hyperrectangles. 

At each step we need to check that the transformations imposed to the configurations do not alter much their phase-space volume, their energy, and keep them close to the given tagged process $\bar{P}$.

One of the additional technical difficulties is that the density of the equilibrium measure $\meseq$ is in general not bounded from below (see the assumptions \ref{H5}) and that the support $\Sigma$ cannot be exactly tiled by hyperrectangles. Thus, we first remove a thin layer near the boundary and near the zero set of $\meseq$, where some artificial point configurations will be placed at the end, with negligible contributions to the energy and to the volume. This will require to adapt the corresponding argument from \cite{petrache2014next}, since we are working under the more general Assumption \ref{H5}.
 
In this section, $\bPst$ denotes a stationary tagged point process such that $\bttW(\bPst, \meseq)$ is finite, otherwise Proposition \ref{quasicontinuite} reduces to Proposition \ref{SanovbQN}.

\subsection{Subdividing the domain} 
In what follows, $\um$ is some fixed number in $(0,1)$, which will be sent to $0$ at the end of the construction. Until the end of Section \ref{sec:construction}, $\cc$ and $\cC$ will denote positive constants depending only on $\meseq$.
\subsubsection{Interior and boundary.}\label{sec611}
 We start the construction by dividing the domain between a neighborhood of $\Gamma$ (see \ref{H5}), where the density might not be bounded below and which must be treated “by hand”, and a large interior. We recall that $\Sigma$ is the support of the equilibrium measure $\meseq$, that for $N \geq 1$ we let $\Sigma' := N^{1/\d} \Sigma$ and $\meseq' (x') := \meseq(N^{-1/\d} x')$. We also recall that, by Assumption \ref{H4}, the density of $\meseq'$ is bounded above by some $\om$.

Let us recall that $\{\Gammaj\}_{j \in J}$ are the connected components from the boundary of $\{\meseq>0\}$ and we let $\Gammapj := N^{1/\d} \Gammaj$. From Assumption \ref{H5}, each $\Gamma^{(j)}$ is a $C^1$ manifold of dimension $0 \leq \llj \le \d-1$ and there exists $\alphaj \ge 0$ such that 
\begin{align}
\label{reass5} & \cc N^{-\alphaj/\d}\, \dist (x, \Gammapj)^{\alphaj} \le \meseq'(x)\le \cC  N^{-\alphaj/\d} \,\dist(x, \Gammapj)^{\alphaj}, \\
\label{reass5b} & \|\meseq' \|_{C^{0,\min(\alphaj,1)}} \le C N^{-\min(\alphaj,1)/\d}
 \end{align}in a neighborhood of $\Gammapj$. 
In the case $\alphaj \geq 1$, we also have 
\begin{equation*}
\label{reass52} 
 |\nab \mu_V' (x) |  \le C N^{-\alphaj/\d}   \dist(x, \Gammapj)^{\alphaj-1}
 \end{equation*}
in a neighborhood of $\Gammapj$. 
 
\def \cu{c(\um)}
\begin{lem}
For any  $\um > 0$, there exists a constant $\cu > 0$ such that for $N \geq 1$ and each $j \in J$, we may find $T_j$ satisfying
\begin{equation} \label{choixdeT}
 T_j \in \left[\frac{N^{1/{\d}}}{\cc^{1/\alphaj}}\um^{1/\alphaj}, \frac{ N^{1/{\d}} }{\cc^{1/\alphaj} } \um^{1/\alphaj} + \cu \right]
 \end{equation}
 and such that, letting
 \begin{equation}
 \label{def:SigmaTp}
 \Sigma_T' :=\{x\in \Sigma' , \dist(x, \Gammapj) \ge T_j\text{ for } j \in J\},
 \end{equation}
we have 
$$
\int_{\Sigma_T'} d\meseq' \text{ is an integer,}
$$
and the equilibrium density is bounded below as 
\begin{equation} \label{borneinfmeseqsigmaT}
\meseq' \ge \um \text{ on } \Sigma'_{T}.
\end{equation}
\end{lem}
\begin{proof}
Taking any $T_j \geq \frac{N^{1/{\d}}}{\cc^{1/\alphaj}}\um^{1/\alphaj}$ ensures, by \eqref{reass5}, that \eqref{borneinfmeseqsigmaT} will be satisfied, and in particular the constraint that the total mass (for $\meseq'$) is an integer can be ensured on each $\Gammapj$ by an mean value argument applied on an interval of constant length (depending on $\mu$)\footnote{The boundary of $\Gammapj$ has a length of order $N^{\llj/\d}$. If $\llj \geq 1$ we could take an even smaller interval for \eqref{choixdeT}.} as in \eqref{choixdeT}. 
\end{proof}
We see that $\partial \Sigma_T'$ is formed of disjoint $C^1$ connected components   
$$\Gammapj_{T_j} := \{x\in \Sigma' , \dist(x, \Gammapj) = T_j \}.$$
 Moreover, by $C^1$ regularity of the $\Gamma^{(j)}$ we have
\begin{equation} \label{longueur}
\cc N^{\llj/\d} \leq \mathcal{H}^{\llj}(\Gammapj_{T_j}) \leq \cC N^{\llj/\d}, 
\end{equation}
where $\mathcal{H}^{\llj}$ denotes the Hausdorff measure of integer dimension $\llj$.

Finally, by \eqref{choixdeT}, each number $T_j N^{-1/\d}$ converges to a constant value as $N\to \infty$, so the rescaling $N^{-1/\d} \Sigma_T'$ converges (in Hausdorff distance) as $N \to \infty$ to a set $\Sigma_{\um}$ with piecewise $C^1$ boundary. We also let $\Sigma_{\um}' := N^{1/\d} \Sigma_{\um}$.

\subsubsection{Tiling the interior.}  \label{sec:notationtiling}
We now tile $\Sigma_T'$ by hypercubes whose size is large but independent of $N$. 
\begin{lem}[Tiling the interior of the domain]\label{tiling}
There exists a constant $C_0>0$ (depending on $\um, \om$) such that, for any $R>1$ and any $N \geq 1$, there exists a family $\mathcal K_N$ of closed hyperrectangles such that
\begin{itemize}
\item For all $K \in K_N$, $K \subset \Sigma_T'$.
\item The sidelengths of $K$ are between $R$ and $R+C_0/R$, and the sides are parallel to the axes of $\Rd$.
\item They have pairwise disjoint interiors.
\item They fill up $\Sigma_T'$ up to some boundary region
\begin{equation}\label{tile1}
\left\lbrace x\in \Sigma_T': \dist(x,\partial \Sigma_T') \ge C_0 R \right\rbrace \subset  \bigcup_{K\in\mathcal K_N} K.
\end{equation}
\item For all $K\in\mathcal K_N$, we have
\begin{equation}\label{inttile}
\int_K \meseq' \in \mathbb N.
\end{equation}
\end{itemize} 
\end{lem}
\begin{proof}
It is a straightforward modification  of \cite[Lemma 6.5]{sandier20152d}.
\end{proof}
With the notation of Lemma \ref{tiling}, for any $R > 1$ and $N \geq 1$, we define $\Sigmatil$ as
\begin{equation}
\label{def:Sigmapint} \Sigmatil := \bigcup_{K\in\mathcal K_N} K, 
\end{equation}
and we let $\Ntil$ be the integer defined as
 \begin{equation}\label{610b}
\Ntil := \meseq'(\Sigmatil).
\end{equation}

We let $\mNR$ be the number of hyperrectangles in $\mc{K}_N$ and we enumerate the elements of $\mc{K}_N$ as $K_1, \dots, K_{\mNR}$. For any $i \in \{1, \dots, \mNR\}$
\begin{itemize}
\item We let $\x_i$ be the center of $K_i$.
\item We let $\barcarr_i$ be the hypercube of sidelength $R$ with center $\x_i$, and in particular $\barcarr_i \subset K_i$.
\item We define $\Int_{\epsilon,i}, \Ext_{\epsilon,i}$ as
\begin{equation} \label{defintext23}
\Int_{\epsilon,i} := \{x \in \barcarr_i \ | \ \dist(x, \partial \barcarr_i) \geq 2\epsilon R \}, \quad \Ext_{\epsilon,i} := \barcarr_i \backslash \Int_{\epsilon, i}.
\end{equation}
\item We let $N_i := \int_{K_i} \meseq'$, which is by construction an integer. 
\item We denote by $\Nn_i$ (resp. $\Nnint_i$) the number of points of a configuration in $\barcarr_i$ (resp. in $\Int_{\epsilon,i}$).
\end{itemize}

We now gather some estimates about quantities related to the tiling.
\begin{lem} \label{usefulgeometric} We have:
 \begin{align}
 \label{volumetiling} & \lim_{R \ti} \lim_{N \ti} \frac{R^{\d}}{N} \mNR = |\Sigma_{\um}|, \\
\label{Ki} & |K_i| = R^{\d} + R^{\d-2}O_R(1), \\
\label{bornepointstiling} & C_1 R^{\d} \leq N_i \leq C_2 R^{\d}, \\
\label{Ni} & N_i = R^{\d} \meseq'(\x_i) + R^{\d+\kappa} N^{-\kappa/\d} \|\meseq\|_{C^{0, \kappa}(\Sigma)} O_R(1) + R^{\d-2} O_R(1), \\
\label{volumeExterieur} &  |\Ext_{\epsilon,i}| = 2\d\epsilon R^{\d} + \epsilon^2 O_{\epsilon}(1) R^{\d}.
 \end{align}
 The terms $O_R(1), O_N(1)$ and the positive constants $C_1, C_2$ depend only on $\um, \om, \d$. The term $ O_{\epsilon}(1)$ depends only on $\d$.
 \end{lem}
\begin{proof} The asymptotics \eqref{volumetiling} are implied by the more precise estimate
\begin{equation} \label{volumetilingprecis}
\frac{R^{\d}}{N} \mNR= |\Sigma_{\um}|(1+N^{-1/\d} R O_N(1))(1+R^{-2} O_{R}(1)),
\end{equation}
so let us prove the latter. By construction the $m_{N,R}$ hyperrectangles partition $\Sigmatil$ and have sidelengths in $[R, R + C_0/R]$, with $C_0$ depending only on $\um, \om$, hence the following holds 
\begin{equation*} 
m_{N,R}  R^{\d} \leq  |\Sigmatil| \leq m_{N,R} \left(R+\frac{C_0}{R}\right)^\d,
\end{equation*}
and in particular 
\begin{equation} \label{vt2}
m_{N,R}  R^{\d} =  |\Sigmatil| \left(1 +  R^{-2}  O_{R} (1) \right).
\end{equation}
On the other hand, from \eqref{tile1} we see that 
\begin{equation*}
|\Sigma'_T - \Sigmatil| \leq \left| \{x \in \Sigma'_T, \dist(x, \partial  \Sigma'_T) \leq C_0 R\} \right|,
\end{equation*}
whereas from \eqref{longueur} we deduce 
\begin{equation*}
| \{x \in \Sigma'_T, \dist(x, \partial  \Sigma'_T) \leq C_0 R\}| = R N^{1- 1/\d} O_N(1),
\end{equation*}
with a $O_N(1)$ depending only on $\um$.  We thus get
\begin{equation} \label{vt1}
 |\Sigma'_T| = |\Sigmatil| +  N^{1-1/{\d}} R O_N(1).
\end{equation}
Moreover, by definition of $\Sigma_{\um}$ and regularity of $\partial \Sigma$ we have
\begin{equation} \label{vt3}
|\Sigma'_{T}| = |\Sigma'_{\um}| + N^{1-1/{\d}} O_N(1).
\end{equation}
Combining \eqref{vt2}, \eqref{vt1} and \eqref{vt3} yields \eqref{volumetilingprecis}.

 Since the sidelengths of $K_i$ are in $[R, R + C_0/R]$ with a constant $C_0$ depending only on $\um$, we  get \eqref{Ki}. 

 Since $\meseq'$ is bounded above and below on $\Sigma'_{\um}$, \eqref{bornepointstiling} holds with constants depending on $\um, \om$. 

 To get \eqref{Ni} we combine \eqref{Ki} with the Assumption \ref{H4} on the Hölder regularity of the equilibrium measure, which yields
\begin{equation*}
\|\mu'_V(x) - \mu'_V(\x_i)\|_{L^{\infty}(K_i)} \leq \|\meseq\|_{C^{0, \kappa}} R^{\kappa} N^{-\kappa/\d}
\end{equation*}

 The bound \eqref{volumeExterieur} follows from the definition \eqref{defintext23} and elementary estimates.	
\end{proof}

From now on, and until Section \ref{sec6.6} we work only in $\Sigmatil$ defined in \eqref{tile1}. We recall that it can be written as a disjoint union of hyperrectangles (see Figure \ref{figuresigmaint}, where the grey region  corresponds to $\Sigmatil$).
\begin{figure}[h]
\begin{center}
\includegraphics[scale=1]{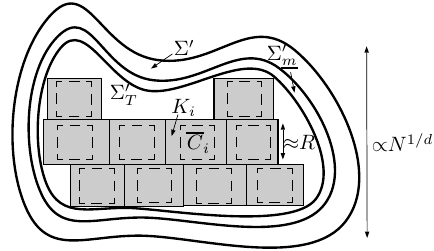}
\end{center}
\caption{The tiling of $\Sigma'$.}
\label{figuresigmaint}
\end{figure}

\subsection{Generating approximating microstates}
We now state a result, in the spirit of Sanov's theorem, in order to generate abstractly a whole family of point configurations in $\Sigmatil$ whose continuous and discrete averages over translations are close to some fixed tagged point process. We call these configurations \textit{approximating microstates}.

For any $\bPst$ in $\probas_{s} (\Sigma \times \config)$ we let 
$\bPst_{\um}$ be the tagged point process induced by restricting the “tag” coordinates to $\Sigma_{\um} \subset \Sigma$ i.e.
$$
\bPst_{\um} := \frac{1}{|\Sigma_{\um}|} \int_{\Sigma_{\um}} \bPst^{x} dx.
$$
The specific relative entropy of $\bPst_{\um}$ will be taken by restricting the tags accordingly
$$
\bERS[\bPst_{\um}|\Poisson^1] := \int_{\Sigma_{\um}} \ERS[\bPst^{x}|\Poisson] \, dx.
$$

\begin{lem} \label{restreindrecestpasgrave}
The tagged point process and its restriction are uniformly close as $\um \to 0$
\begin{equation}
\lim_{\um \to 0}  \sup_{\bPst \in \probas_{s} (\Sigma \times \config)} \dist_{\probas(\Sigma \times \config)} \left(\bPst_{\um}, \bPst\right) = 0
\end{equation}
In particular, for $r> 0$, for $\um$ small enough depending on $r$, we have
\begin{equation*}
B(\bPst_{\um}, r/2) \subset B(\bPst, r) \subset B(\bPst_{\um}, 2r).
\end{equation*}
\end{lem}
\begin{proof}
This follows from the fact that $|\Sigma - \Sigma_{\um}| \to 0$ as $\um \to 0$.
\end{proof}

For $R > 0$, we also denote by $\bPst_{\um,R}$ the restriction of $\bPst_{\um}$ to the hypercube $\carr_R$  (the push-forward of $\bPst_{\um}$ by the map $(x, \mc{C}) \mapsto (x, \mc{C} \cap \carr_R)$). 

Finally, we introduce a constant:
\begin{equation}\label{defcm}
\cmuvum := \mu_V(\Sigma_{\um}) \log \frac{\mu_V(\Sigma_{\um}) }{|\Sigma|}- \mu_V(\Sigma_{\um}) + |\Sigma_{\um}|.
\end{equation}

We recall that we have a family of $\mNR$ hypercubes $\barcarr_i$ of sidelength $R$, with centers $\x_i$ and that the total mass of $\meseq'$ on $\Sigmatil$ is an integer $\Ntil$, as defined in \eqref{610b}.
The following lemma says that the discrete space average as well as the continuum space average of randomly chosen configurations in $\barcarr_i$ occupy a volume around $\bP$ which is given by the specific relative entropy of $\bP$. 

\begin{lem} \label{SanovbQN2} 
Let $(\C_1, \dots, \C_{\mNR})$ be $m_{N,R}$ independent Poisson point processes of intensity~$1$ in each hypercube $\barcarr_i$. Letting $\Numb_i$ be the number of points in $\barcarr_i$, we condition on the following event
\begin{align}
\label{bonnbSanov} & \sum_{i=1}^{\mNR} \Numb_i = \Ntil,
% \label{UISanov} & \sum_{i=1}^{\mNR} \Numb_i \sqrt{\log(1 + \Numb_i)} \leq \CUI N,
\end{align}
i.e. in \eqref{EqSB2}, \eqref{EqSB3} below we consider the intersection with that event, without conditioning.

We let $\C$ be the point process obtained as the union of the configurations $\C_i$, namely
\begin{equation*}
\C := \sum_{i=1}^{\mNR} \C_i.
\end{equation*}

We define $\bM_{N,R}$ as the law of the following random variable with values in $\Sigmatil  \times \config$:
\begin{equation} \label{moyennediscretisee}
\DiscrAv(\C) := \frac{1}{m_{N,R}} \sum_{i=1}^{m_{N,R}} \delta_{(N^{-1/{\d}} \x_i, \theta_{\x_i} \cdot \C_i)}.
\end{equation}
It is the discrete average of the point configurations contained in the hypercubes $\barcarr_i$, translated by $\x_i\in \barcarr_i$.

We also define $\HM_{N,R}$ as the law of the random variable in $\Sigma'_{\um} \times \config$
\begin{equation*}
\ContAv(\C) := \frac{1}{N|\Sigma_{\um}|} \int_{\Sigma'_{\um}} \delta_{(N^{-1/{\d}} \x, \theta_\x \cdot \C)} \, d\x.
\end{equation*}
It is the continuous average of point configurations contained in the hypercubes $\barcarr_i$, translated by $\x \in \Sigma'_{\um}$. 

Then for any $\bPst \in \probas_{s,1}(\Sigma \times \config)$ the following inequality holds~:
\begin{equation} \label{EqSB2}
\liminf_{R \ti}\frac{1}{R^{\d}} \lim_{\nu \to 0}   \liminf_{N \ti} \frac{1}{m_{N,R}} \log \bM_{N,R}(B(\bPst_{\um,R}, \nu)) \geq - \bERS[\bPst_{\um}|\Poisson^1] - \cmuvum,
\end{equation}
and moreover, for any $\delta > 0$ we have 
\begin{multline} \label{EqSB3}
    \liminf_{R \ti} \frac{1}{R^{\d}} \lim_{\nu \to 0}    \liminf_{N \ti} \frac{1}{m_{N,R}} \log (\bM_{N,R},\HM_{N,R}) \left(B(\bPst_{\um,R}, \nu) \times B(\bPst_{\um}, \delta)\right)
    \\ \geq - \bERS[\bPst_{\um}|\Poisson^1] - \cmuvum 
\end{multline}
where $(\bM_{N,R},\HM_{N,R})$ denotes the joint law of $\bM_N$ and $\HM_N$, with the natural coupling.
\end{lem}
In contrast to the LDP stated in Proposition \ref{SanovbQN}, in the previous lemma we only care about the $\liminf$ because we will only use lower bounds to provide us with some family of configurations.
\begin{proof}
The proof is given in Section \ref{sec:preuvesanovdiscret}.
\end{proof}

\subsection{Regularizing and screening microstates}
Taking the approximating microstates from the previous Lemma \ref{SanovbQN2}, we apply to them the screening-then-regularization procedure described in Section \ref{sectionupsilon}. We obtain the following:
\begin{lem}\label{msregularized} Let $\bar{P}\in \probas_{s,1}(\Sigma \times \config)$ of finite energy, and 
\begin{itemize}
\item Let $M > 1, R > 0, \epsilon > 0$ satisfying the conditions \eqref{condiScr}. 
\item Let $0 < \eta < \hal \eta_0$, with $\eta_0$ depending only on $\d, \om$ as in Proposition \ref{Lemscreening}.
\item Let $0 < \tau < \hal \eta^2$.
\item Let $\delta_1 > 0$ and $\nu > 0$.
\item Let $N \geq 1$.
\end{itemize}
There exists a family $\Amod$ (depending on all the previous parameters) of point configurations in $\Sigmatil$, such that each configuration $\Cmod \in \Amod$ can be decomposed as  
$$ \Cmod= \sum_{i=1}^{m_{N,R}}\Cmod_i,$$ 
where $\Cmod_i$ is a configuration in $K_i$, and satisfies the following properties.
\begin{itemize}
\item The continuous average $\ContAv(\Cmod)$, defined as
$$
\ContAv(\Cmod) := \frac{1}{|\Sigma_{\um}'|} \int_{\Sigma_{ \um}'} \delta_{(N^{-1/{\d}} \x, \theta_{\x} \cdot \Cmod)} \, d\x, 
$$
is asymptotically close to $\bPst_{\um}$ for the distance on $\probas(\Sigma_{\um} \times \config)$
\begin{equation} \label{moycontrestpr}
\limsup_{\epsilon \to 0, M \ti, R \ti, \tau \to 0, \nu \to 0, N \ti} \sup_{\Cmod \in \Amod} d_{\probas(\Sigma_{\um} \times \config)} \left( \ContAv(\Cmod), \bPst_{\um} \right) \leq \delta_1.
\end{equation}
\item  The truncation error is asymptotically small
\begin{equation} \label{errdevpet}
\limsup_{\eta \t0, M \ti, R \ti, \tau \to 0, \nu \to 0, N \ti} \sup_{\Cmod\in \Amod } \frac{1}{N} \sum_{p \neq q \in \Cmod, |p - q| \leq \eta} \g(p - q) = 0.
\end{equation}
\item  For any $\Cmod \in \Amod$ there exists $\Emod$ in $\Screen(\Cmod, \meseq', \Sigmatil)$ satisfying 
\begin{equation} \label{usc2} 
\limsup_{\epsilon \t0, M\ti, R\ti,\tau\t0,\nu \t0, N\ti} \left( \frac{1}{|\Sigmatil|} \int_{\Sigmatil \times \R^\k} \yg |\Emod_\eta|^2 - \Esp_{\bPst_{\um}} \left[ \H \right] \right)\le 0.
\end{equation}
\item There is a good volume of such microstates
\begin{equation} \label{volrestgrd}
\liminf_{\epsilon \t0, M \ti, R \ti, \tau \to 0, \nu \t0, N\to \infty}   \frac{1}{|\Sigmatil| } \log \frac{\Leb^{\Ntil}}{|\Sigmatil|^{\Ntil}} \left( \Amod \right)  \geq - \bERS[\bP_{\um}|\Poisson^1] - \cmuvum  .
\end{equation}
\end{itemize}
\end{lem}
\begin{proof} If a point configuration $\C$ can be decomposed as
$$
\C := \sum_{i=1}^{m_{N,R}} \C_i, 
$$
where $\C_i$ is a point configuration in $K_i$, we will write as above
\begin{equation*}
 \DiscrAv(\C) := \frac{1}{m_{N,R}} \sum_{i=1}^{m_{N,R}} \delta_{(N^{-1/{\d}} \x_i, \theta_{\x_i} \cdot \C_i)}, \quad\ContAv(\C) := \frac{1}{|\Sigma_{\um}'|} \int_{\Sigma_{ \um}'} \delta_{(N^{-1/{\d}} \x, \theta_{\x} \cdot \C)} \, d\x.
\end{equation*}

\subsubsection{The family $\Aabs$.}
For any $\delta, \nu, N, R$ positive, by Lemma \ref{SanovbQN2} we know that there exists a family\footnote{“abs” as “abstract” because we generate them abstractly using Sanov's theorem} $\Aabs$of configurations which can be decomposed as 
\begin{equation} \label{Cabsdecompo}
\Cabs = \sum_{i=1}^{m_{N,R}} \Cabs_i,
\end{equation}
where $\Cabs_i$ is a point configuration in the hypercube $\barcarr_i$, and with $\Ntil$ total points, such that 
\begin{equation} \label{DiscretContgent}
\DiscrAv(\Cabs) \in B(\bP_{\um, R}, \nu), \quad \ContAv(\Cabs) \in B(\bP_{\um}, \delta).
\end{equation}
According to \eqref{EqSB3}, the logarithmic volume of this family can be bounded below asymptotically as
\begin{equation} \label{boundbelvolreg}
\liminf_{R\ti, \nu \to 0, N\ti} 
\frac{1}{m_{N,R} R^{\d}} \log \frac{\Leb^{\Ntil}}{(m_{N,R}R^{\d})^{\Ntil}} \left( \Aabs \right)\\ \ge -\bERS[\bPst_{\um}|\Poisson^1] - \cmuvum.
\end{equation}
%Indeed, the law of the point process $\C$ of Lemma \ref{SanovbQN2} coincides with the law of the point process induced by the $\Ntil$-th product of the normalized Lebesgue measure on $\cup_{i=1}^{m_{N,R}} \barcarr_i$, and thus \eqref{EqSB3} implies \eqref{boundbelvolreg}.

\subsubsection{From $\Aabs$ to $\Amod$.}
We let $\Amod$ be the set of configurations obtained after applying the procedure described in Section \ref{sec:screening}. For each $\Cabs$ in $\Aabs$ we decompose $\Cabs$ as in \eqref{Cabsdecompo} and for any $i= 1, \dots, m_{N,R}$ we let $\Phimod_i(\Cabs)$ be the set of configurations obtained after applying the map $\Phimod$. In the following, for good definition, we have to conjugate by the translation of vector $\x_i$, so that $\Phimod_i$ is eventually a family of configurations in $K_i$.
$$
\Phimod_i(\Cabs) := \theta_{- \x_i} \cdot \Phimod(\theta_{\x_i} \cdot \Cabs_i, \mu'_V(\x_i + \cdot), \theta_{\x_i}\cdot K_i).
$$
We then let $\overline{\Phimod}(\Cabs)$ be the Cartesian product
$$
\overline{\Phimod}(\Cabs) := \prod_{i=1}^{m_{N,R}}  \Phimod_i(\Cabs)
$$
and $\Amod$ (“mod” as “modified”) is defined as the set of configurations 
$$
\Cmod := \sum_{i=1}^{m_{N,R}} \Cmod_i,  
$$
where $(\Cmod_1, \dots, \Cmod_{\mNR})$ belongs to $\overline{\Phimod}(\Cabs)$.

For any $\Cabs$ in $\Aabs$, with $\Cabs = \sum_{i=1}^{\mNR} \Cabs_i$, we denote by $I_1$ the set of indices such that $(\Cabs_i, \meseq')$ is screenable. The following fact will be used repeatedly. 
\begin{claim} \label{claim:souventecrantable}
\begin{equation} \label{souventecrantableN}
\lim_{M \ti, R \ti, \nu \t0, N \ti} \inf_{\Cabs \in \Aabs} \frac{\# I_1}{m_{N,R}} =1.
\end{equation}
\end{claim}
\begin{proof}
Combining \eqref{presquetousscreenable} with the fact that $\DiscrAv(\Cabs) \in B(\bP_{\um, R}, \nu)$ (see \eqref{DiscretContgent}), we see that \eqref{souventecrantableN} holds when asking whether $(\Cabs_i, \meseq'(\x_i))$ is screenable instead of $(\Cabs_i, \meseq')$ (with a varying background). To deal with the variation of $\meseq$, we use the following Hölder bound (with a $O_R(1)$ depending only on $\um, \d$)
$$
\| \meseq'(\x) - \meseq'(\x_i) \|_{L^{\infty}(K_i)} \leq \|\mu\|_{C^{0,\kappa}(\Sigma)}  R^{\kappa} N^{-\kappa/\d} O_R(1),
$$
together with the fact that $\Sp$ is open as stated in Lemma \ref{lem:ouverture}.
\end{proof}

Let us now check that the family $\Amod$ satisfies the desired properties.

\subsubsection{Distance.} 
We want to show that the continuous average satisfies \eqref{moycontrestpr}. We thus claim that the screening-then-regularizing procedure preserves the closeness of  the continuous average to $\bar{P}_{\um}$ (however in general it does not preserve that of the discrete average). To prove that claim, we have to distinguish between hyperrectangles where the configuration is screenable (where the configuration is only modified in a thin layer or by moving points by a distance at most $\tau$) and hyperrectangles where it is not (where the configuration is then completely modified).

Let $\Cmod = \sum_{i=1}^{m_{N,R}} \Cmod_i$ be in $\Amod$ (where $\Cmod_i$ is the point configuration in the hyperrectangle $K_i$), we may find 
$\Cabs = \sum_{i=1}^{m_{N,R}} \Cabs_i$  in $\Aabs$ such that $\Cmod$ has been obtained from $\Cabs$ by screening-then-regularizing. 

\begin{claim} \label{claim:conttodiscr}
We may evaluate the distance between the continuous averages of $\Cabs$ and $\Cmod$ in terms of the distance between the configurations in each hypercube $K_i$. 
\begin{multline} \label{rameneaKi}
\limsup_{\um \to 0, \epsilon \to 0, R \ti, N \ti} \sup_{\Cabs, \Cmod} d_{\probas(\Sigmaum \times \config)} \left(\ContAv(\Cabs), \ContAv(\Cmod) \right)  \\ - \frac{1}{\mNR} \sum_{i=1}^{\mNR} \dconfig(\theta_{\x_i} \Cabs_i, \theta_{\x_i} \Cmod_i) = 0.
\end{multline}
\end{claim}
\begin{proof}
For $\delta > 0$ fixed, by Lemma \ref{Ldistconfig} we may find $k \geq 1$ such that for any $\C \in \config$ 
$$
\dconfig(\C, \C \cap \carr_k) \leq \delta.
$$
For any $i = 1, \dots, \mNR$, for any $\x \in K_i$ such that $\dist(\x, K_i \backslash \barcarr_i) \geq k$, we have
$$
\left(\theta_{\x} \cdot \C\right) \cap \carr_k = \left(\theta_{\x} \cdot \C_i \right) \cap \carr_k, 
$$
and also there exists $C$ depending only on $\d$ such that, for $F \in \Lip(\Sigmaum, \config)$,
$$
\left| F\left(N^{-1/d} \x, \theta_{\x} \cdot \C_i \right) - F\left(N^{-1/d} \x_i, \theta_{\x_i} \cdot \C_i\right) \right| \leq C R N^{-1/\d}.
$$
In particular, the distance between the continuous average over translates on points $\x \in K_i$ such that $\dist(\x, K_i \backslash \barcarr_i) \geq k$ and a Dirac mass at $\left(N^{-1/d} \x_i, \theta_{\x_i} \cdot \C_i\right)$ is bounded by $C R N^{-1/\d}$. This allows us to compare a continuous average with a discrete one. 

In the continuous average $\ContAv(\Cmod)$ or $\ContAv(\Cabs)$ we may omit the points $\x \in \Sigmaum$ such that $\x \in \Sigmaum \backslash \Sigmatil$, or that $\x \in \Sigmatil$ but $\dist(\x, K_i \backslash \barcarr_i) \geq k$ (for $i$ such that $\x \in K_i$),  up to an error which is negligible as $N \ti, R \ti, \epsilon \t0, \um \to 0$ because such points represent a negligible volume of translates.

We thus have, for any fixed $\delta > 0$
\begin{equation*}
\limsup_{\um \to 0, \epsilon \t0, R \ti, N \ti} \sup_{\C = \sum \C_i} d_{\probas(\Sigmaum \times \config)} \left(\ContAv(\C), \frac{1}{\mNR} \sum_{i=1}^{\mNR} \delta_{\left(N^{-1/d} \x_i, \theta_{\x_i} \cdot \C_i\right)} \right) \leq  \delta.
\end{equation*}
On the other hand, we have of course
$$
\left| F\left(N^{-1/d} \x_i, \theta_{\x_i} \cdot \Cabs_i\right) - F\left(N^{-1/d} \x_i, \theta_{\x_i} \cdot \Cmod_i\right) \right| \leq  \dconfig (\theta_{\x_i}  \Cabs_i, \theta_{\x_i}  \Cmod_i).
$$
By the triangle inequality, we obtain \eqref{rameneaKi}.
\end{proof}

Next, for any $i = 1, \dots, m_{N,R}$ we want to evaluate $\dconfig \left(\Cabs_i, \Cmod_i\right)$. Let us recall that $I_1$ denotes the set of indices $i = 1, \dots, m_{N,R}$ such that $(\Cabs_i, \meseq')$ is in $\Sp$ and $I_2$  the set of indices such that $(\Cabs_i, \meseq')$ is not in $\Sp$.
\begin{claim} \label{claim:discrproche}
We have 
\begin{equation}
\limsup_{\tau \to 0, R \ti} \frac{1}{\mNR} \sum_{i=1}^{\mNR} \dconfig(\theta_{\x_i} \Cabs_i, \theta_{\x_i} \Cmod_i) - \frac{\# I_2}{\mNR}  = 0
\end{equation}
\end{claim}
\begin{proof}
The distance $\dconfig$ is bounded by $1$, hence we may write
\begin{equation} \label{distancecasI2}
\sum_{i \in I_2} \dconfig\left(\theta_{\x_i}  \Cabs_i, \theta_{\x_i} \Cmod_i\right) \leq \# I_2.
\end{equation}

On the other hand, we have
\begin{equation} \label{distancecasI1}
\limsup_{\tau \to 0,  R \ti} \sup_{\Cabs, \Cmod} \dconfig \left(\theta_{\x_i} \Cabs_i, \theta_{\x_i} \Cmod_i\right)  = 0.
\end{equation}
Indeed, the screening procedure does not affect the points in $\theta_{\x_i} \Int_{\epsilon, i}$ which contains an hypercube $\carr_{R/2}$, hence as $R \ti$ the distance between $\theta_{\x_i}  \Cabs_i$ and any associated screened configuration tends to $0$ (uniformly with respect to all the other parameters). Then the regularization procedure has a negligible effect on the distance as $\tau \to 0$ (also uniformly with respect to all the other parameters) as stated in \eqref{distanceregular}.
\end{proof}

Finally, combining Claim \ref{claim:souventecrantable}, Claim \ref{claim:conttodiscr} and Claim \ref{claim:discrproche}, we obtain \eqref{moycontrestpr}.

\subsubsection{Truncation.}  \label{sectruncation}
Let $\Cmod = \sum_{i=1}^{m_{N,R}} \Cmod_i$ be in $\Amod$ and let $\Cabs$ such that $\Cmod$ has been obtained from $\Cabs$ by screening-then-regularizing. 

By construction (see \eqref{loindubord} for $i \in I_1$ and \eqref{loindubordartificiel} for $i \in I_2$, if $\eta$ is small enough (depending only on $\d, \um, \om$) the only pairs of points in $\Cmod$ at distance less than $\eta$ are included in $\barcarr_i$ for some $i \in I_1$.

For $i \in I_1$, we may apply \eqref{controletroncatureregul} and write, with $C$ depending only on $\d, \s$:
\begin{multline*} 
\sum_{p \neq q \in \Cmod_i, |p - q| \leq \eta} \g(p - q) \\ 
\leq
C \g(\tau) \sum_{\vec{i} \in 6\tau \Zd} \left(\Numb_{12\tau} [\vec{i}]^2(\Cabs_i)-1\right)_+ 
+ \sum_{p\neq q \in \Cabs_i , \tau \le |p-q|\le 2\eta } \g(p-q).
\end{multline*}

The condition \eqref{rajouteborneNscreen} of screenability implies that $\Cabs_i$ has at most $MR^{\d}$ points hence we  could write the previous equation as
\begin{multline} \label{controletruncationI12}
\sum_{p \neq q \in \Cmod_i, |p - q| \leq \eta} \g(p - q) \leq
C \g(\tau) \sum_{\vec{i} \in \tau \Zd}(\Numb_{12\tau} [\vec{i}]^2(\Cabs_i)-1)_+ \wedge M^2R^{2\d}) \\
+ \sum_{p\neq q \in \Cabs_i , \tau \le|p-q|\le 2\eta } \g(p-q) \wedge MR^{\d} \g(\tau).
\end{multline}
This re-writing is mostly technical, indeed we need to use \textit{bounded} functions in order to test them against the weak convergence of point processes.

By a Fubini-like argument we get
\begin{equation*}
\frac{1}{\mNR} \sum_{i = 1}^{\mNR}\sum_{\vec{i} \in \tau \Zd} (\Numb_{12\tau} [\vec{i}]^2(\Cabs_i)-1)_+ \leq \frac{C}{\tau^\d} \Esp_{\DiscrAv(\Cabs)} \left[ \left( \Numb_{12 \tau}(x)^2 - 1 \right)_+ \right], 
\end{equation*} 
and on the other hand we have
\begin{equation*}
\frac{1}{\mNR} \sum_{i = 1}^{\mNR} \sum_{p \neq q \in \Cabs_i , \tau \le|p-q|\le 2\eta } \g(p-q) \leq \Esp_{\DiscrAv(\Cabs)} \left[ \sum_{p \neq q \in \C \cap \carr_R , \tau \le|p-q|\le 2\eta } \g(p-q) \right],
\end{equation*}
and the same expressions hold with the bounded functions of \eqref{controletruncationI12}.

We now combine \eqref{HvsW2a}, \eqref{HvsW2b} with the assumption that $\DiscrAv(\Cabs)$ is close to $\bP$ (see \eqref{DiscretContgent}), and we get \eqref{errdevpet}.

\subsubsection{Energy.}
By construction, for any $i =1 ,\dots, \mNR$ we have a vector field $\Emod_i$ in $\Screen(\Cmod_i, \meseq', K_i)$. Setting 
$$
\Emod := \sum_{i=1}^{\mNR} \Emod_i \indic_{K_i\times \R^\k}
$$ 
provides a vector field in $\Screen(\Cmod, \meseq', \Sigmatil)$, whose energy we now have to bound.
 
For $i \in I_1$ the energy is bounded after screening as in \eqref{erreurEcrantage} and after regularization as in \eqref{regulnetouchepasE} (see also \eqref{regulexplicite}). It yields, for some constant $C_1$ depending only on $\s,\d,\um,\om$ and some constant $C_2$ depending only on $R, \d, \om, \eta$
\begin{multline}
\label{energiesurI1} 
\int_{K \times \R^\k} \yg |\Emod_{i,\eta}|^2 \leq \H(\Cmod_i, \meseq')  \\ + C_1 \left(\g(\eta) M  \ep R^\d + R^{\d+3-\gamma} N^{-2/\d} + \left( \g(\eta) M  \ep R^{2\d+3-\gamma} N^{-2/\d} \right)^{\hal} \right)  + C_2 R^\d \tau.
\end{multline}
For $i \in I_2$, the energy is bounded as in \eqref{contEnerconstr2}. It yields, for some constant $C_1$ depending only on $\s,\d,\um,\om, \|\meseq\|_{C^{0, \kappa}(\Sigma)}$ 
\begin{equation}
\label{energiesurI2}
\int_{K\times \R^\k} \yg |\Egene|^2   \leq  - \c N_i \g(\eta) +  C_1 R^\d +
C_1 R^{\d+3-\gamma}N^{-2/\d}.
\end{equation}
In \eqref{energiesurI1}, \eqref{energiesurI2} we have used the fact that the Hölder norm of $\meseq'$ on $K_i$ is bounded by  $\|\meseq\|_{C^{0, \kappa}} N^{-1/\d}$.

We thus have, combining \eqref{energiesurI1}, \eqref{energiesurI2} and summing the contributions of each $K_i$,
\begin{multline}
\int_{\Sigmatil \times \R^{\k}} \yg |\Emod_{\eta}|^2 + \c \Ntil \g(\eta) \leq 
\sum_{i \in I_1} \left( \H(\Cmod_i, \meseq') + \c N_i \g(\eta) \right)  \\ + \# I_2 C_1 R^\d 
+ \# I_1 C_1 \left(\g(\eta) M  \ep R^\d + R^{\d+3-\gamma} N^{-2/\d} + \left( \g(\eta) M  \ep R^{2\d+3-\gamma} N^{-2/\d} \right)^{\hal} \right)  \\ + \# I_1 C_2 R^\d \tau + \# I_2 C_1 R^{\d+3-\gamma}N^{-2/\d}.
\end{multline}

Using \eqref{HvsW} and \eqref{DiscretContgent} we get
\begin{equation}
\limsup_{\eta \to 0, M \ti,  R \ti,  \nu \to 0} \frac{1}{\mNR} \sum_{i \in I_1} \left( \H(\Cmod_i, \meseq') + \c N_i \g(\eta) \right) \leq \bttW(\bP_{\um}, \meseq).
\end{equation}

On the other hand, the error terms are seen to be negligible with respect to $N$ when taking the limits in the correct order.  Indeed, using \eqref{souventecrantableN} we get
$$
\limsup_{M \ti, R \ti, \nu \t0, N \ti} \frac{1}{N} C_1 \# I_2  R^\d = 0.
$$
We also have, by a direct computation
\begin{align*}
& \limsup_{\eta \to 0, \epsilon \to 0, M \ti, R \ti, N \ti} \frac{\mNR}{N}  \g(\eta) M  \ep R^\d + R^{\d+3-\gamma} N^{-2/\d} = 0 \\
& \limsup_{\eta \to 0, \epsilon \to 0, M \ti, R \ti, N \ti} \frac{\mNR}{N}   \left( \g(\eta) M  \ep R^{2\d+3-\gamma} N^{-2/\d} \right)^{\hal} = 0, \\
& \limsup_{R \ti, \tau \to 0, N \ti }  \frac{\mNR}{N} C_2 R^\d \tau = 0.
\end{align*}
We thus finally obtain \eqref{usc2}.

\subsubsection{Volume.}
We will use the notation of Section \ref{sec:notationtiling} and the results of Lemma~\ref{usefulgeometric}.

In order to apply Lemma \ref{volumeoperation}, where the volume estimates have been summarized, we find $\v$ such that $\v \leq \v_0$ (as in \eqref{def:v0}) and that for any $i \in I_1$
\begin{equation} \label{condnextc}
N_i - \Nnint_i \leq \frac{|\mathrm{Ext}_\ep|}{2 \v}.
\end{equation}

\begin{claim} \label{claim:changementv0}
There exists $C$ depending only on $\d, \um, \om$ such that, setting
\begin{equation}
\label{def:v1}
\v(\d,\om, \um, \epsilon) := \min(\v_0, C \epsilon),
\end{equation}
the condition \eqref{condnextc} is satisfied.
\end{claim}
\begin{proof}
By Lemma \ref{usefulgeometric} we know that $N_i \leq C R^\d$, and $|\mathrm{Ext}_\ep| \geq  \d \epsilon R^{\d}$ (for $\epsilon$ small enough depending only on $\d$), hence taking $\v$ smaller than some multiple of $\epsilon$ ensures  \eqref{condnextc}.
\end{proof}

Applying the conclusions of Lemma \ref{volumeoperation} for $i \in I_1$ and $i \in I_2$ separately, we obtain, for some $C, \v_1$ depending only on $\d, \um, \om$.
\begin{multline}\label{contVol} 
\log \Leb^{\otimes \Ntil} \left( \Amod \right) -  \log \Leb^{\otimes \Ntil}  \left( \Aabs \right) \\ 
\geq \int_{\Aabs} \sum_{i \in I_1}  \log\left( (N_i - \Nnint_i)! \left(\frac{\v}{|\Ext_{\epsilon}|}\right)^{N_i - \Nnint_i}\right) 
+ (N_i- \Nn_i) \log |\Ext_{\ep}| \\
- C_1 \sum_{i \in I_1} \sum_{\vec{i} \in 6 \tau \Zd} \Numb(\vec{i}, 6 \tau) \log \Numb(\vec{i}, 6 \tau) + \sum_{i\in I_2 }\log \left( N_i ! \v_1^{N_i} R^{-\d\Nn_i}\right) \, d\Leb^{\otimes \Ntil}(\Cabs). 
\end{multline} 
In \eqref{contVol}, for $i \in I_1$ fixed, the sum $\sum_{\vec{i} \in 6 \tau \Zd} \Numb(\vec{i}, 6 \tau) \log \Numb(\vec{i}, 6\tau)$ is implicitely applied to the configuration $\Cabs_i$.

First, we control the terms concerning $i \in I_1$. The successive estimates are split into several claims.
\begin{claim} \label{claim:StirJen}
For any $\Cabs$ we have
\begin{multline} \label{StirJen}
\sum_{i \in I_1}  \log\left( (N_i - \Nnint_i)! \left(\frac{\v}{|\Ext_{\epsilon}|}\right)^{N_i - \Nnint_i}\right) 
+ (N_i- \Nn_i) \log |\Ext_{\ep}| 
\\ 
\geq \# I_1 |\Ext_{\epsilon}| \left( \frac{1}{\# I_1} \sum_{i \in I_1} \frac{N_i - \Nnint_i}{|\Ext_{\ep}|} \right) \log \left( \frac{1}{\# I_1} \sum_{i \in I_1} \frac{N_i - \Nnint_i}{|\Ext_{\ep}|} \right) \\
+ \sum_{i \in I_1} \left(\Nnint_i - \Nn_i \right) \log |\Ext_{\ep}| + \sum_{i \in I_1} \left(N_i - \Nnint_i\right) (\log \v -1).
\end{multline}
\end{claim}
\begin{proof}
First, we use Stirling's estimate and get
\begin{multline*}
\sum_{i \in I_1}  \log\left( (N_i - \Nnint_i)! \left(\frac{\v}{|\Ext_{\epsilon}|}\right)^{N_i - \Nnint_i}\right) + (N_i- \Nn_i) \log |\Ext_{\ep}|
\\ \geq  \sum_{i \in I_1} (N_i - \Nnint_i) \log (N_i - \Nnint_i) - (N_i - \Nnint_i) - (N_i - \Nnint_i) \log |\Ext_{\epsilon}| \\ + (N_i - \Nn_i) \log |\Ext_{\epsilon}| + (N_i - \Nnint_i) \log \v.
 \end{multline*}
 We may  re-write this as
 \begin{multline*}
 \# I_1 |\Ext_{\epsilon}|  \left( \frac{1} {\# I_1} \sum_{i \in I_1} \frac{N_i - \Nnint_i}{|\Ext_{\ep}|} \log \frac{ N_i - \Nnint_i}{|\Ext_{\ep}|} \right) \\ + \sum_{i \in I_1} (N_i - \Nn_i) \log |\Ext_{\epsilon}| + (N_i - \Nnint_i) (\log \v  - 1). 
 \end{multline*}
and we get \eqref{StirJen} by applying Jensen's inequality to the convex map $x \mapsto x \log x$.
 \end{proof}

\begin{claim} \label{claim:decompoNinint}
\begin{multline} \label{decompoNiNninti}
\sum_{i \in I_1} N_i - \Nnint_i =  \#I_1 |\Extep| + \sum_{i \in I_1} \left(\meseq'(\x_i) |\Intep| - \Nnint_i\right)  \\
+ \#I_1 |\Extep| \left(\frac{\mNR}{\# I_1} \frac{\meseq'(\Sigmatil)}{|\Sigmatil|}(1 + O_R(R^{-2})) - 1\right) + N\left(R^{\kappa} N^{-\kappa/\d} O_R(1) + R^{-2} O_R(1) \right).
\end{multline}
\end{claim}
\begin{proof}
We start by the following decomposition
\begin{multline} \label{predecompoNiint}
\sum_{i \in I_1} N_i - \Nnint_i = \sum_{i \in I_1} \left( \meseq'(\x_i) R^\d - \Nnint_i \right) + \sum_{i \in I_1} \left( N_i - \meseq'(\x_i) R^\d \right) \\
= \sum_{i \in I_1} \left( \meseq'(\x_i)  |\Intep| - \Nnint_i \right) + \sum_{i \in I_1} \meseq'(\x_i) |\Extep| + N \left(R^{\kappa} N^{-\kappa/\d} O_R(1) + R^{-2} O_R(1)\right).
\end{multline}
We have used the fact that $|\Intep| + |\Extep| = R^{\d}$, and we have used \eqref{Ni} to control the error $N_i - \meseq'(\x_i) R^\d$. Using again \eqref{Ni} together with \eqref{volumetiling} we find
\begin{multline} \label{sommemuxi}
\sum_{i \in I_1} \meseq'(\x_i)  |\Extep|  = \# I_1 |\Extep| \left(\frac{\mNR}{\# I_1} \frac{\meseq'(\Sigmatil)}{|\Sigmatil|}(1 + O_R(R^{-2}))\right) \\
+  N \left(R^{\kappa} N^{-\kappa/\d} O_R(1) + R^{-2} O_R(1)\right).
\end{multline}
Combining \eqref{predecompoNiint} and \eqref{sommemuxi} yields \eqref{decompoNiNninti}.
\end{proof}

\begin{claim}  \label{claim:volumepourI1a}
We have:
\begin{align}
\label{limnuNintNi} & \limsup_{R \ti, \nu \t0, N \ti} \frac{1}{N} \sum_{i \in I_1} \left|\Nnint_i - \meseq'(\x_i) |\Int_{\epsilon}|\right| = 0,\\
 \label{sharpdiscrestimate} &\limsup_{R \ti, \nu \t0, N \ti} \frac{1}{N} \sum_{i \in I_1} \left|\Nn_i - \meseq'(x_i) R^\d\right|  = 0.
\end{align}
  \end{claim}
\begin{proof}
Let $\Dint(x,\C)$ be the discrepancy\footnote{We add a truncation $\wedge MR^\d$ in order to get a bounded function. We may do it because \eqref{rajouteborneNscreen} holds.}
$$
\Dint(x,\C) := \left( \Numb_R(\C) - \meseq(x) |\Int_{\epsilon}|\right) \wedge MR^\d
$$
We may observe that
\begin{equation} \label{discrpardiscr}
\frac{1}{\mNR} \sum_{i\in I_1} \left| \Nnint_i - \meseq'(\x_i) |\Int_{\epsilon}|  \right| \leq \Esp_{\DiscrAv(\Cabs)} \left[ \left| \Dint(x,\C) \right| \right],
\end{equation}
On the other hand, the discrepancy estimate of Lemma \ref{LemmeDiscr} implies that
\begin{equation*}
\Esp_{\bPum} \left[ \left(\Dint(x,\C)\right)^2 \right] = R^{\d+\s} O_R(1)
\end{equation*}
with a $O_R(1)$ depending only on $\d, \s, \bP$. Using Cauchy-Schwarz inequality and the fact that $\s < \d$ we thus obtain
\begin{equation} \label{discrparP}
\Esp_{\bPum} \left[ \left| \Dint(x,\C) \right| \right]  = R^{\d} o_R(1),
\end{equation}
a $o_R(1)$ depending only on $\d, \s, \bP$.

Since \eqref{DiscretContgent} holds, combining \eqref{discrpardiscr} and \eqref{discrparP}, we get \eqref{limnuNintNi}, and  \eqref{sharpdiscrestimate} follows by a similar argument.
\end{proof}

Combining Claim \ref{claim:StirJen}, Claim \ref{claim:decompoNinint} and Claim \ref{claim:volumepourI1a}, we can settle the first terms concerning $i \in I_1$.
\begin{claim} \label{claim:voltermesI11}
\begin{multline} \label{voltermesI11}
\liminf_{\um \to 0, \epsilon \to 0, R \ti, \nu \t0, N \ti}  \frac{1}{N} \sum_{i \in I_1}  \log\left( (N_i - \Nnint_i)! \left(\frac{\v}{|\Ext_{\epsilon}|}\right)^{N_i - \Nnint_i}\right)  \\ 
+  \frac{1}{N} \sum_{i \in I_1} (N_i- \Nn_i) \log |\Ext_{\ep}| = 0.
\end{multline}
\end{claim} 
\begin{proof}
It remains to check that the following quantity from \eqref{StirJen} tends to $0$
$$
\frac{1}{N} \sum_{i \in I_1} \left(N_i - \Nnint_i\right) (\log \v -1).
$$
We decompose $N_i - \Nnint_i$ as 
$$
N_i - \Nnint_i = \left( N_i - \meseq'(\x_i) R^{\d}\right) + \left( \meseq'(\x_i) |\Intep| - \Nnint_i\right) + \meseq'(\x_i) |\Extep|.
$$
The contribution of the first two terms in the right-hand side are small as $R \ti, \nu \to 0, N \ti$ (for fixed $\epsilon$) according to Claim \ref{claim:volumepourI1a}. The last term gives a contribution
$$
\sum_{i \in I_1} \meseq'(\x_i) |\Extep|,
$$
which was proven to be of order $\frac{N}{R^{\d}} |\Extep|$ (see \eqref{sommemuxi}), with $|\Extep|$ of order $\epsilon R^\d$. On the other hand, from the choice of $\v$ as in \eqref{def:v1} we see that $\log \v - 1$ is of order $\log \epsilon$.  We thus have
\begin{equation*}
\frac{1}{N} \sum_{i \in I_1} \meseq'(\x_i) |\Extep| (\log \v -1) = O\left( \epsilon \log \epsilon \right),
\end{equation*} 
and thus goes to zero as $\epsilon \to 0$ (depending only on $\um, \d, \s$).
\end{proof}

Concerning the terms $i \in I_1$, we are left to bound the volume loss due to the regularization.
\begin{claim} \label{claim:volumelossI1regul}
\begin{equation}
\label{volumelossI1regul} 
\lim_{R \ti, \tau \t0, \nu \to 0, N \ti}   \frac{1}{N} \sum_{i \in I_1} \sum_{\vec{i} \in 6 \tau \Zd} \Numb_{6\tau} [\vec{i}] \log \Numb_{6\tau} [\vec{i}] = 0
\end{equation}
\end{claim}
\begin{proof}
We argue as in Section \ref{sectruncation}, together with the trivial bound $n \log n \leq (n^2-1)_+$.
\end{proof}

Finally, we turn to the terms concerning $i \in I_2$. 
\begin{claim} \label{claim:I2}
\begin{equation}
\liminf_{R \ti, \nu \to 0, N \ti} \frac{1}{N} \sum_{i\in I_2 }\log \left( N_i ! \v_1^{N_i} R^{-\d\Nn_i}\right) = 0.
\end{equation}
\end{claim}
\begin{proof}
Stirling's formula and elementary manipulations yield
$$
\log \left( N_i ! \v_1^{N_i} R^{-\d\Nn_i}\right) \geq \left(N_i - \Nn_i \right) \log R^{\d} + N_i \log \frac{N_i}{R^\d} - N_i (1 - \log \v_1).
$$
Using \eqref{Ni} we obtain, with $C$ depending only on $\d, \um, \om$
$$
\sum_{i \in I_2} \log \left( N_i ! \v_1^{N_i} R^{-\d\Nn_i}\right) \geq \log R^{\d} \sum_{i \in I_2} \left(N_i - \Nn_i \right) - C \# I_2 R^\d \log \epsilon. 
$$
Discrepancy estimates as in Claim \ref{claim:volumepourI1a} yield that
$$
\frac{1}{\mNR} \sum_{i \in I_2} |N_i - \Nn_i|
$$
 is small with respect to $R^\d$, and in fact it is of order $R^{\hal(\d + \s)}$, with $\s < \d$. In particular, 
$$
\limsup_{R \ti, \nu \to 0, N \ti} \log R^{\d} \frac{1}{N} \sum_{i \in I_2} |N_i - \Nn_i| = 0.
$$
On the other hand, we know that $\# I_2 R^\d$ is negligible with respect to $N$, hence for any $\epsilon$ fixed
$$
\limsup_{R \ti, \nu \to 0, N \ti} \frac{1}{N} \# I_2 R^\d \log \epsilon = 0.
$$
This proves the claim.
\end{proof}

Combining all the previous estimates, we see that the volume loss is negligible with respect to $N$, and thus \eqref{volrestgrd} holds.
\end{proof}
This concludes the proof of Lemma \ref{msregularized}.

\subsection{Completing the construction and conclusion}\label{sec6.6}
There remains to complete the construction in the thin layer $\Sigma' \backslash \Sigmatil$.
\begin{lem} \label{lem:pointsexterieurs} 
Let $N, R$ be fixed and $\eta \in (0,1)$. Let $\um > 0$. 

There exists a family $\Aext$ of point configurations with $N -\Ntil$ points in $\Sigma' \backslash \Sigmatil$ and which satisfy the following.
\begin{enumerate}
\item For any $\Cext$ in $\Aext$, we have
\begin{equation*}
\min_{p \neq q \in \Cext} |p-q| \geq \eta_0, \quad \min_{p \in \Cext} \dist\left(p, \partial \left( \Sigma' \backslash \Sigmatil \right) \right) \geq \eta_0.
\end{equation*}
\item For any $\Cext$ in $\Aext$, there exists a vector field $\Eext \in \Screen\left(\Cext, \meseq', \Sigma' \backslash \Sigmatil\right)$, and it satisfies
\begin{equation} \label{energieext}
\int_{(\Sigma' \backslash \Sigmatil) \times \R^\k} \yg |\Eext_{\eta}|^2 \leq C (|\Sigma'| - |\Sigmatil|),
\end{equation} 
for some $C$ depending only on $\d, \s, \om$.
\item The volume of $\Aext$ is bounded below
\begin{equation} \label{volumeext}
\Leb^{\tens (N -\Ntil)} \left(\Aext \right) \geq C^{(N - \Ntil)} (N-\Ntil)! \ , 
\end{equation}
for some $C$ depending only on $\d,\s$ and $\om$.
\end{enumerate}
\end{lem}
\begin{proof}
This was performed in \cite[Proposition 7.3, Step 3]{petrache2014next} with a more restrictive assumption concerning the behavior of the equilibrium density near the boundary, namely that
$$
\cc \dist (x, \partial \Sigma)^\alpha \le \meseq(x) \leq \cC \dist(x, \partial \Sigma)^\alpha, 
$$
with the constraint
$$
0< \kappa  \leq 1, \qquad 0\le \alpha \le \frac{2\kappa \d}{2\d-\s}.
$$
This assumption includes the case of the semi-circular law in $\d = 1$, of the circular law in $\d = 2$, and all the cases where the density does not vanish on $\Sigma$ (hence, most of the interesting $\LogD$ cases). However, we would like to cover \textit{critical cases} in $\LogU$ where the density may vanish faster than a square root at the edges, or somewhere in the bulk, and they require a more general assumption near the boundary, as in \ref{H5}. In Section \ref{preuve:finirconstruction}, we explain how to generalize the construction of \cite{petrache2014next} to Assumption \ref{H5}.
\end{proof}

To complete the proof of  Proposition~\ref{quasicontinuite} we need to connect the preceding construction with a large enough volume under $\QNbeta$. The probability that $\QNbeta$ has all its points in $\Sigma$ is bounded below as follows
\begin{equation} \label{borneinfcondi}
\liminf_{N \ti}  \frac{1}{N} \log \QN \left( \left\lbrace \XN \in \Sigma^N \right\rbrace \right) \geq \log \frac{|\Sigma|}{|\omega|}.
\end{equation}
This is an easy consequence from the definition \eqref{def:QNbeta} of $\QNbeta$, of $\Sigma$ and $\omega$ as in \eqref{defomega}, and of \eqref{asymptotiquezetaNbeta}. Combining the constant $\log \frac{|\Sigma|}{|\omega|}$ with the one of \eqref{defcm}, which  converges to $|\Sigma|-1- \log |\Sigma|$ as $\um \to 0$,  yields the result.

\section{Large deviations for the reference measure} \label{preuvesanov}
In this section, we prove Proposition \ref{SanovbQN}, as well as Lemma \ref{SanovbQN2}. Proposition \ref{SanovbQN} is a \textit{process-level} (or \textit{type~3}) large deviation principle, whereas Lemma \ref{SanovbQN2} is rather Sanov-like (or \textit{type~2}). The proof of Proposition \ref{SanovbQN} relies on a similar result from \cite{georgiz} where the Poisson process $\Poisson^1$ is used as the reference measure instead of $\QNbeta$ (as defined in Section \ref{sec:empwihtoutint}). On the other hand, Lemma \ref{SanovbQN2} relies on the classical Sanov theorem with some adaption to our setting. We believe that at least parts of these (mostly technical) variations around classical results belong to folklore knowledge within the community of Gibbs point processes, but we provide a proof for the sake of completeness.

First, let us say a word about the topology. Large deviation principles for empirical fields (in the \textit{non-interacting case}) hold with a stronger topology on $\config$, called the $\tau$-topology (it can be described as the initial topology associated to the maps $\C \mapsto f(\C)$ for  all bounded measurable functions $f$ which are local in the sense of \eqref{funlocal}), see e.g. \cite{Georgii1} or \cite{seppalainen}. Although we expect both the intermediate result of Proposition \ref{SanovbQN} and our main LDP to hold for the $\tau$-topology, with essentially the same proof, we do not pursue this generality here.
% However, let us emphasize that even when restating Proposition \ref{SanovbQN} in the $\tau$-topology, our LDP (in the \textit{interacting case}) does not follow from an easy application of Varadhan's integral lemma, because the interaction is neither bounded (it needs to be truncated) nor local (that is why we use the screening procedure). \cm{deja dit}

\subsection{Two comparison lemmas}
We start by introducing a notion that allows to replace point processes by equivalent ones (as far as LDP's are concerned). 
\begin{defi} \label{defeas} Let $(X, d_X)$ be a metric space and let $\{R_N\}_N$ and $\{R'_N\}_N$ be two coupled sequences of random variables with values in $X$, defined on some probability spaces $\{(\Omega_N, \mc{B}_N, \pi_N)\}_N$. For any $\delta > 0$ we say that $\{R_N\}_N$ and $\{R'_N\}_N$ are \textit{eventually almost surely (e.a.s.) $\delta$-close} when for $N$ large enough we have
\begin{equation*}
d_X(R_N, R'_N) \leq \delta, \pi_N\text{-a.s.}
\end{equation*}
It two sequences are e.a.s. $\delta$-close for any $\delta >0$ we say that they are \textit{eventually almost surely equivalent (e.a.s.e.)}.
\end{defi}
Let us emphasize that being eventually almost surely equivalent is much stronger than the usual convergence in probability. It is also easily seen to be strictly stronger than the classical notion of \textit{exponential equivalence} (see \cite[Section 4.2.2]{dz}) and thanks to that, large deviation principles may be transfered from one sequence to the other.

\begin{lem} \label{transfertLDP} If the sequences $\{R_N\}_N$ and $\{R'_N\}_N$ are eventually almost surely equivalent and  an LDP with good rate function holds for the law of $\{R_N\}_N$, then the same LDP holds for the law $\{R'_N\}_N$.
\end{lem}
\begin{proof}
It is a straightforward consequence of \cite[Theorem 4.2.13]{dz}.
\end{proof}

A first example is given by the averages of a configuration over (translations in) two close sequences of sets. In what follows, $\triangle$ denotes the symmetric difference. 
\begin{lem} \label{lem:reductionvolume} Let $\{V_N\}_N,\{W_N\}_N$ be two sequences of Borel sets in $\Rd$ of bounded Lebesgue measure, such that
\begin{equation*}
\lim_{N \ti} \frac{|W_N \triangle V_N|}{|W_N|} = 0.
\end{equation*}
Let $f$ be a bounded measurable function on $\config$ and let $\Pst$ be in $\probas(\config)$. Then the random variables obtained as the push-forward of $P$ by the (“empirical fields”-like) maps
\begin{equation*}
\mathcal{C} \mapsto \frac{1}{|W_N|} \int_{W_N} \delta_{\theta_{x} \cdot \mathcal{C}} \, dx \text{ and }\, \mathcal{C} \mapsto \frac{1}{|V_N|} \int_{V_N} \delta_{\theta_{x} \cdot \mathcal{C}}  \,  dx
\end{equation*}
are eventually almost surely equivalent in the sense of Definition \ref{defeas}.
\end{lem}
\begin{proof}
For any $\mathcal{C} \in \config$ we have
\begin{equation} \label{reductionvolume}
\frac{1}{|W_N|} \left|\int_{W_N} f(\theta_{x} \cdot \mathcal{C}) dx - \int_{V_N} f(\theta_{x} \cdot \mathcal{C}) dx\right| \leq \frac{|W_N \triangle V_N|}{|W_N|} \|f\|_{\infty}.
\end{equation}
For any $\delta > 0$, to get e.a.s. $\delta$-closeness it suffices to recall that the distance between point processes is defined by testing against functions in $\Lip(\config)$ (which are in particular bounded in sup-norm).
\end{proof} 

The following lemma shows that when considering “empirical fields”-like random variables, obtained by averaging a configuration over translations in some large domain, we may restrict the configuration to this domain. 
\begin{lem} \label{LLDPinduit}
Let $\{\Lambda_N\}_N$ be a sequence of Borel sets of $\Rd$ of finite Lebesgue measure, such that 
\begin{equation} \label{largeportion}
\forall k \in \mathbb{N}, \quad \lim_{N \ti} \frac{1}{|\Lambda_N|} \left|\{ x \in \Lambda_N, d(x, \partial  \Lambda_N) \geq k \}\right| = 0.
\end{equation}
Let $P$ be in $\probas(\config)$ and let us denote by $R_N$, resp. $R_N'$ the push-forward of $P$ by the maps 
\begin{equation*}
\mathcal{C} \mapsto \frac{1}{|\Lambda_N|} \int_{\Lambda_N} \delta_{\theta_x \cdot \mc{C}} dx,  \text{ resp. } \mathcal{C} \mapsto \frac{1}{|\Lambda_N|} \int_{\Lambda_N} \delta_{\theta_x \cdot \left(\mc{C} \cap \Lambda_N \right)}dx.
\end{equation*}
Then the sequences $\{R_N\}_N$ and $\{R'_N\}_N$ are e.a.s.e. in the sense of Definition \ref{defeas}.
\end{lem}
We may observe that \eqref{largeportion} holds in particular when $\Lambda_N$ is taken to be $N^{1/\d} \Lambda$ for some compact $\Lambda$ with piecewise $C^1$ boundary.
\begin{proof}
For any $k \geq 1$, we have
\begin{equation*}
\left(\theta_x \cdot \mathcal{C}\right) \cap \carr_k =  \left(\theta_{x} \cdot (\mathcal{C} \cap \Lambda_N))\right) \cap \carr_k
\text{ for all $x$ such that $d(x, \partial  \Lambda_N) \geq k^{1/{\d}}$.} 
\end{equation*} 
Thus if $f$ is in $\Loc_k(\config)$ (as defined in \eqref{funlocal}) we have
\begin{multline*}
\Esp_{P} \left[ \frac{1}{|\Lambda_N|} \int_{\Lambda_N} f(\theta_x \cdot \mc{C}) dx\right] - \Esp_{P} \left[ \frac{1}{|\Lambda_N|} \int_{\Lambda_N} \delta_{\theta_x \cdot \left(\mc{C} \cap \Lambda_N \right)}dx\right]  \\ \leq \|f\|_{L^{\infty}} \left(1 - \frac{1}{|\Lambda_N|} \left|\{ x \in \Lambda_N, d(x, \partial  \Lambda_N) \geq k^{1/\d} \}\right|\right).
\end{multline*}
Using the assumption \eqref{largeportion} we see that the right-hand side goes to $0$ as $N \ti$, uniformly for $f \in \Loc_k(\config)$ such that $\|f\|_{L^\infty} \leq 2$. 

Combining the uniform approximation of functions in $\Lip(\config)$ by bounded local functions (as in Lemma \ref{Ldistconfig}) and the definition of $d_{\probas(\config)}$ as testing against $\Lip(\config)$, we obtain that for any $\delta > 0$ there exists $k \geq 1$ such that $P$-a.s.
\begin{equation} \label{LDPinduit}
d_{\probas(\config)}(R_N, R_N') \leq \delta + o_{N,k}(1),  
\end{equation}
where $o_{N, k}(1)$ goes to $0$ as $N \ti$ (for fixed $k$). Hence $R_N, R'_N$ are e.a.s. $2\delta$-close,  and this holds for any $\delta > 0$, hence $\{R_N\}_N$ and $\{R'_N\}_N$ are e.a.s.e. in the sense of Definition \ref{defeas}.
 \end{proof}

\subsection{Continuous average: proof of Proposition \ref{SanovbQN}}
Our starting point is the following known large deviation principle for empirical fields.
\begin{prop}[Georgii-Zessin] \label{Sanovconnu}  
Let $\{\Lambda_N\}_N$ be a fixed sequence of cubes increasing to $\Rd$ and let $R_N$ be the push-forward of $\Poisson^1$ by the map
$$
\mathcal{C} \mapsto \frac{1}{|\Lambda_N|} \int_{\Lambda_N} \delta_{\theta_x \cdot \mc{C}} dx.
$$
Then $\{R_N\}_N$ satisfies a large deviation principle at speed $|\Lambda_N|$ with rate function $\ERS[\cdot |\Poisson^1]$.
\end{prop}
\begin{proof}
It is a consequence of \cite[Theorem 3.1]{georgiz} together with \cite[Remark 2.4]{georgiz} to get rid of the periodization used in their definition of $R_N$ (see also \cite{FollmerOrey}). One could also use the method of \cite[Chapter 6]{seppalainen}, where the G\"{a}rtner-Ellis theorem is used through proving the existence of a pressure and studying its Legendre-Fenchel transform, by adapting the proof from the discrete case (point processes on $\Z^d$) to the case of point processes on $\Rd$. 
\end{proof}

\subsubsection{Extension to general shapes} 
In this first step we extend the LDP of Proposition \ref{Sanovconnu} to more general shapes of $\{\Lambda_N\}_N$. 
\begin{lem}\label{Sanovdomaine} Let $\Lambda$ be a compact set of $\Rd$ with a non-empty interior and a Lipschitz boundary, and let $\Lambda_N := N^{1/{\d}} \Lambda$. Let $R_N$ be the push-forward of $\Poisson^1$ by the map
\begin{equation*}
\mathcal{C} \mapsto \frac{1}{|\Lambda_N|} \int_{\Lambda_N} \delta_{\theta_x \cdot \mc{C}}\,  dx.
\end{equation*}
Then $\{R_N\}_N$ satisfies a large deviation principle at speed $N |\Lambda|$ with rate function $\ERS[\cdot |\Poisson^1]$.
\end{lem}
\begin{proof} 
The idea is to tile $\Lambda_N$ into large hypercubes (up to some boundary part whose volume will be negligible) and to apply  Proposition \ref{Sanovconnu} on each hypercube, together with Lemma \ref{LLDPinduit}.

For any $r > 0$ let us consider the hypercubes centered at the points of $\Lambda \cap \frac{1}{r} \Z^\d$ and of sidelength $\frac{1}{r}$. We remove those that are centered at points in 
\begin{equation*}
A_r:=\left\lbrace x \in \Lambda \cap \frac{1}{r} \Z^\d, \dist(x, \partial  \Lambda) \leq 2 C r^{-1/{\d}}\right\rbrace,
\end{equation*}
 where $C$ is the distance between the center of the unit hypercube in dimension $\d$ and any vertex of this hypercube.
Since the boundary of $\Lambda$ is Lipschitz, we have 
 \begin{equation} \label{volumelostAn}
 \lim_{r \ti} |A_r| = 0.
 \end{equation}
 
 Now, let $N \geq 1$. From \eqref{volumelostAn} we see that for $r$ large enough (depending on $N$), the volume lost when removing the boundary hypercubes (with center in $A_r$) is less than $2^{-N} |\Lambda|$. We may thus construct a family of $m = m(N)$ hypercubes $\{\Lambda^{(i, N)}\}_{i=1}^{m(N)}$ included in $\Lambda$ and such that 
\begin{equation} \label{volumelostAn2}
\frac{1}{|\Lambda|} |\Lambda \triangle \left(\cup_{i=1}^m \Lambda^{(i,N)}\right)| \leq 2^{-N}.
\end{equation}
Let us also observe that we may also choose $r$ arbitrarily large such that the following technical point is satisfied
\begin{equation} \label{technicality}
\sum_{i=1}^{m} |\Lambda^{(i,N)}| \in r \N^{\d},
\end{equation}
as follows e.g. from an elementary argument based on the intermediate value theorem.

Next, we define $\tilde{\Lambda}^{(N)}$ as the hypercube of center $0$ and such that
\begin{equation}
\label{goodtiling}
|\tilde{\Lambda}^{(N)}| = \sum_{i=1}^{m} |\Lambda^{(i,N)}|.
\end{equation}
Since \eqref{technicality} holds, we may split $\tilde{\Lambda}^{(N)}$ into hypercubes of sidelength $r$, and there exists a measurable bijection $\Phi_{N} : \cup_{i=1}^{m} \Lambda^{(i,N)} \rightarrow \tilde{\Lambda}^{(N)}$ which is a translation on each hypercube $\Lambda^{(i,N)}$ ($i=1, \dots, m$). 

As defined above, $R_N$ is the push-forward of $\Poisson^1$ by the map
\begin{equation*}
\mathcal{C} \mapsto \frac{1}{|\Lambda_N|} \int_{\Lambda_N} \delta_{\theta_x \cdot \mc{C}}\, dx.
\end{equation*}
We also introduce $R'_N$ as the push-forward of $\Poisson^1$ by the map 
\begin{equation*}
\mathcal{C} \mapsto \frac{1}{ Nm |\Lambda^{(1,N)}|} \int_{\cup_{i=1}^{m} N^{1/{\d}} \Lambda^{(i,N)}} \delta_{\theta_x \cdot \mc{C}} \, dx.
\end{equation*}
Finally, from any configuration of points $\mathcal{C}$ in $\cup_{i=1}^{m} N^{1/{\d}} \Lambda^{(i,N)}$ we get by applying $x \mapsto N^{1/{\d}} \Phi_{N}(N^{-1/{\d}}(x))$ a configuration in $N^{1/{\d}} \tilde{\Lambda}^{(N)}$, which by abusing notation we denote again by $\Phi_{N}(\mathcal{C})$. We denote by $R''_N$ the push-forward of $\Poisson^1_{|\Lambda_N}$ by:
$$
\mathcal{C} \mapsto \frac{1}{N|\tilde{\Lambda}^{(N)}|} \int_{\Lambda_N} \delta_{\theta_{\Phi_{N}(x)} \cdot \Phi_{N}(\mathcal{C})} dx.
$$
We impose that the random variables $R_N, R'_N, R''_N$ are coupled together the natural way.

It is easily seen that the push-forward of $\Poisson^1_{|\Lambda_N}$ -- or more precisely  of the process induced on the subset $\cup_{i=1}^{m} N^{1/{\d}} \Lambda^{(i,N)}$ -- by the map $\mc{C} \mapsto \Phi_{N}(\mc{C})$ is equal in law to $\Poisson^1_{N^{1/{\d}}\tilde{\Lambda}^{(N)}}$. The sequence of hypercubes $\{N^{1/{\d}}\tilde{\Lambda}^{(N)}\}_N$ satisfies the hypothesis of Proposition \ref{Sanovconnu} hence a LDP holds at speed $|\Lambda| N$ for the sequence $\{R''_N\}_N$ (the fact that we consider the push-forward of $\Poisson^1_{|\Lambda_N}$ instead of that of $\Poisson^1$ is irrelevant thanks to Lemma \ref{LLDPinduit}). To show that the same principle holds for $\{R_N\}_N$ it is enough to show that the two sequences are e.a.s. equivalent in the sense of Definition \ref{defeas}.

Let us first observe that the sequences $\{R_N\}_N$ and $\{R'_N\}$ are e.a.s.e. because as a consequence of \eqref{goodtiling} the tiling of $\Lambda_N$ by the hypercubes $\cup_{i=1}^{m} N^{1/{\d}} \Lambda^{(i,N)}$ only misses a $o(1)$ fraction of the volume of $\Lambda_N$ and e.a.s. equivalence is then a consequence of Lemma  \ref{lem:reductionvolume}.

As for the pair of sequences $\{R'_N\}_N$ and $\{R''_N\}_N$, let us observe that for any $k \geq 1$ we have
 $$
\left( \theta_x \cdot \mathcal{C} \right) \cap \carr_k = \left( \theta_{\Phi_{N}(x)} \cdot \Phi_{N}(\mathcal{C}) \right) \cap \carr_k
 $$
 for any $x$ in one of the tiling hypercubes $\cup_{i=1}^{m} N^{1/{\d}} \Lambda^{(i,N)}$ except for the points that are near the boundary of their hypercube - those such that 
 $$\dist \left(x, \cup_{i=1}^{m} \partial  N^{1/{\d}} \Lambda^{(i,N)} \right) \leq |\carr_k|^{1/{\d}}.$$
For any $k$ the fraction of volume of points in the hypercube that are close to the boundary in the previous sense is negligible as $N \ti$. Arguing as in the proof of Lemma \ref{LLDPinduit} gives the result.
\end{proof}

\subsubsection{Tagged point processes}\label{sectagged}
We now recast the result of Lemma \ref{Sanovdomaine} in the context of tagged point processes (as defined in Section \ref{sec-ppp}) which necessitates to replace the specific relative entropy $\ERS$ by its analogue with tags.
\begin{lem} \label{SanovStep1} Let $\Lambda$ be a compact set of $\Rd$ with $C^1$ boundary and non-empty interior and let $\bar{R}_N$ be the push-forward of $\Poisson^1$ by the map
\begin{equation*}
\mathcal{C} \mapsto \frac{1}{|\Lambda|} \int_{\Lambda} \delta_{(x,\theta_{N^{1/{\d}} x} \cdot \mc{C})}\, dx.
\end{equation*}
Then $\{\bar{R}_N\}_N$ satisfies a large deviation principle at speed $N$ with rate function $\bERS[\bP|\Poisson^1]$.\end{lem}
Let us recall that the quantity $\bERS[\bP | \Poisson^1]$ has been defined in \eqref{def:bERS}.
\begin{proof}
\textbf{Upper bound.} Let $\bPst$ be a stationary tagged point process. We claim that
\begin{equation}\label{MarkUB}
\limsup_{\epsilon \t0} \limsup_{N\ti} \frac{1}{N} \log \bR_N \left( B(\bPst, \epsilon) \right) \leq - \int_{\Lambda} \ERS[\bPst^x | \Poisson^1]dx.
\end{equation}
Let us observe that the \textit{forgetful map} $\varphi : \probas(\Lambda \times \config) \rightarrow \probas(\config)$ obtained by pushing forward with the map $(x, \mc{C}) \mapsto \mc{C}$ is continuous. This yields 
\begin{equation} 
\label{compoubli}
\limsup_{\epsilon \t0} \limsup_{N\ti} \frac{1}{N} \log \bR_N \left( B(\bPst, \epsilon) \right) \leq \limsup_{\epsilon \t0} \limsup_{N\ti} \frac{1}{N} \log \varphi \sharp \bR_N \left( B(\varphi(\bPst), \epsilon)\right),
\end{equation}
where $\varphi \sharp \bR_N $ is the push-forward of $\bR_N$ by $\varphi$. By definition, we may observe that this coincides with the push-forward of $\Poisson^1$ by the map
\begin{equation*}
\mathcal{C} \mapsto \frac{1}{N|\Lambda|} \int_{\Lambda_N} \delta_{\theta_x \cdot \mc{C}} \, dx.
\end{equation*}
 From Lemma \ref{Sanovdomaine} we know that an LDP holds for $R_N$ at speed $|\Lambda| N$ with rate function $\ERS[\cdot|\Poisson^1]$ (or equivalently at speed $N$ with rate function $|\Lambda| \ERS[\cdot | \Poisson^1]$) hence the right-hand side of \eqref{compoubli} is bounded by $|\Lambda|\ERS[\varphi(\bPst)|\Poisson^1]$. 
 
Now, since the relative specific entropy is affine, we have
\begin{equation*}
\ERS[\varphi(\bPst)|\Poisson^1] = \frac{1}{|\Lambda|} \int_{\Lambda} \ERS[\bPst^x|\Poisson^1] \, dx
\end{equation*}
 which yields the LDP upper bound.

\textbf{Lower bound.} Let $\bPst$ be a tagged point process. We want to prove that
\begin{equation}\label{MarkLB}
\liminf_{\epsilon \t0} \liminf_{N\ti} \frac{1}{N} \log \bR_N \left( B(\bPst, \epsilon)\right) \geq -  \int_{\Lambda} \ERS[\bPst^x | \Poisson^1] dx.
\end{equation}

Let $\epsilon > 0$. By a standard approximation argument, we see that there exists $M \geq 1$ and $k > 0$ (both depending only on $\epsilon$) such that for any function $F \in \Lip(\Lambda \times \config)$, there exists 
\begin{itemize}
\item a covering of $\Lambda$ by compact sets\footnote{This covering can actually be chosen independently of $F$ in $\Lip(\Lambda \times \config)$.} $A_1, \dots, A_M \subset \Lambda$ of pairwise disjoint non-empty interiors  such that each set $A_i$ has a Lipschitz boundary,
\item for any $i \in \{1, \dots, M\}$, a point $x_i$ in $A_i$,
\item a family $\{F_i\}_{i =1, \dots, M}$ of functions in $\Lip(\config) \cap \Loc_k(\config)$,
\end{itemize}
such that
\begin{equation*}
\|F(x,\C) - \sum_{i=1}^M \mathbf{1}_{A_i}(x) F_i(\C) \|_{L^{\infty}(\Lambda \times \config)} \leq \epsilon.
\end{equation*}
The fact that the $F_i$'s can be taken in $\Loc_k$ (with $k$ depending only on $\epsilon$) follows from the approximation result of Lemma \ref{Ldistconfig} (in fact we can take $F_i = F(x_i, \cdot \cap \carr_k)$ for $k$ large enough).

For any $\bar{Q}$ in $\probas(\Lambda \times \config)$, and $F \in \Lip(\Lambda \times \config)$ we now have
\begin{equation*}
\left| \int F(x, \C) d\bar{Q}(x,\C) - \sum_{i=1}^N \int_{x \in A_i} F_i(\C) d\bar{Q}(x,\C) \right| \leq C \epsilon, 
\end{equation*}
for some $C$ independent on $\epsilon, N$. 

%In particular, if $Q_i \in \probas(\config)$ is such that for any $G$ Lipschitz
%\begin{equation} \label{QdansAi}
%\left| \int F_i(\C) dQ(\C) - \int_{x \in A_i} F_i(\C) d\bP(x, \C) \right| \leq \epsilon, 
%\end{equation}
%then forming 
%\begin{equation*}
%\bar{Q}:= \sum_{i=1}^M \frac{1}{|A_i|} \int_{A_i \times \config} \delta_{(x, C)} \, dx \, dQ(\C),
%\end{equation*}
%we obtain $\bar{Q} \in B(\bP, 2C \epsilon))$.

For each $i$ we define $P_i \in \probas(\config)$ as
\begin{equation*}
 P_i := \frac{1}{|A_i|} \int_{x \in A_i} \delta_{\C}\, d\bP(x, \C).
 \end{equation*}
 We may observe that
\begin{equation*}
\ERS[P_i | \Poisson^1] = \frac{1}{|A_i|} \int_{x \in A_i} \ERS[\bP^{x} | \Poisson^1].
\end{equation*}
Applying the lower bound of Lemma \ref{Sanovdomaine} on $B(P_i, \epsilon)$, we see that
$$
\lim_{N \ti} \frac{1}{N} \log R^{(i)}_N(B(P_i(\epsilon)) \geq - \frac{1}{|A_i|} \int_{x \in A_i} \ERS[\bP^{x} | \Poisson^1], 
$$
where $R^{(i)}_N$ is the empirical field of $\Poisson^1$ with average on $A_i$. We can paste such empirical fields together to obtain an empirical field with average on $\Lambda$ which is close to $\bP$, up to a negligible error (arguing as in Lemma \ref{lem:reductionvolume}, Lemma \ref{LLDPinduit}). Combining the LDP lower bounds we see that 
\begin{equation*}
\lim_{N \ti} \log \bR_N \left( B(\bPst, \epsilon)\right) \geq  - \sum_{i=1}^M |A_i| \ERS[P_i | \Poisson^1] + o_{\epsilon}(1).
\end{equation*}
Since $ - \sum_{i=1}^M |A_i| \ERS[P_i | \Poisson^1] = \int_{x \in \Lambda} \ERS[\bP^x|\Poisson^1] dx$ coincides with the definition of $\bERS[\bP|\Poisson^1]$, the lower bound follows by taking $\epsilon \to 0$.

\textbf{Conclusion.}
From \eqref{MarkUB} and \eqref{MarkLB} we get a weak LDP for the sequence $\{\bR_N\}_N$. The full LDP is obtained by observing that $\{\bR_N\}_N$ is exponentially tight, a fact for which we only sketch the (elementary) proof : for any integer $M$ we may find an integer $T(M)$ large enough such that a point process has less than $T(M)$ points in $\carr_M$ expect for a fraction $\leq \frac{1}{M}$ of the configurations, with $\bR_N$-probability bounded below (when $N \ti$) by $1 - e^{-NM}$. The union (on $N \geq N_0$ large enough) of such events has a large $\bR_N$-probability (bounded below by $1 - e^{-NM}$ when $N \ti$) and is easily seen to be compact.
\end{proof}

\subsubsection{From Poisson to Bernoulli}\label{secstep3}
In what follows, when $\Lambda$ is fixed, for $N \geq 1$ we let $\B_N$ be the Bernoulli point process with $N$ points in $N^{1/{\d}} \Lambda$ (in other words, it is the law of $N$ points chosen uniformly and independently in $N^{1/{\d}} \Lambda$).

If we replace $\Poisson^1$ by $\B_N$ as the reference measure i.e. if we constrain $\Poisson^1$ to have a fixed number of points in $N^{1/{\d}} \Lambda$ then the LDP for (tagged) empirical fields is modified: the large deviation upper bound holds but the large deviation lower bound needs a technical adaption. Let us recall that we denote by $\probas_{s,1}(\Lambda \times \config)$ the set of stationary tagged point processes such that the integral on $x \in \Lambda$ of the intensity of the disintegration measure $\bPst^x$ is $1$.

\begin{lem} \label{SanovBernou} 
Let $\Lambda$ be a compact set of $\Rd$ with $C^1$ boundary and non-empty interior and let $\bS_N$ be the push-forward of $\B_N$ by the map
\begin{equation*}
\mathcal{C} \mapsto \frac{1}{N|\Lambda|} \int_{N^{1/{\d}}\Lambda} \delta_{(N^{-1/{\d}}x,\theta_{x} \cdot \mc{C})} dx.
\end{equation*}
Then for any $A \subset \probas_s(\Lambda \times \config)$ we have:
\begin{multline} \label{EqSB5}
- \inf_{\bPst \in \mathring{A} \cap \probas_{s,1}(\Lambda \times \config)} \bERS[\bP | \Poisson^1]  - (\log |\Lambda| - |\Lambda| + 1) \leq \liminf_{N \ti} \frac{1}{N} \log \bS_N (A) \\ \leq  \limsup_{N \ti} \frac{1}{N} \log \bS_N (A) \leq - \inf_{\bPst \in \bar{A}} \bERS[\bP | \Poisson^1] - (\log |\Lambda| - |\Lambda| + 1).
\end{multline}
\end{lem}

\begin{proof}
\textit{Step 1.} \textbf{Upper bound.} \\
The upper bound of \eqref{EqSB5} follows from the LDP upper bound of Lemma \ref{SanovStep1}. The conditional expectation of $\Poisson^1$ conditioned into having $N$ points in $N^{1/{\d}} \Lambda$, is equal to the law of a Bernoulli point process. In particular, for any $A \subset \probas_s(\Lambda \times \config)$ we have (with $\bR_N$ as in Lemma \ref{SanovStep1})
\begin{multline}
\frac{1}{N} \log \bR_N (A) \geq \frac{1}{N} \log \bS_N(A) + \frac{1}{N} \log \Poisson^1 \left( \text{$N$ points in $N^{1/{\d}} \Lambda$} \right) \\ = \frac{1}{N} \log \bS_N (A) + \frac{1}{N} \log \left(\exp\left(-N|\Lambda|\right) \frac{1}{N!} (N|\Lambda|)^N\right) \\ = \frac{1}{N} \log \bS_N (A) + (\log |\Lambda| - |\Lambda| + 1) + o_{N \ti}(1), 
\end{multline}
which yields the upper bound in \eqref{EqSB5}.

\textit{Step 2.} \textbf{Lower bound.} \\
For simplicity, we will prove a weak lower bound\footnote{The general LDP lower bound follows from a similar argument or by using exponential tightness.}, namely for any $\bP$ in $\probas_{s,1}(\Lambda \times \config)$ we claim that:
\begin{equation}
\label{weaklowerboundbernou} 
\liminf_{\epsilon \to 0} \liminf_{N \ti} \frac{1}{N} \log \bS_N (B(\bP, \epsilon))  \geq - \bERS[\bP | \Poisson^1] - (\log |\Lambda| - |\Lambda| + 1).
\end{equation}
We want to use the LDP lower bound of Lemma \ref{SanovStep1}, but we have to argue that (enough) configurations $\C$ in $N^{1/\d} \Lambda$, for which the tagged empirical field is close from $\bP$ as in \eqref{MarkLB}, have approximately $N$ points (then the constant $- (\log |\Lambda| - |\Lambda| + 1)$ comes from normalizing the probability measure).

We sketch the argument here. Let $\chi$ be a smooth, non-negative function with compact support in the unit ball of $\Rd$, and such that $\int \chi = 1$, we denote by $\tilde{\chi}$ the map $\C \mapsto \int_{\Rd} \chi\, d\C$. We have
\begin{equation*}
\Esp_{\bP}[\tilde{\chi}] = \Esp_{\bP} [\Numb_1] = 1,
\end{equation*}
and by continuity, if $\bar{Q}$ is close from $\bP$ we must have $\Esp_{\bar{Q}} [ \tilde{\chi} ] \approx 1$. But on the other hand, if $\bar{Q}$ is an empirical field obtained by averaging over translations in some domain, $\Esp_{\bar{Q}} [ \tilde{\chi} ]$ detects the average number of points per unit volume in the domain, thus the associated realization of the Poisson process has $N$ points. Three points must be precised in order to make the previous argument rigorous:
\begin{itemize}
\item First, the map $\C \mapsto \tilde{\chi}(\C)$ is continuous but not bounded (it can be of the same order as the number of points in the unit ball), hence cannot be tested against weakly converging point processes. 
\item Next, if $\bar{Q}$ is an empirical field, then $\Esp_{\bar{Q}} [ \tilde{\chi} ]$ is slightly less than the average density due to some boundary effect.
\item Finally, we may only deduce that the associated realization of the Poisson process has $N(1+o_N(1))$ points.
\end{itemize}
For the first issue, for any $M > 0$ we may replace $\tilde{\chi}$ by $\tilde{\chi} \wedge M$ (which can now be tested against weak convergence). Using elementary properties of the Poisson point process, we see that the errors due to this truncation become negligible as $M \ti$, up to an event of negligible probability. The second point can be dealt with by observing that the number of boundary points is negligible (with respect to $N$), again up to an event of negligible probability. Finally, once we have a realization with $N(1+o_N(1))$ points, we may add or delete $No_N(1)$ points without changing the empirical field too much.
\end{proof}

\subsubsection{From Bernoulli to imperfect confinement} \label{sec:preuveSanovbQNoui}
Finally, we extend the previous results to the case where the points are sampled according to $\QNbeta$ (as defined in \eqref{def:QNbeta}) and not a Bernoulli process. Let us recall that $\QNbeta$ has a constant density on $\omega^N$ ($\zeta$ vanishes on $\omega$, see definitions \eqref{defzeta}, \eqref{defomega}), and that its marginal density tends to zero as $\exp(-\beta N \zeta(x))$ outside $\omega$. Therefore we expect $\QNbeta$ to behave like a Bernoulli point process with $N$ points on $N^{1/{\d}} \omega$ (which would correspond to a “perfect confinement” where $\zeta = + \infty$ outside $\omega$). 

We may now turn to proving Proposition \ref{SanovbQN}.
\begin{proof} 
\textit{Step 1.} \textbf{Lower bound.} \\
The probability under $\QNbeta$ that all the points fall inside $N^{1/{\d}}\Sigma$ is equal to
\begin{equation*}
\left( \frac{|\Sigma|}{ \int_{\R^{\d} } \exp\left(-\beta N^{1-\frac{\s}{\d}} \zeta(x) \right)\, dx} \right)^N.
\end{equation*}
Using the dominated convergence theorem, we see that 
\begin{equation} \label{pointshorsuppb}
\lim_{N \ti} \frac{1}{N} \log \QNbeta\left( \text{$N$ points in $N^{1/\d} \Sigma$} \right) = \log |\Sigma| - \log |\omega|.
\end{equation}

The conditional expectation with respect to this event is a Bernoulli point process with $N$ points in $N^{1/\d} \Sigma$. Using the LDP lower bound of Lemma \ref{SanovBernou} together with \eqref{pointshorsuppb} we obtain the lower bound for Proposition \ref{SanovbQN}.

\textit{Step 2.} \textbf{Upper bound.} \\
The law of total probabilities yields, for any $A \subset \probas_s(\Sigma \times \config)$ (denoting by $\# \Sigma_{N}$ the number of points in $N^{1/\d} \Sigma$):
\begin{equation*}
\limsup_{N \ti} \frac{1}{N} \log \bQpN(A) \leq \limsup_{N \ti} \frac{1}{N} \log \sum_{k = 0}^N \bQpN(A \cap  \# \Sigma_{N} = k) \QNbeta( \# \Sigma_{N} = k).
\end{equation*}
Conditionally to the event $\{\# \Sigma_{N} = k\}$ the law of $\QNbeta$ is equal to that of a  Bernoulli point process with $k \leq N$ points in $N^{1/{\d}} \Sigma$, and the LDP upper bound of Lemma \ref{SanovBernou}  allows us to bound each term.

More precisely it is easy to see that with overwhelming probability the number of points $\# \Sigma_N$ tends to infinity (e.g.  as $\sqrt{N}$) so that we may bound 
\begin{multline*}
\frac{1}{N} \log \sum_{k = 0}^N \bQpN(A \cap  \# \Sigma_{N} = k) \QNbeta( \# \Sigma_{N} = k)\\ \leq \frac{1}{N} \log  \sum_{k = \sqrt{N}}^N \bQpN(A \cap  \# \Sigma_{N} = k) \QNbeta( \# \Sigma_{N} = k).
\end{multline*}
Bounding $\QNbeta( \# \Sigma_{N} = k)$ by $1$ and the terms $\frac{1}{N} \log \bQpN(A \cap  \# \Sigma_{N} = k)$ by $$\frac{1}{k} \log \bQpN(A \cap  \# \Sigma_{N} = k),$$ and using Lemma \ref{SanovBernou} we get the result.
\end{proof}

\subsection{Discrete average} \label{sec:preuvebQN2} \label{sec:preuvesanovdiscret}
In this section we give the proof of Lemma \ref{SanovbQN2}. The argument is analogous as the proof of the continuous case and we will only sketch it here.
\begin{proof} 
First, let us forget about the condition on the total number of points (i.e. we consider independent Poisson point processes) and about the tags (i.e. the coordinate in $\Sigmatil$), then there holds for any fixed $R$ a LDP for $\{\bM_{N,R}\}_N$ at speed $m_{N,R}$ with rate function $\Ent[ \cdot | \Poisson^1_{|\carr_R}]$. This is a consequence of the classical Sanov theorem (see\cite[Section 6.2]{dz}) since in this case the random variables $\theta_{x_i} \cdot \mathbf{C}_i$ are independent and identically distributed Poisson point processes on each hypercube. Taking the limit $R \ti$ yields, in view of the asymptotics \eqref{volumetiling} on $m_{N,R}$ and the definition \eqref{def:ERS} of the specific relative entropy,  
\begin{equation} \label{SanDiscPoisslim}
\lim_{R \ti} \liminf_{\nu \t0} \liminf_{N \ti} \frac{1}{N} \log \bM_{N,R}( B(\Pst_{\um|\carr_R}, \epsilon))  \geq - \ERS\left[\Pst_{\um} | \Poisson^1\right].
\end{equation}
We may then extend this LDP to the context of tagged point processes by following the same argument as in the continuous case.

Then, we argue as above (when passing from Poisson to Bernoulli). To condition the point process into having $\Ntil \approx N\mu_V(\Sigma_{\um})$ points in $\Sigmatil \approx \Sigma'_{\um}$ modifies the LDP lower bound obtained from Sanov's theorem by a quantity
$$
\limsup_{N \ti} \frac{1}{N} \log\left( e^{-N|\Sigma_{\um}|} \frac{\left(N|\Sigma_{\um}|\right)^{N\mu_V(\Sigma_{\um})}}{\left(N\mu_V(\Sigma_{\um})\right)!}\right),
$$
hence the constant $\cmuvum$ in \eqref{EqSB2}. This settles the first point of Lemma \ref{SanovbQN2}. 

The second point follows from the first one by elementary manipulations. The idea is that if one knows that a discrete average of large hypercubes is very close to some point process $\Pst$, then the continuous average of much smaller hypercubes is also close to $\Pst$ since it can be re-written using the discrete average up to a small error. More precisely for any fixed $\delta > 0$ establishing that a point process is in $B(\Pst, \delta)$ can be done by testing against local functions in $\Loc_k$ for some $k$ large enough. For $R,N$ large enough, an overwhelming majority of all translates of $\carr_k$ by a point in $\Sigma'_{\um}$ is included in one of the hypercubes $\barcarr_i$ ($i = 1 \dots m_{N,R}$) (this follows from the definitions and \eqref{volumetiling}).

For any such local function $f \in \Loc_k$ we have
\begin{equation} \label{discrtocont}
\frac{1}{|\Sigma_{\um}'|} \int_{\Sigma_{ \um}'} f(\theta_{x} \cdot \mc{C}) \approx \frac{1}{m_{N,R}} \sum_{i=1}^{m_{N,R}} \frac{1}{R^{\d}} \int_{\bar{\carr}_i} f(\theta_{x} \cdot \mc{C}) dx, 
\end{equation}
which allows us to pass from the assumption that the discrete average (in the right-hand side of \eqref{discrtocont}) of a configuration is close to $\Pst$ to the fact that the continuous average (in the left-hand side of \eqref{discrtocont}) is close to $\Pst$. These considerations are easily adapted to the situation of tagged point processes.
\end{proof}

\section{Additional proofs}\label{annex}
We collect here the proofs of various lemmas used in the course of the paper.
\subsection{Proof of Lemma \ref{Ldistconfig}}
\label{sec:preuveLdistconfig}
\begin{proof}
The first point ($\config$ is a Polish space) is well-known, see e.g. \cite[Proposition 9.1.IV]{dvj2}). 

It is easy to see that $\dconfig$ is a well-defined distance (the only point to check is the separation property). It is also clear that any sequence converging for $\dconfig$ converges for the topology on $\config$. Conversely, let $\{\mu_n\}_n$ be a sequence in $\config$ which converges vaguely to $\mu$ and let $\ep > 0$. There exists an integer $K$ such that $\sum_{k \geq K} \frac{1}{2^k} \leq \frac{\ep}{2}$ so we might restrict ourselves to the first $K$ terms in the series defining $\dconfig(\mu_n, \mu)$. 

For each $k =1, \dots, K$ and for any $n$, let $\mu_{n,k}$ (resp. $\mu_k$) be the restriction to the hypercube $\carr_k$ of $\mu_n$ (resp. of $\mu$). For any $k = 1, \dots, K$ the masses $(\mu_{n,k}(\carr_k))_{n \geq 1}$ are integers, and up to passing to a common subsequence (using a standard diagonal argument) we may assume that for each $k$ the sequence $\{\mu_{n,k}(\carr_k)\}_{n \geq 1}$ is either constant or diverging to $+ \infty$. The terms for which this could diverge give a negligible contribution to the distance.

We may then restrict ourselves to the terms $k$ for which $\mu_{n,k}$ is some constant $N_k$. By compactness we may then assume that the $N_k$ points of the configuration converge to some $N_k$-tuple $x_1, \dots, x_{N_k}$ of points in $\carr_k$. It is easy to see that $N_k$ must be equal to $\mu_k(\carr_k)$ and that the points $x_1 \dots x_{N_k}$ must correspond to the points of the configuration $\mu_k$. This implies the convergence in the sense of $\dconfig$. 

In conclusion, from any sequence $\{\mu_n\}_n$ converging to $\mu$ for the topology on $\config$, we may extract a subsequence which converges to $\mu$ in the sense of $\dconfig$, and conversely. Therefore $\dconfig$ is compatible with the topology on $\config$.

We now prove the approximation property. Let $F$ be in $\Lip(\config)$ and $\delta > 0$. From the definition \eqref{dconfig} of $d_{\config}$ we see that there exists $k$ such that if two configurations $\mc{C}, \mc{C}'$ coincide on $\carr_k$ then $d_{\config}(\mc{C}, \mc{C}') \leq \delta$. We let $f_k := F(\mc{C} \cap \carr_k)$. By definition $f$ is a local function in $\Loc_k$, we have chosen $k$ such that $d_{\config}(\mc{C}, \mc{C} \cap \carr_k) \leq \delta$ for any configuration $\mc{C}$ and since by assumption $F$ is $1$-Lipschitz we have
$$
|F(\mc{C}) - f(\mc{C})| = |F(\mc{C}) - F(\mc{C} \cap \carr_k)| \leq d_{\config}(\mc{C}, \mc{C} \cap \carr_k) \leq \delta,
$$
and $k$ here depends only on $\delta$, which concludes the proof of Lemma \ref{Ldistconfig}.
\end{proof}

\subsection{Proof of Lemma \ref{infatteint}}
\label{sec:proofinfatteint}
\begin{proof}
Let us denote as in Section \ref{sec:deflap} by $X=(x,y)$ the coordinates in $\R^{\d} \times \R^\k$. We also recall that $\gamma\in (-1,1)$.

Let $E_1$ and $E_2$ be elements of $\Elec_m$ compatible with the same configuration. Then we have $E_1- E_2= \nab u$ where $u$ solves $-\div (\yg \nab u)=0$. We can also observe that $\nab_x u$ (where $\nab_x$ denote the vector of derivatives in the $x$ directions only), is also a solution to the same equation (this should be understood component by component).
This is  a divergence form equation with a weight $\yg$ which belongs to the so-called Muckenhoupt class $A_2$. The result of \cite[Theorem 2.3.12] {KFS} then says 
that there exists $\lambda>0$ such that for $0<r<R$,
\begin{equation}\label{kfs}
\mathrm{osc}(\nab_x u, B(X, r)) \le C \left( \frac{1}{\int_{B(X,R)} \yg} \int_{B(X,R)} \yg |\nab_x u|^2 \right)^{1/2} (r/R)^{\lambda},
\end{equation} 
where $\mathrm{osc}(u, B(X,r))= \max_{B(X,r)}u -\min_{B(X,r)} u$.
We note that the condition that $\mc{W}(E_1, m) $ and $\mc{W}(E_2,m)$ are finite implies without difficulty that 
\begin{equation}\label{limsu}
\limsup_{R\to \infty}\frac{1}{R^{\d}} \int_{K_R\times \R^{\d} }\yg|\nab u|^2 <+\infty.
\end{equation} 
Applying \eqref{kfs} to $X$ which belongs to a fixed compact set, and inserting \eqref{limsu} we find that 
\begin{align*}
&\mathrm{osc}(\nab_x u, B(X,r))\leq C\left( R^{-(\d+1+\gamma)} R^{\d}\right)^{1/2} (r/R)^{\lambda} 
\quad \text{ for $\k =1$,} \\
&\mathrm{osc}(\nab_x u, B(X,r)) \le C \left( R^{-\d} R^{\d} \right)^{1/2} (r/R)^\lambda \quad \text{for $\k = 0$.}
\end{align*} 
In both cases, letting $R \to \infty$, we deduce that $\mathrm{osc}(\nab_x u, B(X,r))=0$, which means that $\nab_x u$ is constant on every compact set of $\R^{\d+\k}$.

 In the case $\k=0$, this concludes the proof that $u$ is affine, and then $E_1$ and $E_2$ differ by a constant vector.

In the case $\k=1$, this implies that $u$ is an affine function of $x$, for each given $y$. We may thus write $u(x,y)= a(y)\cdot  x+ b(y).$
Inserting into the equation $\div (\yg \nab u)=0$, we find that 
$\partial_y(\yg ( a'(y) x+b'(y))=0$, i.e. $a'(y)x+b'(y)= \frac{c(x)}{\yg}$.
But the fact that $\int_\R \yg |\partial_y u|^2\, dy $ is convergent implies that $\int \frac{c(x)^2}{\yg}\, dy$ must be, which  implies that $c(x)=0$ and thus $\partial_y u=0$. This means that $u(x,y)=f(x)$. But then again $\int \yg |\nab u|^2\, dy $ is convergent so we must have $\nab f(x)=0$  and $u$ is constant. Thus $E_1=E_2$ as claimed.

In the case $\k=1$, it follows that $\W(\mc{C}, m)$ (if it is finite) becomes an $\inf$ over a singleton, hence is achieved.

Let us now turn to the case $\k=0$ (i.e. \LogD \ or \Riesz \ with $\s = \d-2$). Let $E$ be in $\Elec(\C, m)$ such that $\mc{W}(E, m)$ is finite, and let $c$ be a constant vector in $\R^{\d}$,  then 
\begin{equation}\label{unifconti}
 \dashint_{\carr_R}  |E_\eta+c |^2 - m\c \g(\eta)=\dashint_{\carr_R} |E_\eta|^2- m \c  \g(\eta)+ |c|^2 + 2 c\cdot \dashint_{\carr_R}E_\eta.
 \end{equation}
We claim that $\dashint_{\carr_R} E_\eta$ is bounded independently of $\eta$ and $R$.  So the right-hand side of \eqref{unifconti} is a quadratic function of $c$, with fixed  quadratic coefficients and  linear and constant coefficients  which are bounded with respect to $R $ and $\eta$. A little bit of convex analysis implies that  $c \mapsto \mc{W}(E+c, m)$ being a limsup (over $R $ and $\eta$) of such functions is strictly convex, coercive and locally Lipschitz, hence it achieves its minimum for a unique $c$.
This means that the infimum defining $\W( \cdot , m)$ is a uniquely achieved minimum.

To conclude the proof, we just need to justify that $\dashint_{\carr_R} E_\eta$ is bounded independently of $\eta$ and $R$. We may write 
$$\dashint_{\carr_R} E_\eta= \dashint_{\carr_R} E_1+ \dashint_{\carr_R} (\nab \f_1 - \nab \f_\eta) * \mc{C} ,$$
where $\f_\eta$ is as in \eqref{def:truncation}.

Since, in the cases $\k = 0$, the functions $\nab \f_\eta$ and $\nab \f_1$ are integrable, we may check that 
\begin{equation*}
\int_{\carr_R}  (\nab \f_1 - \nab f_\eta) * \mc{C} \leq C \Numb_R(\C),
\end{equation*}
 where $C$ is independent of $R$ and $\eta$. But since $\mc{W}(E, m) <\infty$ and $E$ is $\Elec_m$, we have 
$$
\lim_{R \to \infty} \frac{1}{R^\d} \Numb_R(\C) = m,
$$ 
as follows e.g. from \cite[Lemma 2.1]{petrache2014next}. We deduce that
$$
\left|\dashint_{\carr_R} E_\eta\right|\le C(1+  \mc{W}(E, 1) + m)
$$
and by almost monotonicity of $\mc{W}$ (Lemma \ref{prodecr2}) the claim follows.
\end{proof}

\subsection{Proof of Lemma \ref{discerrtronc}}
\label{sec:preuvelowerboundtrunca}
\begin{proof}
From \cite[(2.29)]{petrache2014next} we know the following: let $\C$ be in $\config$, $m \geq 0$ and $E \in \Elec(\C, m)$ such that $\mc{W}(E,m)$ is finite. For any $\eta \in (0,1)$  we have 
\begin{equation} \label{prodecr2} \mc{W}_{\tau} (E,m)-\mc{W}_\eta (E,m)  \geq 
C \limsup_{R\to \infty} \frac{1}{R^{\d}} \sum_{p\neq q \in \mc{C} \cap \carr_R, |p-q|\le \eta} (\g(|p-q|+\tau)- \g(\eta))_+ -  m^2 o_{\eta}(1).
\end{equation}
where $C$ depends only on $\s$ and $\d$.

To prove \eqref{eqlemdisc} we take the expectation under $\bP$ (more precisely, under a tagged electric process compatible with $(\bP, \meseq)$) of \eqref{prodecr2}. Using stationarity, we obtain
\begin{multline*}
\bttW_{\tau} (\bP, \meseq)- \bttW_\eta(\bP, \meseq)
\\ \geq  C \Esp_{\bP} \left[\limsup_{R\to \infty} \sum_{p\neq q\in \mc{C}\cap \carr_R} \left( g(|p-q|+ \tau) - \g(\eta)\right)_+\right] - o_{\eta}(1)
\\ = C \limsup_{\tau\to 0} \Esp_{\bP}\left[ \sum_{p\neq q\in \mc{C}\cap \carr_1} \( \g(|p-q|+ \tau) - \g(\eta)\right)_+\right] - o_{\eta}(1).
\end{multline*}

In all cases \LogU \ , \LogD \ , \Riesz, there exists $C>0$ depending only on $\d, \s$ such that for any $0 < \tau<\eta^2/2$, 
\begin{multline*}
\sum_{p\neq q \in \mc{C}\cap \carr_1} \left( \g(|p-q| + \tau) - \g(\eta)\right)_+
\\
\ge C \g(2\tau)  \sum_{\vec {i} \in  \carr_1 \cap   \tau \Z^d}   ( \Numb_{\tau}(\vec{i})(\C)^2 -1)_{+}, + C \sum_{p\neq q \in \mc{C}\cap \carr_1, |p-q| \leq \eta^2/2} \left( \g(p-q) \right)_+, 
 \end{multline*}
where we denote by  $ \Numb_{\tau}(\vec{i})(\C)$ the number of points of the configuration $\mc{C}$ in the hypercube of center $\vec{i}$ and sidelength $\tau$,  with  $\vec{i}$ on the lattice  $\tau \Z^d$. On the other hand, stationarity implies
$$
\Esp_{\bP} \left[\sum_{\vec {i} \in  \carr_1 \cap   \tau \Z^d}   ( \Numb_{\tau}(\vec{i})(\C)^2 -1)_{+}\right] = \frac{1}{\tau^d} \Esp_{\bP} \left[ (\Numb_{\tau}^2 -1)_{+}\right],
$$
which concludes the proof of the lemma. 
\end{proof}

\subsection{Proof of Lemma \ref{lem:concordance}} \label{sec:preuveconcordance}
\begin{proof}
Let $P\in \probas_s(\config)$ be  such that $\Esp_{P} [ \W(\C, 1)]$ is finite. For $P$-a.e. $\C$, the energy $\W(\C,1)$ is thus finite, and we may find (according to Lemma \ref{infatteint}) an electric field $E$ compatible with $(\C, 1)$ and such that $\mc{W}(E, 1) = \W(\C,1)$. Let $\Pelec$ be the push-forward of $P$ by this map $\C \mapsto E$. By definition, $\Pelec$ is compatible with $P$ and its energy is that of $P$, but it may happen that $\Pelec$ is not stationary. In that case, we consider a \textit{stationarizing sequence}, namely a sequence of averages of translations of $\Pelec$ over large hypercubes. Each element of this sequence is still compatible with $(P,1)$ (because $P$ is stationary) and has the correct energy. Any limit point is compatible, has the correct energy, and moreover is stationary.  We thus have
$$
\ttW(\Pst,1) \geq \min \left\lbrace \Esp_{\Pelec}\left[\tW( \cdot, 1)\right] \ | \ \Pelec \textit{ stationary and compatible with } (P,1) \right\rbrace.
$$
The reverse inequality is obvious by definition of $\ttW_1$. 
\end{proof}

\subsection{Proof of Lemma \ref{LemmeDiscr}}
\label{sec:proofLemDiscr}
\begin{proof} 
If we admit \eqref{borneDiscr}, we easily get that $P$ has intensity $1$. Indeed, Jensen's inequality yields
\begin{equation*} 
\lim_{R \ti} \frac{\Esp_{\Pst}[\dis_R ]}{R^\d} = 0,
\end{equation*}
and on the other hand the stationarity assumption implies that for any $R > 0$
\begin{equation*}
\frac{\Esp_{\Pst}[\dis_R]}{R^{\d}} = \Esp_{\Pst}[\dis_1].
\end{equation*} 
Hence we get $\Esp_{\Pst}[\dis_R] = 0$ for any $R > 0$, which implies that $\Pst$ has intensity $1$. 

We now turn to proving \eqref{borneDiscr}.\\
\textit{Step 1.} \textbf{Preliminary bounds.} \\
 From Lemma~\ref{lem:concordance} we obtain a stationary electric process $\Pelec$ compatible with $(P,1)$ and such that $\tW(\Pelec,1) = \ttW(P,1)$.  We set $\eta_0 :=\frac{1}{4}$, and the monotonicity property \eqref{prodecr2} implies that
$$
\Esp_{\Pelec} \left[ \mc{W}_{\eta_0}(E,1) \right] \leq \tW(\Pelec,1) + C \leq \ttW(P, 1) + C,
$$
with a constant $C$ depending only on $\d,\s$. 

In the case $\k=1$, by stationarity and the definition of $\mc{W}$ we see that
\begin{equation*}
\Esp_{\Pelec} \left[\int_{\carr_1 \times \R^\k} \yg |E_{\eta_0}|^2\right] =\Esp_{\Pelec} \left[ \mc{W}_{\eta_0}(E, 1) \right] + \cds \g(\eta_0) \leq \tW(\Pelec, 1)  + C = \ttW(P,1) + C,
\end{equation*}
with a constant $C$ depending only on $\d,\s$. For any $R > 0$, by a mean value argument we may find $T \in (R, 2R)$ such that
\begin{equation} \label{choixTdiscr}
\Esp_{\Pelec} \left[ \int_{\carr_1 \times \{-T,T\}} \yg |E_{\eta_0}|^2\right] \leq  \frac{1}{R} \left(\ttW(P,1)  + C\right).
\end{equation}

\textit{Step 2.} \textbf{Expressing the discrepancy in terms of the field.} \\
Letting $\tcarr_R$ be the hyperrectangle $\carr_R \times[-T, T]^\k$ we have, integrating by parts
\begin{equation} \label{circulation}
\int_{\partial   \tcarr_R} \yg E_{\eta_0} \cdot \vec{\nu} = \int_{\tcarr_R} - \div(\yg E_{\eta_0}) = \c(\dis_R(\C) + r_{\eta_0}),
\end{equation}
where $\C$ is the point configuration with which $E$ is compatible, and the error term $r_{\eta_0}$ is bounded by  $n_{\eta_0}$, the number of points of $\C$ in a $\eta_0$-neighborhood of $\partial \carr_R$. We may roughly include this neighborhood in the disjoint union of $O(R^{\d-1})$ hypercubes of sidelength $1$. By stationarity we have 
\begin{equation} \label{controleneta0}
\Esp_{\Pst}[n_{\eta_0}^2] \leq CR^{2\d-2} \Esp_P \left[\Numb_1^2\right].
\end{equation}
Taking the expectation of \eqref{circulation} against $\Pelec$ and using elementary inequalities and \eqref{controleneta0} we get
\begin{equation} \label{discrnelle}
\Esp_{\Pst} \left[\dis_R^2\right] \leq C \Esp_{\Pelec} \left[ \int_{\partial  \tcarr_R} \yg |E_{\eta_0}|^2 \right] \left(\int_{\partial  \tcarr_R} \yg \right) + C R^{2\d-2} \Esp_P\left[\Numb_1^2\right].
\end{equation}

We can compute
\begin{align}
\label{estimygdiscr1} & \int_{\partial  \tcarr_R} \yg = CR^{\d-1} = CR^{\s+1} \quad \text{ for $\k = 0$},\\
\label{estimygdiscr2} & \int_{\partial  \tcarr_R} \yg \leq C R^{\d-1} \int_{0}^{T} \yg + CR^{\d} T^{\gamma}  \leq CR^{\s+1} \quad \text{ for $\k =1$}.
\end{align}

\textit{Step 3.} \textbf{Control on the boundary terms.} \\
We split $\partial  \tcarr_R$ as the disjoint union of
\begin{enumerate}
\item $2\d$ faces of the type $[-R/2, -R/2] \times \dots \{\pm R/2\} \times \dots \times [-R/2, R/2] \times [-T, T]^\k$, 
\item $0$ (if $\k=0$) or $2$ (if $\k =1$) faces of the type $\carr_R \times \{\pm T\}^\k$.
\end{enumerate}
For each of the $2\d$ faces of the first type we may write, using the stationarity of $\Pelec$,
\begin{multline} \label{usestatdiscr}
\Esp_{\Pelec} \left[ \int_{[-\frac{R}{2}, \frac{R}{2}] \times \cdots \{\pm \frac{R}{2}\} \times \cdots \times [-\frac{R}{2}, \frac{R}{2}] \times [-T, T]^\k} \yg |E_{\eta_0}|^2 \right] \\ = \frac{1}{R} \Esp_{\Pelec} \left[ \int_{\carr_R \times [-T,T]^\k} \yg |E_{\eta_0}|^2 \right] \\
\leq \frac{1}{R}  \Esp_{\Pelec} \left[  \int_{\carr_R \times \R^\k} \yg |E_{\eta_0}|^2 \right] \leq CR^{\d-1} \left(\ttW(\Pst, 1) + C\right),
\end{multline}
whereas for the second type of faces we have, using \eqref{choixTdiscr} and the stationarity of $\Pelec$
\begin{multline} \label{uppfacediscr}
\Esp_{\Pelec} \left[ \int_{\carr_R \times \{-T,T\}^\k} \yg |E_{\eta_0}|^2 \right] = R^{\d} \Esp_{\Pelec} \left[ \int_{\carr_1 \times \{-T, T\}^\k} \yg |E_{\eta_0}|^2\right] \\ \leq CR^{\d-1}(\ttW(\Pst, 1) + C).
\end{multline}
Inserting  \eqref{estimygdiscr1} (if $\k=0$) or \eqref{estimygdiscr2} (if $\k=1$), \eqref{usestatdiscr} and \eqref{uppfacediscr} (if $\k=1$) into \eqref{discrnelle} we obtain
\begin{equation} \label{discrnelle2}
\Esp_{\Pst} \left[\dis_R^2\right] \leq C \left(\ttW(P, 1) + C\right) R^{\d-1+\s+1} + C R^{2\d-2} \Esp\left[\Numb_1^2\right]. 
\end{equation}

\textit{Step 4.} \textbf{Conclusion.} \\
We have thus proven that
\begin{equation}
\Esp_{\Pst} \left[\dis_R^2\right] \leq C \left(\ttW(P, 1) + C\right) R^{\d+\s} + C R^{2\d -2} + C R^{2\d-2} \Esp\left[\dis_1^2\right],
\end{equation}
it remains to show that $\Esp\left[\dis_1^2\right]$ is itself bounded by $C \left(\ttW(P,1) + C\right)$, which would conclude the proof. For $L > 2$ let us define an auxiliary function
$$
\Psi_L(x) := \frac{x^2}{L^\s} \min\left(1, \frac{|x|}{L^\d} \right).
$$ 
We may re-write \cite[Lemma 2.2]{petrache2014next} as 
\begin{equation} \label{auxiliarydiscr}
\Psi_L(\dis_L) \leq C \int_{\carr_{2L} \times [-L,L]^{\k}} \yg |E_{\eta_0}|^2.
\end{equation}
Taking the expectation we get as above
\begin{equation*}
\Esp_{P} \left[ \Psi_L(\dis_L) \right]  \leq CL^\d (\ttW(P, 1) + C).
\end{equation*}
Distinguish between the events $|\dis_L| \leq L^{\d}$ and $|\dis_L| > L^{\d}$, we obtain
\begin{equation} \label{discrplusfaible}
\Esp_{P} \left[ \dis^2_L \right]  \leq \min \left( L^{2\d}, C L^{ \d+ \s} (\ttW(P, 1) + C) \right).
\end{equation}
This is enough for our purpose. We could somehow improve \eqref{discrplusfaible} but we would still get a weaker bound than \eqref{borneDiscr}.
\end{proof}

We now give a proof of the result stated in Remark \ref{rem:Discr1d}.
\begin{proof}
We follow the same line as in the previous proof. We replace \eqref{choixTdiscr} by the following observation: since $\int_{\carr_1 \times \R} \yg |E_{\eta_0}|^2$ is finite we have
\[
\liminf_{T \ti} (T \log T) \int_{\carr_1 \times \{-T, T\}} \yg |E_{\eta_0}|^2 = 0.
\]
In particular we might find an increasing sequence $\{T_k\}_{k}$ with $\lim_{k \ti} T_k = + \infty$ and
\[
\lim_{k \ti} (T_k \log T_k) \int_{\carr_1 \times \{-T_k, T_k\}} \yg |E_{\eta_0}|^2  = 0.
\]
Setting $R_k = T_k \sqrt{\log T_k}$ and keeping the same notation as in the previous proof we see that for $k$ large enough we have
\[
\Esp_{P} [\dis_{R_k}^2] \leq C T_k\left(1 + \frac{R_k}{T_k \log T_k}\right) = o(R_k),
\]
which proves \eqref{discrcas1d}.
\end{proof}

\subsection{Proof of Lemma \ref{constructreseau}}
\label{sec:constructionreseau}
\begin{proof}
Let $m= \dashint_K \mu$ be the average of $\mu $ over $K$. 
We may (see e.g.  \cite[Lemma 6.3]{petrache2014next}) partition $K$ into  $N_K$ hyperrectangles $\mathcal R_i$, which all have volume $1/m$, and whose sidelengths are in $[2^{-\d} m^{-1/{\d}}, 2^\d m^{1/{\d}}]$.
In each of these hyperrectangles we solve 
\begin{equation}\left\lbrace\begin{array}{ll}
-\div(\yg \nab h_i) = \c \left( \delta_{X_i} - m \delta_{\Rd} \right)  & \text{ in } \mathcal R_i  \times [-1,1]^\k\\
\nab h_i \cdot \vec{\nu} =0 & \text{on} \ \partial  (\mathcal R_i\times [-1,1]^\k)\end{array}\right.
\end{equation} 
According to \cite[Lemma 6.5]{petrache2014next}, if $X_i\subset \R^{\d} \times\{0\} $ is at distance $\le 2^{-(\d+1)} m^{-1/{\d}}$ from the center $p_i$ of $\mathcal R_i$ then we have 
$$\lim_{\eta\to 0}\left|\int_{\mathcal R_i \times [-1,1]^\k} \yg |\nab (h_i)_\eta|^2 - \c \g(\eta)\right|\le C$$ where $C$ depends only on $\d$ and $m$.
We may then define $E_i= \nab h_i\indic_{\mathcal R_i \times [-1,1]^\k}$, and by compatibility of the normal components, the vector field $\Egen=\sum_i E_i$ satisfies 
   \begin{equation}\left\lbrace\begin{array}{ll}
-\div(\yg \Egen) = \c \left( \sum_i \delta_{X_i} - m \delta_{\Rd} \right)  & \text{ in } K \times \R^\k\\
\Egen \cdot \vec{\nu} =0 & \text{on}\  \partial  (K \times \R^\k)\end{array}\right.
\end{equation} 
and if $\eta<\eta_0 <2^{-(\d+2)} m^{-1/{\d}}$, 
\begin{equation}\label{estbrut}
\int_{K\times \R^\k} \yg|\Egene|^2 - \c N_K  \g(\eta) \le C N_K\end{equation} with 
$C$ depending only on $\d$ and $m$.
The last step is to rectify for the error made by replacing $\mu$ by $m$. For that, we use the following result from \cite[Lemma 6.4]{petrache2014next}.
 \begin{lem} \label{constvsmoy}
Let $K_R$ be a hyperrectangle whose sidelengths are in $[R, 2R]$, and $\mu$ a bounded measurable function such that $\int_{K_R} \mu$ is an integer, and let $m= \dashint_{K_R} \mu$. The solution (unique up to constant) to 
\begin{equation}\label{espygh}
\left\lbrace\begin{array}{ll}
-\div (\yg \nab h) = \c (\mu -m)  \drd & \quad \text{ in } K_R\times [-R,R]^\k\\
\nab h \cdot\vec{ \nu }  = 0 &  \text{ on } \partial  (K_R\times [-R,R]^\k),\end{array}\right.
\end{equation}
 exists and satisfies 
\begin{equation}\label{esterrect}
\int_{K_R \times [-R,R]^\k} \yg |\nab h|^2 \leq C R^{\d+1-\gamma}\|\mu-m\|_{L^\infty(K_R)}^2.
\end{equation}
\end{lem}
%5\begin{proof}
%First, we can consider only the case $\k=1$, otherwise the result is standard.
%Second, we have the following trace inequality (proven in \cite[Lemma 6.4]{petrache2014next}
%\begin{equation}\label{trace0}
%\int_{K_R\times [-R,R]^\k}  |h(x, 0)|^2 \le C R^{1-\gamma} \int_{K_R} \yg |\nab h|^2.\end{equation}
%Multiplying \eqref{constvsmoy} by $h$ and integrating by parts, we obtain 
%$$\int_{K_R\times [-R,R]^\k} \yg|\nab h|^2 = \c \int_{K_R\times \{0\}} (\mu-m) h.$$
%Using the Cauchy-Schwarz inequality and \eqref{trace0}, the result immediately follows. 
%\end{proof}

Applying this lemma provides a function $h$, 
and we  let $$\hat{E}= E+ \nab h \indic_{K\times [-R,R]^k}.$$
It is obvious that $\hat E$  solves 
\begin{equation}\left\lbrace\begin{array}{ll}
-\div(\yg \hat E) = \c \left( \sum_i \delta_{X_i} -\mu \delta_{\Rd} \right)  & \text{ in } K \times \R^\k\\
\hat E \cdot \vec{\nu} =0 & \text{on}\  \partial  (K \times \R^\k). \end{array}\right.
\end{equation} 
Combining \eqref{esterrect} and \eqref{estbrut} and using the Cauchy-Schwarz inequality, we obtain 
\begin{multline*}
\int_{K\times \R^\k} \yg| \hat E_\eta|^2 \le \c N_K  (\g(\eta) +C) +C R^{\d+1-\gamma}\|\mu-m\|_{L^\infty(K_R)}^2\\
+ C \(N_K \g(\eta)\)^\hal R^{\frac{\d+1-\gamma}{2}} \|\mu - m\|_{L^\infty(K_R)}.
\end{multline*}

Letting then $\mathcal {R}(K, \mu)$ be the family of configurations $\{X_i\}_{i=1}^{N_K}$ above where each $X_i$ varies in $B(p_i, 2^{-(\d+1)} m^{-1/{\d}})$ ($p_i$ being the center of $\mathcal R_i$), and with all possible permutations of the labels, we have thus obtained that for every $\mc{C}\in \mathcal  R(K,\mu)$ the desired results hold.
\end{proof}

\subsection{Proof of Lemma \ref{lem:pointsexterieurs}}
\label{preuve:finirconstruction}
\begin{proof}
\def \Bound{\mathrm{Bound}_{\mu}}
Let us return to the notation of Section \ref{sec611}. Since $\Sigma' \backslash \Sigma_T'$ is formed of disjoint sets of the form $\{x\in \Sigma', \dist (x,  \Gammapj)\le T_j\}$, we may work component by component, and from now on 
drop the indices $j$.
Let us set 
$$
\Bound(t) :=\int_{\{x\in \Sigma', \dist(x, \Gamma) \leq t\}} \meseq'(x)\, dx.
$$
In view of \eqref{reass5} and \eqref{longueur} we have 
\begin{align}
\label{717}
 & \Bound(t) \ge \cc \int_0^t \mathcal{H}^{\ll} (\Gamma) s^{\d-\ll-1}\left(\frac{s}{N^{1/\d}}\right)^\alpha ds\ge  \cc N^{\frac{\ll-\alpha}{\d}}  t^{\alpha+ \d - \ll}, \\
\label{717b}
& \Bound(t) \leq \cC \int_0^t \mathcal{H}^{\ll} (\Gamma) s^{\d-\ll-1}\left(\frac{s}{N^{1/\d}}\right)^\alpha ds \leq \cC N^{\frac{\ll-\alpha}{\d}}  t^{\alpha+ \d - \ll}.
\end{align}

The goal is to build layers that are thick enough to contain a large number of points, and so that the points can be placed in cells of bounded aspect ratios. 
To do so, we define inductively a sequence of $t$'s terminating at $T$, by letting $t_0=0$, and $t_1$ be the smallest such that 
\begin{equation}
\label{def:t1}
\Bound(t_1)\in \mathbb{N}, \quad t_1\ge N^{\frac{\alpha}{\d(\alpha+\d)}}.
\end{equation}
In view of \eqref{717} we have
$$
t_1 \le C N^{\frac{\alpha}{\d(\alpha+\d)}},
$$
for some $C$ depending on the constant $c_1$. We let $n_1 := \Bound(t_1)$ and from \eqref{717}, \eqref{717b} we observe that
$$
\cc N^{\frac{\ell}{\alpha+\d}}\le n_1 \le \cC N^{\frac{\ell}{\alpha+\d}}.
$$

Next, let us assume that $t_{j-1}$ has been constructed for some $j \geq 2$. We define $t_j$ as the smallest positive number such that 
\begin{equation}
\label{def:tj}
\Bound(t_j)-\Bound(t_{j-1})\in \mathbb{N}, \quad  t_j - t_{j-1} \geq  N^{\frac{\alpha}{\d^2}} t_{j-1}^{-\alpha/\d}.
\end{equation}
Elementary computations of exponents yield that for every $j \ge 2$, 
\begin{align*}
& t_j-t_{j-1} \leq  \cC N^{\alpha/\d^2 } t_{j-1}^{-\alpha/\d} \\
& t_j-t_{j-1} \leq  \cC N^{\frac{\alpha}{\d(\alpha+\d)}}.
\end{align*}
We set $n_j := \Bound(t_j)-\Bound(t_{j-1})$, and using \eqref{717}, \eqref{717b} we get 
\begin{equation*}
\cC  (t_j-t_{j-1}) t_{j-1}^{\alpha+\d-\ll-1}  N^{\frac{\ll-\alpha}{\d}} \geq n_j \ge \cc (t_j-t_{j-1}) t_{j-1}^{\alpha+\d-\ll-1}  N^{\frac{\ll-\alpha}{\d}}.
\end{equation*}

The construction terminates at $j=J$ such that $t_J=T$.
%$n_j$ is comparable to $(t_j-t_{j-1})t_{j-1}^{\alpha +\d-l-1} N^{\frac{l-\alpha}{\d}}$, itself comparable to  
%$ \frac{\mathcal{H}^{\d-1}(\Gamma_{t_{j-1}})}{(t_j-t_{j-1})^{\d-1}} $ (using the fact that $\mathcal{H}^{\d-1}(\Gamma_{t_{j-1}} ) $ is comparable to $N^{l/\d} t_{j-1}^{\d-1-l}).$
We may then partition each slice $\Gamma_{t_{j-1}}  \backslash \Gamma_{t_j}$ into $n_j$ regions $\mathcal{R}$ of “size” comparable to $t_j-t_{j-1}$ (which means that they contain a ball of radius $\cc (t_j-t_{j-1})$ and can be included in a ball of radius $\cC(t_j-t_{j-1})$) on which $\meseq'$ has mass $1$ and on which we place $1$ point each. We can construct a screened electric field for each region, as described in \cite{petrache2014next}.

The energy cost of the union of all these cells, for $j$ ranging from $1$ to $J$, is evaluated as in \cite{petrache2014next}, except for the part corresponding to the correction between $\meseq'$ and its average on the cell, which depends on the regularity of $\meseq$ and the size of cell: since Assumption \ref{H5} is more general than the one in \cite{petrache2014next}, this requires some changes.

We show that the sum of these contributions over all the regions $\mathcal{R}$ is negligible compared to $N$ (like all the other energy contributions of this boundary layer). Using Lemma \ref{constvsmoy} (or its analogue over more general shapes than hyperrectangles), we get for each cell $\mathcal{R}$ a term of order
$$
(t_j-t_{j-1})^{\d+1-\gamma} \left\|\meseq'- \( \frac{1}{|\mathcal{R}|} \int_{\mathcal{R}}  \meseq' \right)\right\|^2_{L^\infty}. 
$$
On the one hand, the H\"older control of $\muv$ as in \eqref{reass5b} yields
$$
|\mu_V'- \frac{1}{|\mathcal{R}|} \int_{\mathcal{R}}  \meseq'|\le |t_j-t_{j-1}|^{\min(\alpha,1)} N^{-\min(\alpha,1)/\d},
$$
and since we have seen above that  $t_j-t_{j-1}$ is bounded by $\cC N^{\frac{\alpha}{\d(\alpha+\d)}}$,  we find  that the contribution of each cell $\mathcal{R}$ is bounded by 
\begin{equation}  \label{premiereborne}
N^{\frac{\alpha}{\d(\alpha+\d)}(2\alpha' +\d+1-\gamma)  } N^{-2\alpha'/\d} \le N^{ \frac{\alpha}{\d(\alpha+\d)} ( 2\alpha'+\d+1-\gamma  -2\alpha'  - 2 \frac{\alpha'}{\alpha} \d)   },
\end{equation}
where $\alpha' :=\min(\alpha, 1)$.

On the other hand, the number of such cells is bounded by $\cC \um^{1/\alpha} N$, because they all contribute to a mass at least $1$ of $\meseq'$ and the total mass of this boundary layer is bounded by $\cC \um^{1/\alpha} N$. Therefore, if 
$\d+1-\gamma - 2\d \alpha'/\alpha \le 0$, these contributions will be negligible compared to $N$, as desired. 
In particular, if $\alpha\le 1$, since we always have $\d -1+\gamma=\s\ge 0$, we are done. 

If $\alpha \ge 1$, we can use \eqref{reass5} instead. We then bound the contributions of the cells by 
$$
\cC \sum_{j=1}^J n_j (t_j-t_{j-1})^{\d+3-\gamma} t_{j}^{2(\alpha-1)} N^{-2\alpha/\d}.
$$
Using that $n_j$ is comparable to $ (t_j-t_{j-1})t_{j-1}^{\alpha +\d-\ll-1} N^{\frac{\ll-\alpha}{\d}}$, we bound the previous expression by
$$
\sum_{j=1}^J  (t_j-t_{j-1})^{\d+4-\gamma}     t_{j-1}^{3\alpha+\d-\ll-3   }N^{\frac{\ll-3\alpha}{\d}},
$$
and then, since $t_j-t_{j-1}$ is bounded by $N^{\frac{\alpha}{\d^2} }  t_{j-1}^{-\alpha/\d}$, we bound again the sum by
$$
\sum_{j=1}^J (t_j-t_{j-1})    N^{\frac{\alpha}{\d^2} (\d+3-\gamma) +\frac{\ll-3\alpha}{\d}}  t_{j-1}^{-\alpha (\d+3-\gamma)/\d +3\alpha+\d-\ll-3 }. 
$$ 
This sum is comparable to the integral 
$$
N^{\frac{\alpha}{\d^2} (\d+3-\gamma) +\frac{\ll-3\alpha}{\d}} \int_{t_1}^{T} t^p \, dt, \qquad p= -\alpha(\d+3-\gamma)/\d +3\alpha+\d-\ll-3.$$
If $p > -1$, then we may bound it (using \eqref{choixdeT}) by
$$N^{\frac{\alpha}{\d^2} (\d+3-\gamma) +\frac{\ll-3\alpha}{\d}} T^{p+1} \le \um^{p+1/\alpha} N^{\frac{\alpha}{\d^2} (\d+3-\gamma) +\frac{\ll-3\alpha+p+1 }{\d}}.$$
After some elementary algebra, the  exponent of $N$ turns out to be $1 -\frac{2}{\d} < 1$, hence the contribution in this case is negligible with respect to $N$ any time $p > - 1$.

If $p\le -1$, then the integral is dominated by $t_1^{p+1}$, which in view of $t_1 \geq   N^{\frac{\alpha}{\d(\alpha + \d)}}$ yields a total contribution bounded by 
$$
N^{\frac{\alpha}{\d^2} (\d+3-\gamma) +\frac{\ll-3\alpha}{\d} + \frac{\alpha}{\d (\alpha+\d) } \left( -\alpha (\d+3-\gamma) /\d+ 3\alpha +\d-\ll -2 \right) }.
$$ 
The exponent is 
$$
\frac{\alpha + \d \ll - \gamma \alpha - \alpha \d}{\d (\alpha + \d)}.
$$
Using the fact that $\ll \leq \d -1$ and the expression of $\gamma$ as in \eqref{gs}, we find that the exponent is always $<1$, so the whole contribution is $o(N)$ as needed in that case as well. 
\end{proof}

\bibliographystyle{alpha}
\bibliography{ldpbiblio}

\end{document}